\documentclass[12pt,oneside]{article}
\usepackage[utf8]{inputenc}

\usepackage{amsthm, amsfonts, amssymb, amsmath,enumitem, natbib, bbm, mathabx}
\usepackage{cases, amsthm}
\usepackage{fullpage}
\usepackage{wasysym} 
\usepackage{fancyhdr}
\usepackage{graphicx}
\usepackage{bbm}
\usepackage{caption}
\usepackage{subcaption}
\usepackage{authblk}
 \usepackage{setspace}
\usepackage[colorlinks=true, pdfstartvie w=FitV,
linkcolor=blue, citecolor=blue, urlcolor=blue]{hyperref}
\usepackage{comment}
\usepackage{mathtools}
\DeclarePairedDelimiter\ceil{\lceil}{\rceil}
\DeclarePairedDelimiter\floor{\lfloor}{\rfloor}

\newtheorem{theorem}{Theorem}
\newtheorem{prop}{Proposition}
\newtheorem{lemma}{Lemma}

\newtheorem{remark}{Remark}

\usepackage[linesnumbered,ruled,vlined]{algorithm2e}

\usepackage{mathrsfs}

\usepackage{xcolor}
\usepackage{cleveref}

\DeclareMathOperator*{\argmin}{arg\,min}

\usepackage[top=1in, bottom=1in, left=1in, right=1in]{geometry}
\title{Inference for Median
 and a Generalization of HulC}

\author[1]{Manit Paul}
\author[2]{Arun Kumar Kuchibhotla}
\affil[1]{Department of Statistics \& Data Science, University of Pennsylvania}
\affil[2]{Department of Statistics \& Data Science, Carnegie Mellon University}
\date{}

\begin{document}

\maketitle

\begin{abstract}
It is well-documented in the literature that sample splitting offers significant methodological and theoretical advantages in statistical inference. This, for example, includes cross-fitting in double machine learning, universal inference for parametric inference, and split conformal prediction. The recently proposed inference method, HulC, also falls into this category. HulC operates by viewing the target functional of interest as an approximate median of an estimator and applying the classical distribution-free confidence intervals for median with minimal sample size. When the estimators are asymptotically normal, HulC intervals are shown to be 50\% wider than the Wald intervals asymptotically, on average, for $95\%$ coverage. Interestingly, this ratio of widths converges to a non-degenerate distribution. In this paper, we propose a generalization of HulC that are only 25\% wider than the Wald intervals for asymptotically normal estimators, irrespective of the nominal coverage. Furthermore, similarly to HulC, these generalized intervals remain valid for a significantly wider class of problems with non-normal limiting distributions. To better understand width properties under non-normal limiting distributions, we analyze distribution-free confidence intervals for the median when the Lebesgue density at the median is either zero or infinite. Surprisingly, we find that properly scaled, the interval width converges to a non-degenerate random variable.
\end{abstract}

\section{Introduction}
\label{sec:introduction}
The interval estimation of the median of a distribution based on independent and identically distributed (i.i.d.) observations is a well-studied problem in Statistics. Suppose $X, X_1,\cdots,X_n$ are i.i.d.\ univariate random variables from a distribution $F$. Any $\theta_0\in\mathbb{R}$ is said to be a median of $F$ if
\begin{equation}\label{eq:median-defining-property}
\mathbb{P}(X \le \theta_0) \ge \frac{1}{2}\quad\mbox{and}\quad \mathbb{P}(X \ge \theta_0) \ge \frac{1}{2}.
\end{equation}
It is clear that median of a distribution is not necessarily unique without further assumptions on the distribution $F$. If $\theta_0$ satisfying~\eqref{eq:median-defining-property} is a continuity point of $F$, then the median is unique and is given by $\theta_0$. Constructing valid confidence intervals for any median $\theta_0$ is a problem interesting in itself but has wide-range implications for confidence intervals of arbitrary statistical functionals. 

\cite{kuchibhotla2023median} proved that non-trivial asymptotically valid confidence intervals can be constructed for a statistical functional $\theta(P)\in\mathbb{R}$ based on i.i.d. observations $W_1, \ldots, W_N\sim P$ (in a measurable space) if and only if there exists a non-trivial estimator sequence $\{\widehat{\theta}_m\}_{m\ge1}$ that satisfies
\[
\limsup_{m\to\infty}\,\left(\frac{1}{2} - \min\left\{\mathbb{P}(\widehat{\theta}_m \ge \theta(P)),\,\mathbb{P}(\widehat{\theta}_m \le \theta(P))\right\}\right)_+ = 0.
\]
This condition means that both $\mathbb{P}(\widehat{\theta}_m \le \theta(P))$ and $\mathbb{P}(\widehat{\theta}_m \ge \theta(P))$ are asymptotically above $1/2$, or equivalently, $\theta(P)$ is asymptotically the median of $\widehat{\theta}_n$. One can use the methods developed for median for inference for $\theta(P)$ given an asymptotically median unbiased estimator $\widehat{\theta}_n$ as follows: (1) randomly split the data of $N$ observations into $B = B_N$ non-overlapping blocks; (2) compute versions of $\widehat{\theta}_N$ on each block, naming them $\widehat{\theta}^{(1)}, \ldots, \widehat{\theta}^{(B)}$; (3) now treat $\widehat{\theta}^{(j)}, 1\le j\le B_N$ as i.i.d. observations for a distribution with median $\theta(P)$, and apply the inference methods developed for median. This is, in essence, the general-purpose inference method HulC developed in~\cite{kuchibhotla2021hulc} but with $B = \lceil\log_2(2/\alpha)\rceil$. 
Although HulC shares some features with subsampling or resampling methods well-established in the literature, its validity guarantees remain superior to these existing methods. In particular, there exist several examples where subsampling/bootstrap is provably invalid (uniformly) but HulC remains uniformly valid. On the flip side, when considering the width of the confidence interval under regularity assumptions (i.e., asymptotic normality of estimators), HulC is provably asymptotically wider than the Wald/boostrap/resampling methods and additionally,
\[
\frac{\mbox{Width of HulC}}{\mbox{Width of Wald}} \overset{d}{\to} \max_{1\le j\le B = \lceil\log_2(2/\alpha)\rceil}|Z_j|,
\]
where $Z_1, \ldots, Z_B$ are IID standard normal random variables. This implies that the ratio of widths converges to a non-degenerate distribution. For bootstrap or subsampling, this ratio is asymptotically equal one. 

In this paper, we propose a generalization of HulC called GHulC by allowing $B$ to be larger than $\lceil\log_2(2/\alpha)\rceil$ and even diverge with sample size $N$. Under the standard assumptions (i.e., asymptotic normality), letting $B = B_N$ diverge arbitrary slowly, we prove that
\begin{equation}\label{eq:GHulC-width-ratio}
\frac{\mbox{Width of GHulC}}{\mbox{Width of Wald}} \overset{d}{\to} \sqrt{\frac{\pi}{2}} \approx 1.25.
\end{equation}
(This is, in fact, convergence in probability but to keep the comparison with the result for HulC, we use the notation for convergence in distribution.)
Limit law~\eqref{eq:GHulC-width-ratio} shows that GHulC confidence intervals are approximately 25\% wider than the Wald, while the HulC confidence intervals are approximately 50\% wider for $95\%$ coverage. Readers familiar with Pitman efficiency might recognize $\sqrt{\pi/2}$ as the Pitman efficiency of the sample median against the sample mean for normal mean estimation problem. From this point of view, the limit law~\eqref{eq:GHulC-width-ratio} is expected because GHulC treats the target as the median while Wald treats it as the mean of the estimator. It must be recalled here that for non-normal limiting distributions (as illustrated in~\cite{kuchibhotla2021hulc} and~\cite{mallick2023new}), it is easier to view the statistical functional of interest as the median rather than the mean, and this is precisely the aspect that makes HulC/GHulC more broadly applicable.

Going beyond the standard setting of asymptotic normality, one might ask if the ratio as in~\eqref{eq:GHulC-width-ratio} is meaningful when the Wald interval may be inaccessible. For example, when the limiting distribution of the estimator depends on nuisance components that are unknown or difficult/impossible to estimate, the (oracle) Wald interval might be too optimistic. As an illustrative example, we discuss extensively the width properties of non-parametric distribution-free confidence intervals for median of a distribution beyond the standard assumptions. Our surprising finding is that under the standard setting, the scaled width converges to one, but in the non-standard setting, the scaled width converges to a non-degenerate distribution. Note that the oracle Wald interval always has a degenerate scaled width, and additionally, bootstrap/resampling intervals are provably invalid beyond the standard setting. Given the description of HulC/GHulC, this study has strong implications for the width properties of GHulC. To the best of our knowledge, this is the first study of width of confidence intervals for median in non-standard settings. 

The organization of the remaining paper and our contributions are as follows. We first focus on the problem of confidence intervals for the median of a univariate distribution and then consider the generalization of HulC. The contributions of this paper are as follows:
\begin{enumerate}
\item We rederive a distribution-free finite sample confidence interval for any median of a distribution and provide explicit conditions on $n, \alpha$ such that the confidence interval can be expressed in terms of the order statistics. (Note that any bounded confidence interval is bound to not cover the population median if the sample size is too small.)
\item We derive precise (asymptotically sharp) and easily computable upper and lower bounds for quantiles of the Binomial$(n, 1/2)$ distribution. These precise bounds yield quick computation of the proposed confidence interval and also allow us to characterize in finite samples the closeness between the width of the nonparametric interval and that of the Wald interval (under standard conditions on $F$).
\item Under non-standard conditions on $F$, allowing for zero density at $\theta_0$ and/or jump discontinuity at $\theta_0$, we derive the rate of convergence as well as a precise characterization of the limiting distribution of scaled width. The fact that the scaled width converges to a non-degenerate distribution is the most surprising finding of our work.
\item In addition to the asymptotic limits, we provide a finite sample concentration inequality for the scaled width of the confidence interval under non-standard conditions.
\item Finally, we consider the implications of the results for the median to inference for general parameters/functionals extending the HulC approach. 
\end{enumerate}

\paragraph{Organization.} In Section~\ref{sec:methods}, we introduce the non-parametric distribution-free finite sample confidence intervals for the population median without any assumptions on $F$; this is the Tukey-Scheffe confidence interval~\citep{scheffe1945non}. In Section~\ref{sec:width-analysis}, we show that under the standard conditions required for asymptotic normality of the sample median (or the assumptions under which Bahadur representation exists), the confidence interval that we discuss in this paper performs asymptotically as well as the Wald confidence interval i.e.,\ the width ratio converges to one asymptotically. In addition to this, we also analyze the width of this confidence interval under several non-standard cases discussed in~\cite{knight1998bootstrapping} and~\cite{ghosh1981bahadur}. In Section~\ref{sec:hulc_extension}, we present generalized HulC methodology and compare the performance with HulC of~\cite{kuchibhotla2021hulc}. We present a simulation study in Section~\ref{subsec:sim_ghulc} to corroborate our theoretical results. Finally, we conclude the article with a summary and a discussion of some future directions in Section~\ref{sec:extension}.

The proofs of all the main results along with auxiliary results are presented in the appendix. In Section~\ref{appsec:related-literature}, we review the rich literature on non-parametric distribution-free confidence intervals for the median, and also review the validity of bootstrap/resampling methods. Throughout the paper, we reserve the notation $\alpha\in(0, 1)$ to denote the required miscoverage for the confidence intervals. We use $\mbox{Bin}(n, p)$ to denote the binomial distribution with index $n$ and success probability $p\in(0, 1)$, and use $[n]$ to denote $\{1, 2, \ldots, n\}$.

\section{Confidence Intervals for Median}
\label{sec:methods}     
Suppose $X_1, \ldots, X_n$ are independent random variables all with median $\theta_0$, i.e., $\mathbb{P}(X_i \le \theta_0) \ge 1/2$ and $\mathbb{P}(X_i \ge \theta_0) \ge 1/2$ for $i\in[n]$. The non-parametric confidence intervals for the median are derived from the following stochastic dominance (see $(1.1)$ of \cite{klenke2010stochastic}),
 \begin{equation}
 \label{eq:stoch_dominance}
      \mathbb{P}\left( \sum_{i=1}^n \mathbf{1}\{X_i \le \theta_0\} \geq k \right) \geq  \mathbb{P}\left(M_n \geq k \right)\mbox{  }\mbox{and  }
         \mathbb{P}\left(\sum_{i=1}^n \mathbf{1}\{X_i \ge \theta_0\} \geq k \right) \geq  \mathbb{P}\left(M_n \geq k \right),
 \end{equation}
where $M_n \sim \text{Bin}(n,1/2)$ and $k \in \{0,1,\cdots,n\}$. \Cref{alg:proposed-conf-int} provides the construction of the confidence interval for $\theta_0$. Here, we do not require $X_i$'s to be identically distributed.
\begin{algorithm}
    \caption{A distribution-free finite sample confidence interval for the population median}
    \label{alg:proposed-conf-int}
    \KwIn{Sample: $X_1,\cdots,X_n$ and Confidence Level: $1-\alpha$}
    \KwOut{Distribution-free confidence interval for the population median with finite sample coverage}
    Define $c_{n,\alpha}$ to be the smallest integer $x$ such that $\mathbb{P}(Y_n \geq \floor{n/2}-x) \ge 1 - \alpha/2$, where $Y_n\sim\text{Bin}(n,1/2)$, i.e.,
    \[
    c_{n,\alpha} := \inf\{x:\,\mathbb{P}( Y_n \geq \floor{n/2}-x) \ge 1 - \alpha/2\}.
    \]\\
    Compute the sets
    \begin{align*}
      \widehat{\mathrm{CI}}_{1,n,\alpha} &:= \left\{\theta\in\mathbb{R}:\, \sum_{i=1}^n \mathbf{1}\{X_i \le \theta\} \ge \left\lfloor \frac{n}{2}\right\rfloor - c_{n,\alpha}\right\}, \\
    \widehat{\mathrm{CI}}_{2,n,\alpha} &:= \left\{\theta\in\mathbb{R}:\, \sum_{i=1}^n \mathbf{1}\{X_i \ge \theta\} \ge \left\lfloor \frac{n}{2}\right\rfloor - c_{n,\alpha}\right\}.     
    \end{align*}
    \\
    Return the confidence interval $\widehat{\mathrm{CI}}_{n,\alpha}=\widehat{\mathrm{CI}}_{1,n,\alpha} \cap \widehat{\mathrm{CI}}_{2,n,\alpha}$; see~\eqref{eq:rewritten-CI}.
\end{algorithm}

We shall show that the confidence interval in \Cref{alg:proposed-conf-int} ensures finite sample coverage for all sample sizes and for all distributions of $X_1, \ldots, X_n$ having a common median. Moreover, if the sample size is greater than a certain threshold (depending on $\alpha$), the confidence interval can be represented in terms of order statistics. This is the content of the following theorem (proved in \Cref{appendix:theorem_asym_cn}). 
\begin{theorem}
\label{thm:assym_of_cn}
Suppose $X_1, \ldots, X_n$ are independent random variables all with median $\theta_0$, i.e., $\mathbb{P}(X_i \le \theta_0) \ge 1/2$ and $\mathbb{P}(X_i \ge \theta_0) \ge 1/2$ for all $i \in [n]$. Then the confidence interval returned by~\Cref{alg:proposed-conf-int} satisfies the following:
\begin{enumerate}
\item For all $n\ge1$ and for any $\alpha\in(0, 1)$, $\mathbb{P}(\theta_0\in\widehat{\mathrm{CI}}_{n,\alpha}) \ge 1 - \alpha$.
\label{conclu:coverage-guarantee}
\item For any $\alpha\in(0, 1)$,
\begin{equation}\label{eq:rewritten-CI}
\widehat{\mathrm{CI}}_{n,\alpha} = 
\begin{cases}
\left[X_{(\floor{n/2}-c_{n,\alpha})},\ X_{(\ceil{n/2}+c_{n,\alpha}+1)}\right], &\mathrm{if } \quad n \ge \log_2(2/\alpha),\\
\mathbb{R}, &\mathrm{if } \quad n < \log_2(2/\alpha).
\end{cases}
\end{equation}
\label{conclu:form-of-conf-interval}
\item For any $\alpha\in(0, 1)$, the $c_{n,\alpha}$ defined in Step 1 of~\Cref{alg:proposed-conf-int} satisfies,
\[
c_{n,\alpha}-\frac{\sqrt{n}z_{\alpha/2}}{2} ~\in~ \begin{cases}
    \left[-\frac{1}{2\sqrt{n}}\max\left\{\frac{z_{\alpha/2}^3}{5},\,\sqrt{1 + \frac{2\ln(2) - 1}{n^2}}\right\} - 1.5 ,\,\, 1\right], &\mathrm{if } \quad n \ge \log_2(2/\alpha),\\
    \{\floor{n/2}-(\sqrt{n}z_{\alpha/2})/2\}, &\mathrm{if } \quad n < \log_2(2/\alpha).
\end{cases}
\]
\label{conclu:c_{n,alpha}_bounds}
\end{enumerate}
\end{theorem}

\paragraph{Outline of the proof.}
We use~\eqref{eq:stoch_dominance} and find that the interval $\mathcal{I} := [\lfloor n/2\rfloor - c_{n,\alpha},\,n]$ that contains both $\sum_{i=1}^n \mathbf{1}\{X_i \le \theta_0\}$ and $\sum_{i=1}^n \mathbf{1}\{X_i \ge \theta_0\}$ with probability at least $1-\alpha$; this is done in Step 1 of Algorithm~\ref{alg:proposed-conf-int}. The confidence interval $\widehat{\mathrm{CI}}_{n,\alpha}$ is the set of all $\theta\in\mathbb{R}$ such that both $\sum_{i=1}^n \mathbf{1}\{X_i \le \theta\}$ and $\sum_{i=1}^n \mathbf{1}\{X_i \ge \theta\}$ belong to the interval $\mathcal{I}$; this is done in Step 2 of~\Cref{alg:proposed-conf-int}. This implies that the confidence interval returned by Algorithm~\ref{alg:proposed-conf-int} has a finite sample coverage for all distributions with median $\theta_0$, proving part 1 of Theorem~\ref{thm:assym_of_cn}. 

For part 2 of Theorem~\ref{thm:assym_of_cn}, we note that the confidence interval $\widehat{\mathrm{CI}}_{n,\alpha}$ can be represented in terms of the order statistics if and only if $\mathcal{I}\subseteq[1, n]$ ($c_{n,\alpha} = \lfloor n/2\rfloor$ implies $\mathcal{I}$ contains $0$). In part 2 of Theorem~\ref{thm:assym_of_cn}, we argue that $\mathcal{I}\subseteq[1, n]$ if and only if $n \ge \log_2(2/\alpha)$. This condition on the sample size has an interesting connection to the impossibility of the existence of a finite width confidence intervals for the median $\theta_0$; see Remark~\ref{rem:impossibility-lanke}.

For part 3 of Theorem~\ref{thm:assym_of_cn}, note by the central limit theorem $\mbox{Bin}(n, 1/2)$ is approximately $N(n/2, n/4)$ and hence, $c_{n,\alpha}$ is asymptotically equal to $\sqrt{n/4}z_{\alpha/2}$. To obtain a result that is valid for all $n\ge1$, we use the universal inequalities for the distribution function of binomial law mentioned in \cite{zubkov2012full}. These universal inequalities precisely bound the binomial distribution function using the normal CDF and the KL divergence, $\mathrm{KL}(p, q)$, between two Bernoulli random variables with success probabilities $p$ and $q$, respectively. To verify part-$3$ of \Cref{thm:assym_of_cn}, we prove precise bounds for $\mathrm{KL}(p, 1/2)$ for all $p\in[0, 1]$ and obtain bounds for quantiles of $\mathrm{Bin}(n, 1/2)$. \Cref{fig:sample_size_1} show $c_{n,\alpha}-\sqrt{n}z_{\alpha/2}/2$ and the corresponding upper and lower bounds in \Cref{thm:assym_of_cn} under two scenarios, one when $\alpha$ is fixed to be either $0.01$ or $0.05$ and $n\geq \log_2(2/\alpha)$. \Cref{fig:sample_size_1} show $c_{n,\alpha} - \sqrt{n}z_{\alpha/2}$ and the corresponding upper and lower bounds in~\Cref{thm:assym_of_cn} when the sample size $n$ is fixed to be either 15 or 25 and $\alpha>2^{-(n-1)}$. We observe that the difference between the upper and lower bounds for $c_{n,\alpha}-\sqrt{n}z_{\alpha/2}/2$ mentioned in \Cref{thm:assym_of_cn} lies between three to five in all the four figures. This shows that the bounds obtained for $c_{n,\alpha}-\sqrt{n}z_{\alpha/2}/2$ are very precise. The asymptotics of $c_{n,\alpha}$ plays a very important role in our study of the asymptotics of the width of $\widehat{\mathrm{CI}}_{n,\alpha}$.

\begin{figure}[!htb]
    \centering
    \includegraphics[width=\textwidth,keepaspectratio]{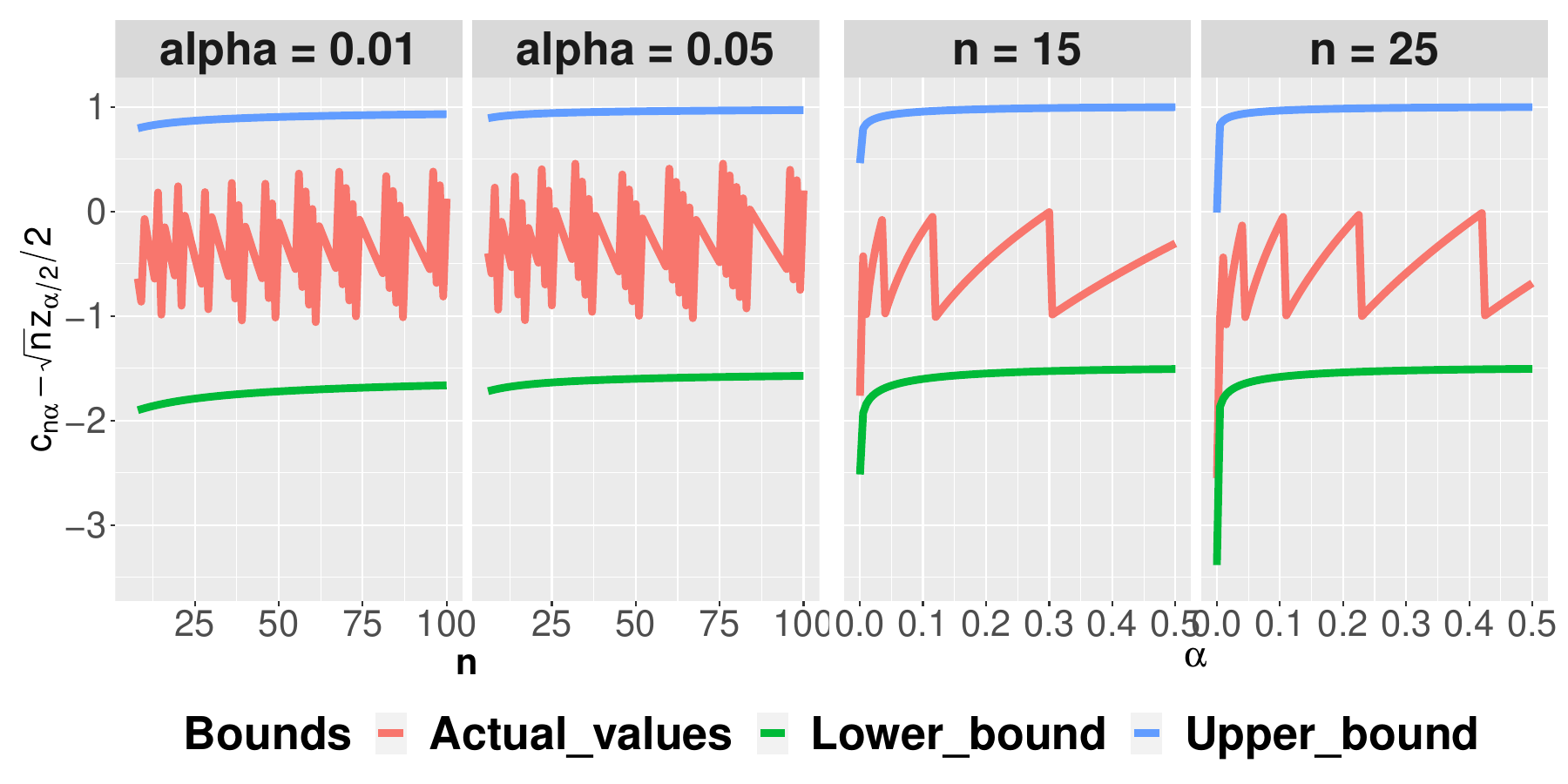}
    \caption{In the first two plots (from left-hand side) we see the plot of $c_{n,\alpha}-\sqrt{n}z_{\alpha/2}/2$ as $n$ varies ($n\geq \log_2(2/\alpha)$) for fixed $\alpha=0.01, 0.05$. In the following two plots we see the plot of $c_{n,\alpha}-\sqrt{n}z_{\alpha/2}/2$ as $\alpha$ varies ($\alpha>2^{-(n-1)}$) for fixed $n=15, 25$.}
    \label{fig:sample_size_1}
\end{figure}

\begin{remark}[Sample size condition and impossibility]\label{rem:impossibility-lanke}
Theorem~2 of \cite{lanke1974interval} claims that there do not exist real-valued, locally bounded functions $g_L$ and $g_R$ on $\mathbb{R}^n$ such that
\[
\mathbb{P}\bigl(g_L(\boldsymbol{X}_n)<\theta<g_R(\boldsymbol{X}_n)\bigr)>1-\alpha
\]
for every random vector $\boldsymbol{X}_n$ whose components are independent copies of a continuous random variable symmetric about $\theta$, provided that $\alpha<2^{-(n-1)}$, or equivalently, $n<\log_2(2/\alpha)$. This statement is false. However, the proof in~\cite{lanke1974interval} does show that it is impossible to construct a locally bounded confidence interval when
$n<\log_2(1/\alpha).$
Since
$\log_2(2/\alpha)-\log_2(1/\alpha)=1,$
our confidence intervals are almost surely bounded whenever the sample size exceeds the absolute minimum required sample size by one observation. Moreover, our confidence interval is shift-invariant.

The only nontrivial bounded confidence interval of which we are aware when $n=\log_2(1/\alpha)$ uses the fact that
\begin{equation}\label{eq:basic-inequality}
\mathbb{P}\bigl(\min\{a,X_i\}\le \theta_0\le \max\{a,X_i\}\bigr)\ge \frac{1}{2}
\qquad\text{for all } a\in\mathbb{R},
\end{equation}
which implies that
\[
\mathbb{P}\left(
\sum_{i=1}^n
\mathbf{1}\{\min\{a,X_i\}\le \theta_0\le \max\{a,X_i\}\}
\ge \lfloor n/2\rfloor-c_{n,2\alpha}
\right)\ge 1-\alpha.
\]
This confidence interval is valid for all $n\ge \log_2(1/\alpha)$, but it is never shift-invariant because $a$ is arbitrary. We conjecture that no shift-invariant, locally bounded confidence interval for the median exists when
$n<\log_2(2/\alpha).$

To verify~\eqref{eq:basic-inequality}, consider three cases. If $a<\theta_0$, then the definition of the median implies that $\mathbb{P}(X_i\ge \theta_0)\ge 1/2$. Hence, with probability at least $1/2$,
\[
\min\{a,X_i\}=a<\theta_0
\qquad\text{and}\qquad
\max\{a,X_i\}=X_i\ge \theta_0.
\]
If $a=\theta_0$, the claim is immediate. Finally, if $a>\theta_0$, the same argument, with the inequalities reversed, gives the result.
\end{remark}

\begin{remark}[Finite sample exact coverage of median]
\label{rem:fin_samp_exact}
The distribution-free confidence interval $\widehat{\mathrm{CI}}_{n,\alpha}$ provides exact coverage in finite samples i.e.\ $\mathbb{P}(\theta_0 \in \widehat{\mathrm{CI}}_{n,\alpha}) = 1 - \alpha$ provided $\theta_0$ is a continuity point of the distribution functions and there exists an integer $ r \in [1, n]$ such that $\alpha = 1 - \sum_{i = r}^{n+1 - r} \binom{n}{i} 2^{-n}$ (refer to \eqref{eq:coverage-lower-bound-Guilbaud} in \Cref{appsec:related-literature} for the details). Exact coverage for other $\alpha$'s requires randomization.
    
\end{remark}
\begin{remark}[Comparison with Wald and bootstrap confidence intervals]\label{rem:comparison}
It is important to note that the Wald confidence interval and the classical bootstrap confidence interval of median (discussed in \Cref{appsec:related-literature}) both ensure only the asymptotic coverage (i.e.\ for large enough sample sizes) to be close to $1-\alpha$. On the other hand, $\widehat{\mathrm{CI}}_{n,\alpha}$ ensures the coverage to be at least $1-\alpha$ for all sample sizes $n$. Moreover as mentioned in Example $5.24$ of \cite{vaart_1998} and Section $3.4.3$ of \cite{bose2018u} respectively, the required confidence level is guaranteed (asymptotically) for both the confidence intervals under the additional assumption of the differentiability of the distribution function at the median and the boundedness of the derivative of the distribution function at the median away from zero. As discussed in Section~\ref{appsec:related-literature}, such a requirement is also necessary for their validity as proved in~\cite{knight1998bootstrapping}. Hence, in cases, where the aforementioned assumption does not hold, $\widehat{\mathrm{CI}}_{n,\alpha}$ performs much better when compared to the Wald or the bootstrap confidence interval. See \Cref{sec:simulation} for a demonstration.  
\end{remark}
\begin{remark}[Coverage of quantiles close to the median]\label{rem:cover_quantile}
If $X_1, \ldots, X_n$ are independent and identically distributed from a CDF $F$, then the proof of~\Cref{thm:assym_of_cn} also implies coverage of quantiles close to the median, i.e., setting $\theta_{h}$ as the $(1/2 + h)$-th quantile of $F$, $\widehat{\mathrm{CI}}_{n,\alpha}$ acts as an asymptotically valid confidence interval for $\theta_h$ if $nh \to 0$. The miscoverage error rate depends quadratically on $nh$; see \Cref{thm:quantile_coverage} for details.     
\end{remark}

\section{Width Analysis}\label{sec:width-analysis}
In this section, we analyze the width of the confidence interval given by \Cref{alg:proposed-conf-int}. We start by noting that \cite{pena2019median} shows that the non-parametric confidence interval $\widehat{\mathrm{CI}}_{n,\alpha}$ minimizes the expected width in the subclass of symmetric distributions among all distribution-free finite sample valid confidence intervals of median.  Although the validity of the confidence interval holds even without the assumption of continuity of the distribution and identical distributions, analyzing the width will benefit from such an assumption. Technically, we can allow for non-identical distributions, but for notational convenience, we assume $X_i$'s are IID from CDF $F$. The continuity assumption implies the uniqueness of the median which gives any confidence interval a chance to shrink (or equivalently, the width of the confidence interval to converge to zero with growing sample size). If median of $F$ is not unique, then there exists an interval $[\theta_-, \theta_+]$ such that any point $\theta_0\in[\theta_-, \theta_+]$ is a median of $F$. From Theorem~\ref{thm:assym_of_cn}, it follows that with probability at least $1-2\alpha$, $\widehat{\mathrm{CI}}_{n,\alpha}$ contains both $\theta_-$ and $\theta_+$. Hence, $\mathrm{Width}(\widehat{\mathrm{CI}}_{n,\alpha}) \ge |\theta_+ - \theta_-|$ with probability at least $1 - 2\alpha$, which implies that the width cannot shrink to zero as $n\to\infty$ (unless $|\theta_+ - \theta_-| \to 0$). For this reason, we will assume in the remaining part of this section that $\theta_0$ is a continuity point of $F$.

Considering the asymptotics of the width of one confidence interval without a benchmark is not insightful. We now introduce the oracle confidence interval for the median that assumes knowledge of the ``smoothness'' of the true distribution. Suppose for the true distribution function $F(\cdot)$ with median $\theta_0$, there exist $\{a_n\}_{n\ge1}$ and a non-decreasing function $\psi:\mathbb{R}\to\mathbb{R}$ such that
\begin{equation}\label{eq:non-standard-setting}
\lim_{n\to\infty}\, n^{1/2}(F(\theta_0 + t/a_n) - F(\theta_0)) ~=~ \psi(t)\quad\mbox{for all}\quad t\in\mathbb{R}.
\end{equation}
Then it follows from~\citet[Theorem 2]{knight1998bootstrapping} that the sample median $X_{(\lceil n/2\rceil)}$ satisfies
\[
a_n(X_{\lceil n/2\rceil} - \theta_0) ~\overset{d}{\to}~ \psi^{-1}(N(0, 1/4)).
\]
Monotonicity of $\psi(\cdot)$ implies that the oracle Wald interval is given by
\begin{equation}\label{eq:oracle-Wald}
\widehat{\mathrm{CI}}_{n,\alpha}^{\mathtt{Wald}} := \left[X_{\lceil n/2\rceil} + \frac{\psi^{-1}(-z_{\alpha/2}/2)}{a_n},\, X_{\lceil n/2\rceil} + \frac{\psi^{-1}(z_{\alpha/2}/2)}{a_n}\right].
\end{equation}
Note that $\widehat{\mathrm{CI}}_{n,\alpha}^{\mathtt{Wald}}$ is inactionable because $\psi(\cdot)$ and $a_n$ are unknown in practice. 
Recall $z_{\alpha/2}$ is the $(1-\alpha/2)$-th quantile of standard Gaussian. With the oracle Wald interval as in~\eqref{eq:oracle-Wald}, we define the width ratio as
\[
\mathrm{WR}_{n,\alpha} ~:=~  ({\mbox{Width of $\widehat{\mathrm{CI}}_{n,\alpha}$}})/({\mbox{Width of $\widehat{\mathrm{CI}}_{n,\alpha}^{\mathtt{Wald}}$}}).
\]
As a special case, if $\psi(t) = M|t|^{\rho}\mbox{sign}(t)$ and $a_n = n^{1/(2\rho)}$, then we get $\psi^{-1}(z) = (z/M)^{1/\rho}\mbox{sign}(z)$. This implies that the width of the oracle Wald interval is $2n^{-1/(2\rho)}(z_{\alpha/2}/(2M))^{1/\rho}$. Although our results extend to the setting of~\eqref{eq:non-standard-setting} with general $\psi(\cdot)$, we focus on the case $\psi(t) = M|t|^{\rho}\mbox{sign}(t)$ for brevity. 
\subsection{Width analysis under standard assumptions}
To begin with, we assume that the underlying distribution function $F$ is differentiable at the population median $\theta_0$ with $F'(\theta_0)>0$. This is a standard assumption under which one can obtain the Bahadur representation for the sample median~\citep{bahadur1966note,ghosh1971new}. Under this assumption, \citet[Section 2.6.3]{serfling2009approximation} shows that the width ratio $\mathrm{WR}_{n,\alpha} \stackrel{P}{\rightarrow} 1$ as $n \rightarrow \infty$. For completeness, we provide the details in \Cref{thm:bahadur_asump} (in the Appendix).

We now provide a finite sample version of the convergence of $\mathrm{WR}_{n,\alpha}$ under a quantification of continuous differentiability. Formally, we assume that there exist $M, C, \delta, \eta > 0$ such that 
\begin{equation}
\label{asump:fin_bahadur}
|F(\theta_0+h)-F(\theta_0)-Mh| \leq C|h|^{1+\delta} \quad \forall\quad |h|<\eta,    
\end{equation}
It is easy to observe that this assumption implies that $F$ is differentiable at the population median $\theta_0$ and $F'(\theta_0)=M>0$. Note that the assumption is only required for $h$ in the neighborhood of zero. Under this assumption, we have the following result.
\begin{theorem}
\label{thm:bahadur_fin_sample}
Let $X_1,X_2,\ldots,X_n \stackrel{iid}{\sim} F$ with $F(\theta_0) = 1/2$. Suppose $F$ satisfies~\eqref{asump:fin_bahadur}. 
Define $\zeta :=(M/2)\min\{\eta,(M/2C)^{1/\delta}\}.$ Then for any $\alpha\in[0, 1]$ and $n\ge\log_2(2/\alpha),$ such that $n\geq 49\log(2n/\alpha)/\zeta^2$, with probability at least $1-6n^{-2}$, 
\begin{equation}\label{eq:bahadur_finite_sample_ineq}
  \left|\mathrm{WR}_{n,\alpha} - 1\right| \leq \frac{1 + 14\log(n)/z_{\alpha/2}}{n^{1/4}} + \sqrt{\frac{\log(2/\alpha)}{8n}} + \frac{2C(14)^{1+\delta}(\log(2n/\alpha))^{(1+\delta)/2}}{z_{\alpha/2}M^{1+\delta}n^{\delta/2}}.
\end{equation}
\end{theorem}
Theorem~\ref{thm:bahadur_fin_sample} studies the rate of convergence of $\mathrm{WR}_{n,\alpha}$ to $1$. For fixed values of $\alpha, M, C, \eta, \delta$, the rate of convergence is $\max\{\log(n)n^{-1/4}, (\log(n))^{(1+\delta)/2}n^{-\delta/2}\}$. For example, with $\delta = 1$, the rate of convergence in $\log(n)/n^{1/4}$. It should be noted here that the appearance of $\log(n)$ factors is only because the bound is guaranteed to hold with probability at least $1 - 6/n^2$ which is converging to $1$ as $n\to\infty$.

The advantage of Theorem~\ref{thm:bahadur_fin_sample} over the asymptotic statement in~\citet[Section 2.6.3]{serfling2009approximation} is that we can allow $\alpha$ and $M$ to tend to $0$ as $n\to\infty$. For example, if $\delta = 1$, then for every fixed $\alpha\in[0, 1]$ (assuming $C = O(1)$), the right hand side of~\eqref{eq:bahadur_finite_sample_ineq} converges to zero as $n\to\infty$ if $M^{-1} = O(n^{1/4}/\sqrt{\log(n)})$; for general $\delta> 0$, this condition becomes $M^{-1} = O(n^{\delta/(2 + 2\delta)}/\sqrt{\log(n)})$. It may be worth pointing that $\delta$ can be larger than 1.
 
A more detailed version of the result and its proof is provided in \Cref{appendix:theorem_bahadur_finsample}. Interested readers may refer to \Cref{subsec:bahadur} for a simulation study comparing the performance of the distribution-free confidence interval with the vanilla Wald confidence interval under the standard assumption that the distribution function is differentiable at the median and the derivative is bounded away from zero. 
\subsection{Width analysis under non-standard assumptions}
In this subsection, we analyze the width of the confidence interval when the assumption that the distribution function $F$ is differentiable at the population median $\theta_0$ with $F'(\theta_0)>0$ does not hold. General theory for the asymptotic limits of the sample median exists under such non-standard assumptions \citep{ghosh1981bahadur,knight1998bootstrapping,knight2002limiting}. We study the general case when density may be zero or may not even exist. The surprising finding of this subsection (and this paper as well) is that $\mathrm{WR}_{n,\alpha}$ does not converge in probability to a constant but converges in distribution to a non-degenerate random variable. 

We provide a finite-sample analysis of the width ratio $\mathrm{WR}_{n,\alpha}$ under non-standard cases and thereby precisely characterize the width of the distribution-free confidence interval.
\begin{theorem}
\label{thm:asym_result_irreg_case}
Let $X_1,X_2,\cdots,X_n \stackrel{iid}{\sim} F $. Suppose that $F$ is a continuous CDF\ with median $\theta_0$ and
\begin{equation}\label{eq:assumption-non-standard}
|F(\theta_0+h)-F(\theta_0)-M|h|^{\rho}\mathrm{sgn}(h)| \leq C|h|^{\rho+\Delta} \quad \forall \quad |h|<\eta,   
\end{equation}
where $0<M,C,\Delta,\eta,\rho<\infty$. Let $\delta = \Delta/ \rho$ and $\zeta =(M/2)\min\{\eta^{\rho},(M/2C)^{1/\delta}\}$. Also, define
\begin{equation*}
    \begin{split}
         Q &= \left(\frac{c_{n,\alpha} + \lceil n/2 \rceil+1}{n} - \frac{1}{2}\right) - \left(\frac{1}{n}\sum_{i=1}^n \mathbf{1}\left\{F(X_i) \le \frac{c_{n,\alpha} + \lceil n/2 \rceil+1}{n}\right\} - \frac{c_{n,\alpha} + \lceil n/2 \rceil+1}{n}\right), 
    \end{split}
\end{equation*}
and
\[
\mathscr{G}(a,b) := |a|^{1/\rho}\mathrm{sgn}(a) - |a - b|^{1/\rho}\mathrm{sgn}(a-b).
\]
Then for every $n\ge2, \alpha\in[0,1]$ such that $n\geq \max\{\log_2(2/\alpha), 49\log(2n/\alpha)/\zeta^2, 4z_{\alpha/2}^2\}$, with probability at least $1 - 1350 n^{-2}$, 

\begin{equation}\label{eq:non-regular-finite-sample}
\begin{split}
&\left|\mathrm{WR}_{n,\alpha} - \frac{n^{1/(2\rho)}}{2^{1 - (1/\rho)}z_{\alpha/2}^{1/\rho}}\mathscr{G}(Q,z_{\alpha/2}/\sqrt{n}) \right|\\  &\quad\leq  \max\{4, 2C_{\rho, \alpha}\} \left[\frac{208(\log n)^{3/4}}{n^{1/4}} + \frac{C}{n^{\delta/2}}\left(\frac{14\sqrt{\log(2n/\alpha)}}{M}\right)^{1+\delta} \right]^{\min\{1, 1/\rho\}} \\
   &\quad+ \max\{2, C_{\rho, \alpha}\}\left[\frac{4.2(\log(2/\alpha))^{1/4}\sqrt{\log(2n/\alpha)}}{n^{1/4}} + \frac{9.5 \sqrt{\log n}}{\sqrt{n}} \right]^{\min\{1, 1/\rho\}} ,
\end{split}
\end{equation}
where $C_{\rho, \alpha}$ is a constant depending on $\rho$ and $\alpha$ (see \eqref{eq:crho} for details). Moreover, for any fixed $\alpha\in[0, 1]$, as $n\to\infty$,
\begin{equation}\label{eq:limiting-distr-non-standard}
\mathrm{WR}_{n,\alpha} ~\stackrel{d}{\xrightarrow{}}~ \frac{1}{2}\mathscr{G}\left(\frac{Z}{z_{\alpha/2}} + 1 ,2\right),\quad \mathrm{where}\quad  Z \sim N(0, 1).
\end{equation}
\end{theorem}

Note that assumption~\eqref{eq:assumption-non-standard} of Theorem~\ref{thm:asym_result_irreg_case} allows for the density to be zero and infinity depending on whether $\rho \ge 1$ or $\rho < 1$. Note that the oracle Wald interval, under~\eqref{eq:assumption-non-standard}, converges to zero at the rate of $n^{1/(2\rho)}$. This is expected because if the density is infinity at the median (i.e., $\rho < 1$), then the sample median converges to $\theta_0$ at a rate faster than $n^{1/2}$ and accordingly, the width of the confidence interval shrinks at a rate faster than $n^{1/2}.$ 
\begin{remark}[Properties of the Limiting Distribution]
\label{cor:irregular_case}
If $\rho \geq 1$, then $\mathscr{G}(Z/z_{\alpha/2} + 1, 2) \le 2$ because of the fact that $h(x) = |x|^{1/\rho}\mathrm{sgn}(x)$ is H{\"o}lder continuous with $|h(x) - h(y)| \leq 2^{1-(1/\rho)} |x - y|^{1/\rho}$ (see Proposition~\ref{prop:holder} for a proof). Therefore \Cref{thm:asym_result_irreg_case} implies that under $\rho \geq 1$, the width ratio $\mathrm{WR}_{n,\alpha}$ is bounded in probability by $1$ i.e.\ $\mathrm{WR}_{n,\alpha} \leq 1 + o_P(1)$. Accordingly, and perhaps surprisingly, the non-parametric interval is asymptotically better than the inactionable oracle Wald interval. From the limiting distribution, it can be verified that with an asymptotic probability of $\alpha$, the width of the $(1-\alpha)$ non-parametric interval is at most $2^{(1/\rho) - 1}$-fraction of the width of the $(1-\alpha)$ Wald interval for $\alpha \in (0,1)$ and for all $\rho \ge 1$ (see Proposition~\ref{prop:boundedness_of_limit} for a proof).
\end{remark}
Theorem~\ref{thm:asym_result_irreg_case} generalizes Theorem~\ref{thm:bahadur_fin_sample} allowing for $\rho \neq 1$. Note that for $\rho = 1$,
\[
\mathscr{G}(a, b) = |a| \mathrm{sgn}(a) - |a - b| \mathrm{sgn}(a - b) = a - (a - b) = b,
\]
which implies $\mathscr{G}(Z/z_{\alpha/2}, 2)/2 = 1$. In other words, we recover the guarantee of \Cref{thm:bahadur_fin_sample} for $\rho = 1$. An interesting consequence of Theorem~\ref{thm:asym_result_irreg_case} is that the width ratio converges to a non-degenerate distribution shown in~\eqref{eq:limiting-distr-non-standard} if $\rho\neq1$. Similar to Theorem~\ref{thm:bahadur_fin_sample}, the variables $M, \alpha$ can be allowed to tend to zero with sample size $n$. For example, for any fixed $\alpha\in(0, 1)$ (assuming $C = O(1)$), the right hand side of~\eqref{eq:non-regular-finite-sample} converges to zero as long as $M^{-1} = O(n^{\delta/(2 + 2\delta)}/\sqrt{\log n})$. 

For $\alpha = 0.1, 0.05, 0.01$, the limiting densities of the width ratio are shown in \Cref{fig:non_reg_equal} as $\rho$ varies from $0.75$ to $10$. It can be seen that for all values of $\alpha, \rho$ the density of $\mathrm{WR}_{n,\alpha}$ has a sharp peak at one. Moreover the variance of the ratio increases as $\alpha$ increases. The limiting distribution is right-skewed for $\rho <1$. The limiting distribution is left-skewed for $\rho \ge 1$ and has a bounded support (see \Cref{cor:irregular_case}).  
\begin{figure}[!htb]
    \centering
    \includegraphics[width=\textwidth,keepaspectratio]{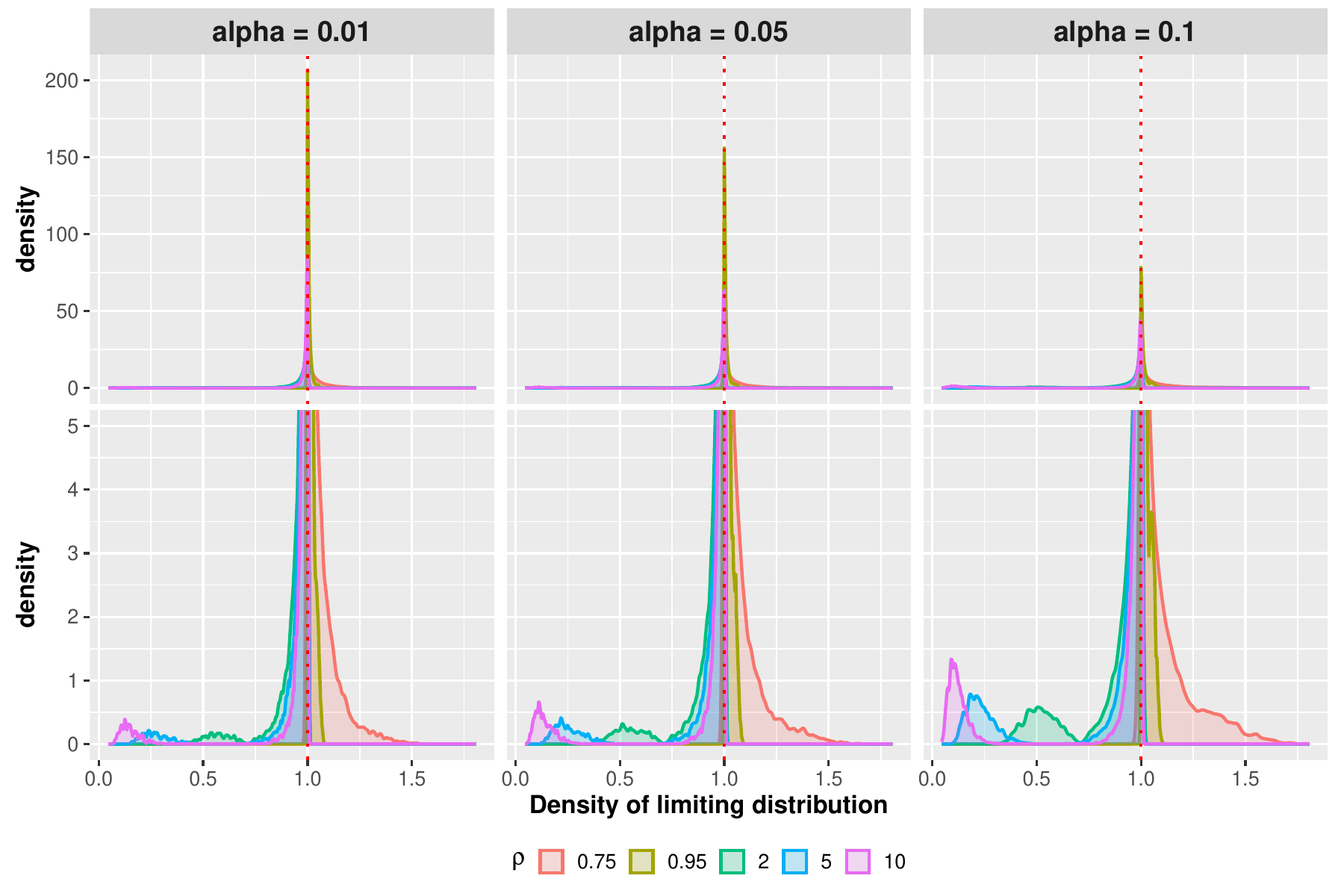}
    \caption{Limiting density of the ratio $\mathrm{WR}_{n,\alpha}$ for different values of the level of significance $\alpha = 0.01, 0.05, 0.1$ as $\rho$ varies from $0.5$ to $10$. The first row shows the exact density plots and the second row shows the zoomed-in plots for better clarity. For $\rho \geq 1$, the width ratio $\mathrm{WR}_{n,\alpha}$ is bounded above in probability by $1$.}
    \label{fig:non_reg_equal}
\end{figure}

A more detailed version of \Cref{thm:asym_result_irreg_case} and its proof are provided in \Cref{appendix_thm_asym_irreg}. Refer to \Cref{sec:simulation} for simulations to study $(i)$ the convergence of the width ratio $\mathrm{WR}_{n,\alpha}$ to $\mathscr{G}((Z/ z_{\alpha/2}) + 1, 2)$, $(ii)$ the coverage and width of $\widehat{\mathrm{CI}}_{n,\alpha}$ in comparison to that of subsampling and bootstrap based confidence intervals. 

Although Theorem~\ref{thm:asym_result_irreg_case} relaxes the assumption of the Bahadur representation, they require the existence of a density (in the extended real line) at $\theta_0$. As a final result, we provide an extension of Theorem~\ref{thm:asym_result_irreg_case} explicitly allowing for the non-existence of density at the median $\theta_0.$

\begin{theorem}
\label{thm:nonreg_unequal_limits}
Let $X_1,X_2,\cdots,X_n \stackrel{iid}{\sim} F $. Suppose that $F$ is a continuous CDF with median $\theta_0$ and
\begin{equation}\label{eq:non-differentiable-non-standard}
|F(\theta_0+h)-F(\theta_0)-|h|^{\rho}\mathrm{sgn}(h)[M_{-}\mathbf{1}\{h<0\}+ M_+\mathbf{1}\{h>0\}]| \leq C|h|^{\rho+\Delta} \quad \forall \quad |h|<\eta,   
\end{equation}
where $0<M_{-},M_+,C,\Delta,\eta,\rho<\infty$. Set 
\[
M = \min\{M_-, M_+\}, \quad \delta = \Delta/ \rho,\quad  \zeta =(M/2)\min\{\eta^{\rho},(M/2C)^{1/\delta}\}. 
\]
Also, define $Q$ as in Theorem~\ref{thm:asym_result_irreg_case} and,
\begin{equation*}
\begin{split}
    &\overline{\mathscr{G}}(a,b) = |a|^{1/\rho}\mathrm{sgn}(a)\left[\frac{\mathbf{1}\{a <0 \}}{M_{-}^{1/\rho}} + \frac{\mathbf{1}\{a >0 \}}{M_{+}^{1/\rho}} \right] -|a-b|^{1/\rho}\mathrm{sgn}(a-b)\left[\frac{\mathbf{1}\{a < b \}}{M_{-}^{1/\rho}} + \frac{\mathbf{1}\{a >b \}}{M_{+}^{1/\rho}} \right].
\end{split}
\end{equation*}
Then for every $n\ge2, \alpha\in[0,1]$ such that $n\geq \max\{\log_2(2/\alpha), 49\log(2n/\alpha)/\zeta^2, 4z_{\alpha/2}^2\}$, with probability at least $1 - 1350 n^{-2}$, inequality~\eqref{eq:non-regular-finite-sample} holds true when $\mathscr{G}$ is replaced with $\overline{\mathscr{G}}$. Moreover for any fixed $\alpha\in[0, 1]$ we have the following distributional convergence as $n\to\infty$
\[
\mathrm{WR}_{n,\alpha} \stackrel{d}{\rightarrow} \frac{1}{2^{-1/\rho}z_{\alpha/2}^{1/\rho}\left[M_-^{-1/\rho} + M_+^{-1/\rho} \right]} \overline{\mathscr{G}}(W,z_{\alpha/2}),\quad \mathrm{where}\quad  W \sim N(z_{\alpha/2}/2,1/4).
\]
\end{theorem}
If $\rho = 1$, then \eqref{eq:non-differentiable-non-standard} implies that the left derivative of $F$ at $\theta_0$ is $M_-$ and the right derivative at $\theta_0$ is $M_+$. If $M_- = M_+$, then condition~\eqref{eq:non-differentiable-non-standard} becomes~\eqref{eq:assumption-non-standard}. 
\begin{remark}(Properties of Limiting Distribution)
\label{rem:limit_unequal}
If $\rho \geq 1$, then $\overline{\mathscr{G}}(W,z_{\alpha/2})$ is bounded above by $ 2^{1 - (1/\rho)}z_{\alpha/2}^{1/\rho}\max\{ M_-^{-1/\rho} , M_+^{-1/\rho} \}$ (see Proposition~\ref{prop:unequal_limits_bound_G} for a proof). Therefore \Cref{thm:nonreg_unequal_limits} implies that under $\rho \geq 1$, the width ratio $\mathrm{WR}_{n,\alpha}$ is bounded in probability by $2\max\{ M_-^{-1/\rho} , M_+^{-1/\rho} \}/ (M_-^{-1/\rho} + M_+^{-1/\rho})$. Note that if $M_- = M_+ = M$ the upper bound simplifies to $\mathrm{WR}_{n,\alpha} \leq 1 + o_P(1)$ i.e.\ we recover the bound in \Cref{cor:irregular_case} when the density exists in the extended real line.    
\end{remark}
\Cref{fig:non_reg_unequal} shows the limiting density of the width ratio $\mathrm{WR}_{n,\alpha}$ when $\rho \in \{0.75, 0.95, 2, 5, 10\}$ and $(M_-, M_+) \in\{(0.5, 0.5), (0.2, 0.8), (0.4, 0.6)\}$. It can be observed that the variance of the limiting distributions decrease with increase in $\rho$. 
The proof of \Cref{thm:nonreg_unequal_limits} follows similar techniques as those of \Cref{thm:asym_result_irreg_case}. A more detailed version of the theorem and its proof is provided in \Cref{appendix:nonreg_unequal_limits}. 
\begin{figure}[!htb]
    \centering
    \includegraphics[width=\textwidth,keepaspectratio]{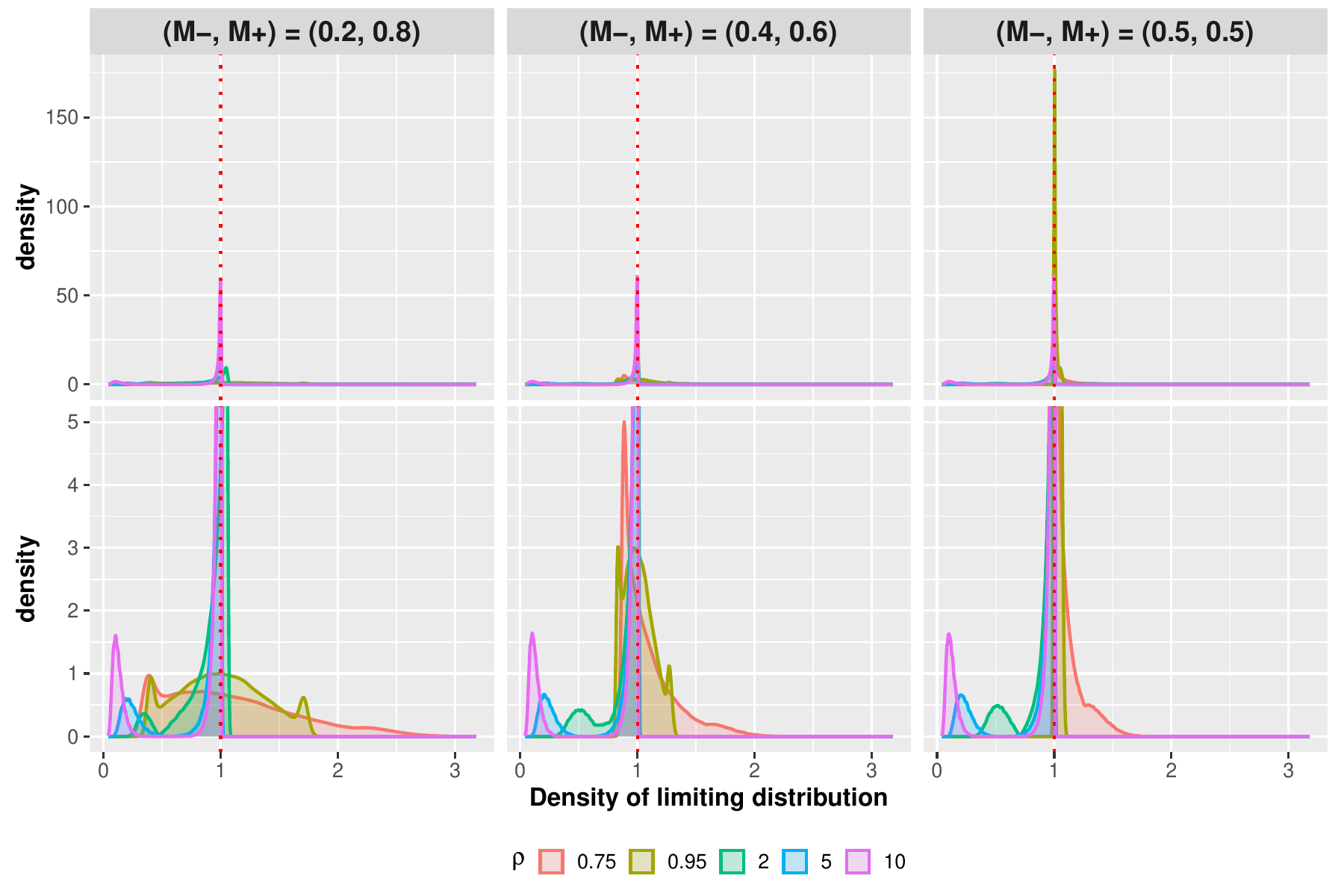}
    \caption{Limiting density of the ratio $\mathrm{WR}_{n,\alpha}$ for different values of $(M_-, M_+) \in \{(0.5, 0.5), (0.2, 0.8), (0.4, 0.6)\}$ at the level of significance $\alpha = 0.01$ as $\rho$ varies from $0.5$ to $10$. The first row shows the exact density plots and the second row shows the zoomed-in plots for better clarity. For $\rho \geq 1$, the width ratio $\mathrm{WR}_{n,\alpha}$ is bounded above in probability by $2\max\{ M_-^{-1/\rho} , M_+^{-1/\rho} \}/ (M_-^{-1/\rho} + M_+^{-1/\rho})$.}
    \label{fig:non_reg_unequal}
\end{figure}

All the results presented in this section assume that the growth rate of $F(\cdot)$ on either side of $\theta_0$ is the same (except maybe for some constants). Even in Theorem~\ref{thm:nonreg_unequal_limits}, our assumption~\eqref{eq:non-differentiable-non-standard} implies $|F(\theta_0 + h) - F(\theta_0)| \asymp |h|^{\rho}$ as $h\to0$. The techniques involved can be easily generalized to handle more general cases where, for example, $|F(\theta_0 + h) - F(\theta_0)| \asymp |h|^{\rho_1}\mathbf{1}\{h > 0\} + |h|^{\rho_2}\mathbf{1}\{h < 0\}$ as $h\to0.$ One can derive a general result under~\eqref{eq:non-standard-setting}, but for brevity, we do not pursue it in this work.

\section{Generalized HulC}\label{sec:hulc_extension}

In this section, we propose a generalized version of HulC (Hull based Confidence Regions) where we update the method developed in \cite{kuchibhotla2021hulc}. The median bias of an estimator $\widehat{\theta}$ for a ``target'' $\theta_0$ is defined as 
\begin{equation*}
    \mbox{Med-Bias}_{\theta_0}(\widehat{\theta}) = \left(\frac{1}{2} - \min\left\{\mathbb{P}(\widehat{\theta} \ge \theta_0),\,\mathbb{P}(\widehat{\theta} \le \theta_0)\right\}\right)_+.
\end{equation*}
Define
\[
P(n,k) = \sum_{i = \lfloor n/2 \rfloor - k}^{\lceil n/2 \rceil + k} \binom{n}{k}\frac{1}{2^{n}}. 
\]

\begin{algorithm}
    \caption{Confidence Interval of $\theta_0$ (GHulC)}
    \label{alg:ghulc}
    \KwIn{Sample: $W_1,\cdots, W_N$ and Confidence Level: $1-\alpha$, estimation procedure: $\mathcal{A}(\cdot)$, Number of batches: $B \ge \log_2(2/\alpha)$.}
    \KwOut{A confidence interval $\widehat{\mathrm{CI}}_{N,\alpha}^{\mathtt{GHulC}}$ such that $\mathbb{P}(\theta_0 \in \widehat{\mathrm{CI}}_{N,\alpha}^{\mathtt{GHulC}}) \geq 1 - \alpha$.}
    Randomly split the data $W_1,\cdots,W_N$ into $B$ disjoint sets $\{\{W_i: i \in S_j\} : 1 \leq j \leq B\}$. These need not be equal sized sets, but having approximately equal sizes yields good width properties. \\
    Compute estimators $\widehat \theta_j := \mathcal{A}(\{W_i: i \in S_j\})$ for $1 \leq j \leq B$. \\
    Compute $c_{B,\alpha} = \inf\{x:\,\mathbb{P}( Y_{B} \geq \lfloor B/2 \rfloor-x) \ge 1 - \alpha/2\}$ where $Y_{B}\sim\text{Bin}(B,1/2)$. \\
    Generate a uniform(0,1) random variable $U$ and set, 
    \begin{equation*}
        \tau_{\alpha} = \frac{P(B,c_{B,\alpha}) - (1-\alpha)}{P(B,c_{B,\alpha}) - P(B,c_{B,\alpha} - 1)}, \quad \quad c_{B,\alpha}^* = \begin{cases}
            c_{B,\alpha} - 1 &\mbox{ if } U \leq  \tau_{\alpha}, \\
            c_{B,\alpha} &\mbox{ if } U >  \tau_{\alpha}.
        \end{cases}
    \end{equation*} . \\
    Return the confidence interval $\widehat{\mathrm{CI}}_{N,\alpha}^{\mathtt{GHulC}}$ 
    \[
    \widehat{\mathrm{CI}}_{N,\alpha}^{\mathtt{GHulC}} := \left[\widehat \theta_{(\lfloor{B/2} \rfloor -c_{B,\alpha}^*)},\ \widehat \theta_{(\lceil{B/2}\rceil +c_{B,\alpha}^*+1)}\right].
    \]
\end{algorithm}
The algorithm for Generalized HulC (GHulC) is proposed in \Cref{alg:ghulc}.

Instead of returning the range of estimators obtained from each split, GHulC splits the data into a larger number of disjoint subsets and instead returns a confidence interval for the population median of the estimators. The advantage of using GHulC is that the algorithm provides confidence intervals of smaller width (when $B > \log_2(2/\alpha)$) while still maintaining the required coverage for asymptotically median unbiased estimators.

\subsection{Coverage analysis of GHulC}
To show that GHulC maintains the required coverage for asymptotically median unbiased estimators we will use Theorem~\ref{thm:quantile_coverage} with $\widehat{\theta}_j, 1\le j\le B$ as the sample. Let $\mathcal{E}_{B}$ be the maximum of median biases of these estimators, i.e., 
\[
\mathcal{E}_B := \max_{1\le j\le B}\mbox{Med-Bias}_{\theta_0}(\widehat{\theta}_j).
\]

\begin{theorem}
\label{thm:coverage_ghulc}
    If $W_1, \ldots, W_N$ are independent observations, then for any $B \ge \log_2(2/\alpha)$,
    \begin{equation}\label{eq:miscoverage-GHulC}
    \begin{split}
    \mathbb{P}(\theta_0 \notin \widehat{\mathrm{CI}}_{N,\alpha}^{\mathtt{GHulC}}) \le~ \alpha\left(1 + 2B^2\mathcal{E}_B^2e^{2B\mathcal{E}_B}\right) \quad \mbox{for every }\alpha \in (0,1).
    \end{split}
    \end{equation}
    Hence, if $B\mathcal{E}_B \to 0$ as $N\to\infty$, then $\widehat{\mathrm{CI}}_{N,\alpha}^{\mathtt{GHulC}}$ is an asymptotically valid $1-\alpha$ confidence interval for $\theta_0$. Additionally, if $\mathbb P(\widehat \theta_j \leq \theta_0) = \mathbb P(\widehat \theta_1 \leq \theta_0) $ and $\mathbb P(\widehat \theta_j = \theta_0) = 0$ for all $j \in \{1, \cdots, B\}$, then
    \[
    \mathbb{P}(\theta_0 \notin \widehat{\mathrm{CI}}_{N,\alpha}^{\mathtt{GHulC}}) \geq \alpha \quad \mbox{for every }\alpha \in (0,1). 
    \]
\end{theorem}
The proof of \Cref{thm:coverage_ghulc} is provided in \Cref{appendix:thm:coverage_ghulc}. The upper bound on the mis-coverage probability of the GHulC confidence interval is exactly the same as that of HulC (see Theorem-2 of \cite{kuchibhotla2021hulc}). Thus in terms of coverage GHulC procedure enjoys the same properties as that of HulC. Theorem~\ref{thm:coverage_ghulc} implies that if $B\mathcal{E}_B\to 0$ as $n\to\infty$ (and allowing for $B\to\infty$), then the miscoverage probability of GHulC converges to $\alpha$, i.e., GHulC is asymptotically exactly with $1-\alpha$ coverage. For details refer to the discussion in remark 2.3 in \cite{kuchibhotla2021hulc}. Moreover GHulC yields shorter confidence intervals when compared to HulC, as shown below. 

Another thing to note is that the coverage probability $P(B, c)$ increases in steps as $c$ increases over the positive integers. This can lead to conservative coverage i.e.\ miscoverage probability strictly less than $\alpha$. This is because there might not exist positive integer $c$ such that $P(B, c) = 1 - \alpha$. To solve this problem we adopt a randomization procedure. We know from previous derivations that $c_{B,\alpha}$ is the smallest positive integer so that $P(B, c_{B,\alpha}) \geq 1 - \alpha$. Thus we randomize between $c_{B,\alpha} - 1$ and $c_{B,\alpha}$ with probability $\tau_{\alpha}$ (see step-4 of \Cref{alg:ghulc}) to ensure that the coverage probability is exactly $1-\alpha$. 

\subsection{Width analysis of GHulC}
Throughout this section we shall assume that $B\mathcal{E}_B \rightarrow 0$ as $N \rightarrow \infty$ so that \Cref{thm:coverage_ghulc} implies the required coverage. We shall analyze the width of confidence interval returned by GHulC under two different assumptions. 
\begin{theorem}
    \label{thm:easy_width_ghulc}
Suppose $ \widehat{\mathrm{CI}}_{N,\alpha}^{\mathtt{GHulC}}$ is the confidence interval returned by GHulC (\Cref{alg:ghulc}) using $B$ splits of approximately equal sizes. Let $\widehat \theta^m$ be an estimator of $\theta_0$ based on a sample of size $m$ and let $r_m$ be its rate of convergence i.e., 
\[
r_m(\widehat \theta^m - \theta_0) = O_p(1) \quad \mathrm{ as }\quad  m \rightarrow \infty. 
\]
We assume the following regarding the distribution function $\Tilde F_{N/B}(\cdot)$ of $r_{N/B}(\widehat \theta_j^{N/B} - \theta_0)$, 
\[
|\Tilde F_{N/B}(x) - \Tilde F_{N/B}(0)| > \mathscr{C}|x|^{\rho} \quad \forall \quad |x| < \Tilde{\Delta},
\]
for some $\mathscr{C}, \Tilde{\Delta} > 0$. Let $|\Tilde F_{N/B}(0) - (1/2)| \leq \mathcal{E}_B$. Then we have the following with probability greater than or equal to $1 - \delta$, 
\begin{equation*}
    \begin{split}
        &  \mathrm{Width}(\widehat{\mathrm{CI}}_{N,\alpha}^{\mathtt{GHulC}}) \le \frac{2}{\mathscr{C}^{1/\rho}r_{N/B}} \left\{\frac{5\log(2/\delta) + \sqrt{2\log(2/\alpha)}}{2\sqrt{B}} + \frac{2}{B} + \mathcal{E}_B \right\}^{1/\rho},
    \end{split}
\end{equation*}
provided $B$ is large enough to ensure that the right hand side is less than $2 \Tilde{\Delta}/r_{N/B}$. 
\end{theorem}
\Cref{thm:easy_width_ghulc} states that the rate of $\mathrm{Width}(\widehat{\mathrm{CI}}_{N,\alpha}^{\mathtt{GHulC}})$ is $B^{1/(2\rho)}r_{N/B}$ i.e.,\ $B^{1/(2\rho)}r_{N/B}\mathrm{Width}(\widehat{\mathrm{CI}}_{N,\alpha}^{\mathtt{GHulC}}) = O_P(1)$ as $N \rightarrow \infty$. The proof of \Cref{thm:easy_width_ghulc} can be seen in \Cref{appendix:easy_ghulc}. 

The next theorem is based on a stronger assumption. The assumption made is similar in nature to those made in \Cref{thm:asym_result_irreg_case} and \Cref{thm:nonreg_unequal_limits}. 
\begin{theorem}
  \label{thm:width_ghulc}  
Suppose $ \widehat{\mathrm{CI}}_{N,\alpha}^{\mathtt{GHulC}}$ is the confidence interval returned by GHulC (\Cref{alg:ghulc}) using approximately equal $B$ splits. Let $\widehat \theta^m$ be an estimator of $\theta_0$ based on a sample of size $m$ and let $r_m$ be its rate of convergence i.e., 
\[
r_m(\widehat \theta^m - \theta_0) = O_p(1) \quad \mathrm{ as }\quad m \rightarrow \infty. 
\]
We assume the following regarding the distribution function $\Tilde F_{N/B}(\cdot)$ of $r_{N/B}(\widehat \theta_j^{N/B} - \theta_0)$, 
\[
\left|\Tilde F_{N/B}(t) - \Tilde F_{N/B}(0) - M_N|t|^{\rho}\mathrm{sgn}(t) \right| \leq C_N |t|^{\rho+\Delta} \quad \forall \quad |t|< \eta,
\]
where $0< M_N,C_N,\Delta,\eta,\rho< \infty$. Let $|\Tilde F_{N/B}(0) - (1/2)| \leq \mathcal{E}_B$ and $B\mathcal{E}_B \to 0$ as $N\to\infty$. Then the following distributional convergence holds as $N/B,B \rightarrow \infty$, 
\begin{equation*}
    \begin{split}
       (B)^{1/2\rho}r_{N/B}M_N^{1/\rho}\mathrm{Width}(\widehat{\mathrm{CI}}_{N,\alpha}^{\mathtt{GHulC}}) 
    \stackrel{d}{\xrightarrow{}} \mathscr{G}(W, z_{\alpha/2}),
    \end{split}
\end{equation*}
where $\mathscr{G}(a,b) := |a|^{1/\rho}\mathrm{sgn}(a) - |a - b|^{1/\rho}\mathrm{sgn}(a-b)$ and $W \sim N(z_{\alpha/2}/2,1/4)$.
\end{theorem}
\Cref{thm:width_ghulc} states that under suitable regularity conditions on the distribution function $\Tilde F_{N/B}$ we can obtain the exact limiting distribution (non-degenerate in most scenarios) of the scaled width of the confidence interval returned by GHulC as $N,B \rightarrow \infty$. As in \Cref{thm:easy_width_ghulc}, the rate of convergence of $\mathrm{Width}(\widehat{\mathrm{CI}}_{N,\alpha}^{\mathtt{GHulC}})$ in \Cref{thm:width_ghulc} is $(B)^{1/2\rho}r_{N/B}$. The main idea behind the proof of \Cref{thm:width_ghulc} is same as that for \Cref{thm:asym_result_irreg_case}. A finite-sample generalisation of \Cref{thm:width_ghulc} has been stated and proved in \Cref{appendix:width_ghulc}.

\begin{remark}[Interpretation of the rate of convergence of width of $ \widehat{\mathrm{CI}}_{N,\alpha}^{\mathtt{GHulC}}$]\label{rem:interpretation_ghulc} We note that the rate of convergence of the width of confidence interval returned by GHulC is $(B)^{1/2\rho}r_{N/B}$. Thus the rate is composed of two components: $(B)^{1/2\rho}$ which models the regularity of the distribution of properly scaled and centered estimator $r_{N/B}(\widehat \theta_j^{N/B} - \theta_0)$ at $0$; and $r_{N/B}$ which is the rate of convergence of each estimator $\widehat \theta_j^{N/B}$ based on a sample of size roughly $N/B$. 
\end{remark}

\subsubsection{Comparison with Wald confidence intervals}
In this sub-section, we shall see how the width of the confidence interval returned by GHulC compares to that of Wald confidence interval, under asymptotic normality. We assume that GHulC is using approximately equal $B$ splits. Suppose the following holds, 
\[
\sqrt{N}(\widehat \theta - \theta_0) \stackrel{d}{\rightarrow} N(0, 1), \quad \quad
\sqrt{N/B}(\widehat \theta_j - \theta_0) \stackrel{d}{\rightarrow} N(0, 1),
\]
as $B,N/B \rightarrow \infty$ for $1 \leq j \leq B$. Here $\widehat \theta$ is the estimator based on the entire data. We also assume that $\Tilde F_{N/B}$ (the distribution function of $\sqrt{N/B}(\widehat \theta_j - \theta_0)$) satisfies the following, 
\begin{equation}\label{eq:local-limit-CLT}
\left| \Tilde F_{N/B}(t) - \Tilde F_{N/B}(0) - (1/(\sqrt{2 \pi}))t \right| \leq C_N |t|^2 \quad \mbox{ for } |t| < \eta,
\end{equation}
where $0 < C_N,\eta < \infty$. The exponent $2$ on the right hand side can be replaced with $1 + \delta$ for any $\delta > 0$. We note that a result of the type~\eqref{eq:local-limit-CLT} can often be obtained either using Edgeworth expansions or local limit theorems; see, e.g., Chapters VI and VII of~\cite{Petrov1975Sums}. Compared to the setting of \Cref{thm:width_ghulc}, we see that $M_N = 1/(\sqrt{2\pi}), \mbox{ }r_{N/B} = \sqrt{N/B},\mbox{ } \rho = 1$. Since $\rho = 1$, $\mathscr{G}(W, z_{\alpha/2}) = W - (W - z_{\alpha/2}) = z_{\alpha/2}$. Using \Cref{thm:width_ghulc} gives us the following as $B,N/B \rightarrow \infty$, 
\begin{equation*}
    \begin{split}
        &\sqrt{B}\sqrt{N/B}\frac{1}{\sqrt{2 \pi}} \mathrm{Width}(\widehat{\mathrm{CI}}_{N,\alpha}^{\mathtt{GHulC}}) \stackrel{d}{\rightarrow} z_{\alpha/2} \\
    \implies & \sqrt{N} \mathrm{Width}(\widehat{\mathrm{CI}}_{N,\alpha}^{\mathtt{GHulC}}) \stackrel{P}{\rightarrow}  \sqrt{2 \pi}  z_{\alpha/2}. 
    \end{split}
\end{equation*}
The oracle Wald confidence interval for this problem is given by $\widehat{\mathrm{CI}}_{N,\alpha}^{\mathtt{Wald}} = [\widehat \theta \pm \sigma z_{\alpha/2}/ \sqrt{N}]$. Hence,
\[
\frac{ \mathrm{Width}(\widehat{\mathrm{CI}}_{N,\alpha}^{\mathtt{GHulC}}) }{ \mathrm{Width}(\widehat{\mathrm{CI}}_{N,\alpha}^{\mathtt{Wald}}) } \stackrel{P}{\rightarrow} \sqrt{\frac{\pi}{2}} > 1. 
\]
We should note here that the Wald confidence interval explicitly makes use of the asymptotic normality of the estimator. GHulC, on the other hand, relies mainly on the (asymptotic) median-unbiasedness of the estimator. It was shown in section 2.3 of \cite{kuchibhotla2021hulc} that the the ratio of the expected width of HulC confidence interval to that of the Wald interval is approximately equal to $\sqrt{\log_2(\log_2(2/\alpha))}$ which grows slowly to $\infty$ as $\alpha \rightarrow 0$. The advantage of GHulC is that the ratio of the widths is independent of $\alpha$ and thus GHulC can potentially produce much smaller valid confidence intervals than HulC as $\alpha \rightarrow 0$.

\subsection{Simulations}
\label{subsec:sim_ghulc}
To understand the validity and power of the confidence intervals generated by GHulC we consider the following numerical example of multivariate quantile regression. Suppose $(X_i,Y_i) \in \mathbb{R}^4 \times \mathbb R$, $1 \leq i \leq n$ are independent and identically distributed random vectors from the linear model,
\[
Y_i = \theta_0^\top X_i + \epsilon_i\quad \mbox{for} \quad i \in \{1, \cdots, n\}. 
\]
We define the estimator $\widehat \theta_n$ as follows,
\[
\widehat \theta_n = \argmin_{\theta \in \mathbb{R}^4}\sum_{i = 1}^n |Y_i - \theta^\top X_i|.
\]

We wish to obtain a valid $(1 - \alpha)$ confidence interval for $\theta_{0,1} = e_1^\top \theta_0$ where $e_1 = (1, 0, 0, 0)$. It can be easily checked that the estimator $\widehat \theta_{n, 1} = e_1^\top \widehat \theta_n$ is asymptotically median-unbiased for $\theta_{0,1}$; see~\cite{knight1998limiting}. 
We generate the random vectors $\{X_i\}_{i = 1}^n$ from the following distribution,
\begin{equation*}
    \begin{split}
        X_i \stackrel{iid}{\sim} \mathcal N(\mu, \Sigma) \quad \mbox{where} \quad \mu = (2, 3, 4, 5), \mbox{ } \Sigma = \begin{pmatrix}
            1 & 0.6 & 0.3 & 0.2 \\
            0.6 & 1 & 0.4 & 0.3 \\
             0.3 & 0.4 & 1 & 0.5 \\
              0.2 & 0.3 & 0.5 & 1
        \end{pmatrix} . 
    \end{split}
\end{equation*}
We suppose that $\epsilon_i$ and $X_i$ are independent and $F_i(x) = \mathbb{P}(\epsilon_i \leq x) = 0.5(1 +  \mathrm{sgn}(x)|x|^{\beta})$ where $x \in [-1,1]$ for some $\beta > 0$. If $\beta = 1$ then this is the standard setting of error distribution with density bounded away from zero. If $\beta < 1$ then the rate of convergence of the quantile estimator is faster than $n^{1/2}$. If $\beta > 1$ then the rate of convergence is slower than $n^{1/2}$. We generate data for values of $\beta \in \left[0,2\right)$ and compare the performance of HulC and GHulC (at level $\alpha = 0.05$) for higher values of $B$. In particular for the purpose of simulations, we have take $B = 12, 18, 24$ which are multiples of $\lceil \log_2(2/\alpha) \rceil = 6$ (for $\alpha = 0.05$). The performance of each procedure is based on $1000$ Monte Carlo replications for each sample size ($n = 200, 500, 1000, 2000$) and each $\beta$. We observe from \Cref{fig:compare_ghulc_multi} that like HulC, the generalized version of HulC also maintains the coverage at the nominal level of $0.95$ for all sample sizes. Moreover from \Cref{fig:compare_ghulc_multi}, we can also infer that GHulC with higher value of $B$ yields confidence intervals of smaller width. \Cref{fig:ghulc_c_multi} clearly suggests that the dispersion of the width of the confidence interval returned by GHulC also tends to decrease with increasing values of $B$. Interested readers can further refer to \Cref{appendix:further_simulations} (figures~\ref{fig:compare_ghulc} and \ref{fig:ghulc_c}) for an application of GHulC to compute confidence interval of the parameter of interest using univariate quantile regression.

\begin{figure}[!ht]
    \centering
    \includegraphics[width=\textwidth,keepaspectratio]{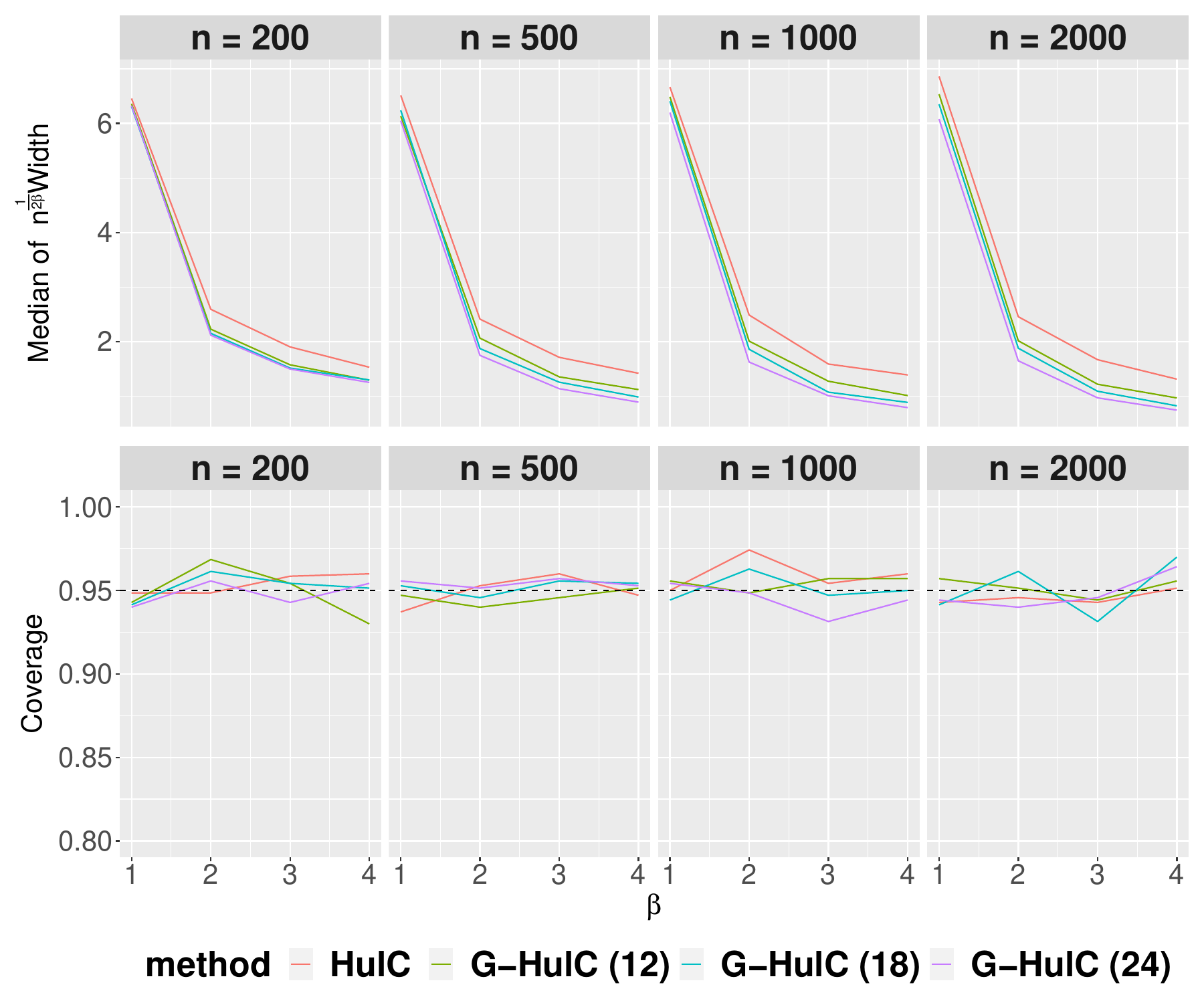}
    \caption{Comparison of the coverage and median of the scaled width ($n^{1/(2\beta)}$Width) of HulC and GHulC (for $B > \log_2(2/\alpha)$) in multivariate quantile regression under non-standard conditions. The sample size is mentioned at the top of each plot and the smoothness parameter of the distribution $\beta$ is on the $x$-axis. The tuning parameter $B$ is mentioned in the parenthesis.}
    \label{fig:compare_ghulc_multi}
\end{figure}

\begin{figure}[!ht]
    \centering
    \includegraphics[width=\textwidth,keepaspectratio]{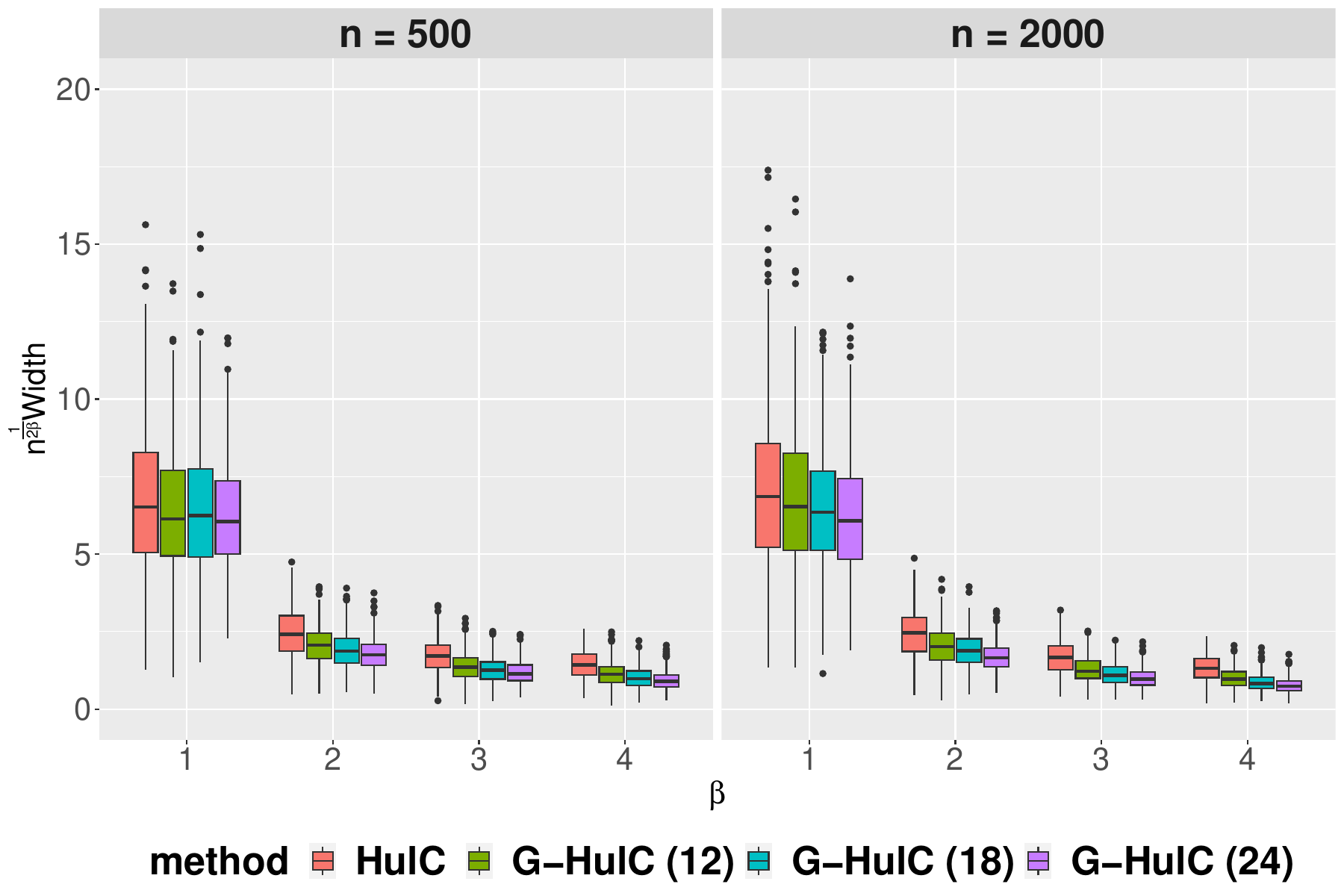}
    \caption{Comparison of box-plots of the scaled width ($n^{1/(2\beta)}$Width) of HulC and GHulC (for $B > \log_2(2/\alpha)$) in multivariate quantile regression under non-standard conditions. The sample size is mentioned at the top of each plot and the smoothness parameter of the distribution $\beta$ is on the $x$-axis. The tuning parameter $B$ is mentioned in the parenthesis.}
    \label{fig:ghulc_c_multi}
\end{figure}

\section{Conclusions and Future Directions}
\label{sec:extension}

In this paper, we study the coverage and width of a distribution-free confidence interval for the median of a distribution. Under the standard assumption that the Lebesgue density at median is bounded away from zero, we show that the width of the distribution-free CI matches that of the Wald interval asymptotically. Under more general assumptions allowing for a zero/infinite density or even a non-existent density, we show that the width when properly scaled converges to a non-degenerate distribution. This is the first ``natural'' example where such a phenomenon is observed. In both standard and non-standard cases, we supplement the asymptotic statements with non-asymptotic analogues. 

Inference for the median is an age-old problem in statistics. In addition to studying the basic properties of the confidence intervals for the median, we show wide ranging implications of our results for inference of arbitrary functionals for which asymptotically median unbiased estimators exist. In particular, we develop a generalization of HulC that provides confidence intervals of improved width.

This work can be extended in several directions. Firstly, it is of interest to know if similar conclusions hold true for distribution-free confidence intervals of other quantiles. Secondly, inference for shift parameter in other location models are also of interest. Inference for shift parameter of a unimodal location family (\cite{edelman1990confidence, paul2025finite}) or a symmetric location family (\cite{lanke1974interval}) are some interesting examples.

\bibliography{references}
\bibliographystyle{plainnat}
\newpage
\appendix
\setcounter{section}{0}
\setcounter{equation}{0}
\setcounter{figure}{0}
\setcounter{remark}{0}
\renewcommand{\thesection}{S.\arabic{section}}
\renewcommand{\theequation}{E.\arabic{equation}}
\renewcommand{\thefigure}{A.\arabic{figure}}
\renewcommand{\theremark}{R.\arabic{remark}}
  \begin{center}
  \Large {\bf Appendix to ``Inference for Median and a Generalization of HulC''}
  \end{center}
\section{Literature Review on Inference for Median}\label{appsec:related-literature}
One of the most commonly used confidence intervals is based on the asymptotic normality of the properly normalized sample median (see Example $5.24$ of \citet[page 54]{vaart_1998} and~\citet[Sec. 2.6]{serfling2009approximation} for details). Formally, suppose $X_1, \ldots, X_n$ are independent and identically distributed random variables with distribution function $F$ such that $F(\theta_0) = 1/2$ and $F'(\theta_0) > 0$. Let $X_{(1)} \le X_{(2)} \le \cdots \le X_{(n)}$ denote the increasing rearrangement (i.e., order statistics) of $X_1, X_2, \ldots, X_n$. The condition $F'(\theta_0) > 0$ implies that $\theta_0$ is a continuity point of $F$ and hence $\theta_0$ is the unique median of $F$. A natural estimator of the median is the sample median $\widehat{\theta}_n$ given by
\[
\widehat{\theta}_n := \begin{cases}
    X_{((n+1)/2)}, &\text{  if $n$ is odd,} \\
    X_{(n/2)}, &\text{ if $n$ is even,} 
\end{cases}
\]
and satisfies
\[
n^{1/2}(\widehat{\theta}_n - \theta_0) \overset{d}{\to} N\left(0, \frac{1}{4(F'(\theta_0))^2}\right).
\]
Using this result, a Wald-type confidence interval can be constructed as 
\begin{equation}
\label{eq:as_norm_cn}
    \widehat{\mathrm{CI}}_{n,\alpha}^{\mathtt{AN}} ~:=~ \left[\widehat{\theta}_n - \frac{\widehat{\sigma}z_{\alpha/2}}{\sqrt{n}},\, \widehat{\theta}_n + \frac{\widehat{\sigma}z_{\alpha/2}}{\sqrt{n}}\right],
\end{equation}
where $\widehat{\sigma}$ is a consistent estimator of the asymptotic standard deviation of the sample median $\widehat{\theta}_n$. Such an estimator can be obtained using kernel density estimator or spacings as described in~\citet[Sec. 6]{sen1966distribution} or using bootstrap as described in~\cite{ghosh1984note}. Jackknife is known to be inconsistent for estimating the asymptotic variance of median~\citep{martin1990using}. 

Another standard method for constructing a confidence interval for $\theta_0$ when no assumptions are given on $F$ is to consider the closed intervals whose end-points are order statistics i.e.\ intervals of the form $[X_{(r)},X_{(s)}]$. One of the first papers that used order statistics to construct distribution-free finite sample confidence intervals for quantiles is \cite{scheffe1945non}. It is well known (see \cite{guilbaud1979interval}, \cite{david2004order}) that for any distribution function $F$,
\begin{equation}\label{eq:coverage-lower-bound-Guilbaud}
    \mathbb{P}(\theta_0 \in [X_{(r)},X_{(s)}]) ~\geq~ \sum_{i=r}^{s-1} \dbinom{n}{i} 2^{-n},
\end{equation}
and the equality holds if and only if $F$ is continuous at the point $\theta_0$. It is generally recommended to take $s=n-r+1$ and then $r$ is chosen so that the required confidence coefficient is achieved, but this choice of $s$ need not necessarily yield the shortest interval. This interval shall be referred to as the $S$-interval in this work. 

Note that the right-hand side of~\eqref{eq:coverage-lower-bound-Guilbaud} can only take a finite set of values when $s, r$ are varied and hence, cannot be set equal to any arbitrary $1-\alpha$. For this reason, alternative distribution-free confidence intervals are studied in the literature. For example, it is possible to construct a confidence interval of median by taking a confidence interval of the form $[Y_{(r)}, Y_{(s)}]$ where $Y_i$'s are the pairwise averages $(X_j+X_k)/2$ where $1 \leq j,k \leq n$. The order statistics $Y_{(r)}$ and $Y_{(s)}$ are chosen in such a way that the required confidence coefficient is attained. This confidence interval originates from the well-known Wilcoxon signed-rank test. For details regarding this work, refer to \citet[Sec. $5.7$]{gibbons2014nonparametric}, \cite{lehmann1963nonparametric}, and \cite{noether1967wilcoxon}. This interval shall be referred to as the $W$-interval in this work.

Another approach in this direction was made by \cite{noether1973some} towards constructing a confidence interval for the center of a {\em symmetric} distribution. He considered a family of confidence intervals, $\{(0.5(X_{(g)}+X_{(n+1-h)}), 0.5(X_{(h)}+X_{(n+1-g)}))| 1\leq g <h, \text{  } g+h\leq n+1\}$, whose end-points are averages of order statistics and then considered the interval with the shortest expected length (referred to as the $(g,h)$-confidence interval). He also compared the asymptotic length of the $(g,h)$-confidence interval with $S$-, $W$-, and the Wald confidence interval for three different distributions namely, normal, logistic, and the double-exponential distribution. He observed that the Wald confidence interval is optimal for normal distributions (i.e.\ it has the shortest asymptotic length), the $W$-interval is optimal for logistic distributions, and the $S$-interval is optimal for double exponential distributions. 

\cite{guilbaud1979interval} considered similar intervals to construct a confidence interval of the median for general (not necessarily symmetric) distributions. He considered intervals of the form $[(X_{(r)}+X_{(r+t)})/2,(X_{(s-t)}+X_{(s)})/2]$ where $1 \leq r \leq s=n-r+1$ and $0 \leq t \leq s-r$. He provided a lower bound, which does not depend on the distribution function $F$, on the probability that $\theta_0$ lies in the aforementioned intervals. Therefore, this result gives an idea about the probability of $\theta_0$ being in the mentioned confidence interval in the general case when no assumptions are made about $F$. Other works in this direction include \cite{guilbaud2006confidence}, where weighted means of adjacent sign intervals for interval estimation of median i.e.\ intervals of the form $[wX_{(r)}+(1-w)X_{(r+1)},(1-w)X_{(s-1)}+wX_{(s)}]$ with $0\leq w \leq 1$ and $1 \leq r < s=n-r+1$ are considered. He obtained the best possible lower bound (which does not depend on the underlying distribution) for the coverage probability of this confidence interval for the general case (without any assumptions on the distributions), for the class of all symmetric distributions, and for the class of all symmetric and unimodal distributions. \cite{nagaraja2020distribution} discusses various distribution-free methods of constructing approximate confidence intervals for quantiles. The main contribution of their work is however the usage of asymmetric spacings of order statistics to construct confidence intervals for quantiles and in particular of median (see Proposition-$5$ of the paper). Their work is an extension of \cite{goh2004smoothing} which uses symmetric spacings to estimate the population median.

Apart from these approaches, one can also obtain a confidence interval of median through bootstrap methods. A popular method of obtaining a confidence interval of the median is through classical bootstrap. \cite{bickel1981some} prove the consistency of the bootstrap under the standard conditions of $F'(\theta_0) > 0$; also see Section $3.4.3$ of \cite{bose2018u} for more details. Another method of interval estimation of the median is by combining information from subsamples as discussed in \cite{knight2002second}. They obtain an estimate of the population median by taking weighted averages of sample medians from non-overlapping subsamples or from balanced overlapping subsamples. Theorem 1 of \cite{knight2002second} provides a result on the distributional convergence of the properly scaled and centered version of this estimate of population median, which we can use to obtain a confidence interval of the required confidence coefficient.

However in all the above works, very little seems to have been done in the direction of analyzing the width of the confidence interval beyond the standard conditions. Under the standard conditions, \citet[Sec. 2.6.3]{serfling2009approximation} analyzes the width of confidence intervals based on order statistics. Except for the methods based on distribution-free lower bounds on coverage, all other methods are not even consistent (i.e., do not have valid coverage) under non-standard conditions. The limiting distribution is non-normal and the rate of convergence of the sample median can be different from $n^{1/2}$ under non-standard conditions; see, for example,~\cite{knight1998bootstrapping,knight1998delta}. In fact,~\cite{Smirnov52} characterized all possible limiting distributions of the sample quantiles; also, see~\cite{knight2002limiting} for a detailed discussion. Moreover, it is well-known that classical Efron's bootstrap is inconsistent under non-standard conditions as shown in~\cite{huang1996bootstrapping} and~\cite{knight1998bootstrapping}. The consistency of $m$-of-$n$ bootstrap with $m/n\to0$ follows from the results of~\cite{huang1996bootstrapping}, but only in the case when the distribution function has a finite non-zero left and right derivatives at $\theta_0$. Under more complicated non-standard conditions, we consider, both $m$-of-$n$ bootstrap and subsampling are not readily applicable because the rate of convergence is usually unknown. It may be worthwhile to mention here that subsampling with an estimated rate of convergence can be applied as suggested in~\cite{bertail1999subsampling}.

It is worth noting that one can also obtain a distribution-free finite-sample valid confidence interval of median using suitable concentration bounds (such as the Hoeffding's inequality). Refer to Proposition~\ref{prop:hoeff} to see such a distribution-free finite sample confidence interval $\widehat{\mathrm{CI}}_{n,\alpha}^{\mathtt{Hoeff}}$ of median based on Hoeffding's inequality and its coverage guarantee. The confidence interval $\widehat{\mathrm{CI}}_{n,\alpha}^{\mathtt{Hoeff}}$ can be conservative as it ignores the exact binomial distribution of $n\widehat{F}_n(\theta_0)$, but serves as a good approximation to the confidence interval based on the exact binomial distribution.

\section{Auxiliary Results}
\label{appendix}
\begin{lemma}
\label{lem:kl_bound}
We have the following upper and lower bound for $H(x,1/2) =  x\ln(2x) + (1-x) \ln(2(1-x))$ for $x \in [0,1]$,
\begin{equation} 
    2\left( x - \frac{1}{2}\right)^2 + \frac{4}{3} \left(x - \frac{1}{2} \right)^4 ~\leq~ H\left(x,\frac{1}{2} \right) ~\leq~ 2\left( x - \frac{1}{2}\right)^2 + (16\ln(2) - 8) \left(x - \frac{1}{2} \right)^4.
\end{equation}
\end{lemma}
\begin{proof}[Proof of \Cref{lem:kl_bound}]
We shall first prove the upper bound for $H(x,1/2)$. Let $f(\cdot)$ be the following function,
\[
f(x) = 2(x - 0.5)^2 + (16\ln(2) - 8)(x - 0.5)^4 - \ln(2) - x\ln(x) - (1-x)\ln(1-x).
\]
Since $f(\cdot)$ is symmetric about $0.5$, it is enough to show that $f(x) \geq 0$ for all $x \geq 0.5$. We note that $f(0.5) = 0$. We compute the derivatives of various order of the function $f(\cdot)$,
\[
\begin{cases}
    f^{(1)}(x) &= 4(x - 0.5) + (64\ln(2)-32)(x-0.5)^3 - (\ln(x) - \ln(1-x)), \\
    f^{(2)}(x) &= 4 + (192\ln(2) - 96)(x - 0.5)^2 - ((1/x)+(1/(1-x))),\\
    f^{(3)}(x) &= (384\ln(2) - 192)(x-0.5) -(-(1/x^2)+(1/(1-x)^2)),\\
    f^{(4)}(x) &= (384\ln(2) - 192) - 2((1/x^3)+(1/(1-x)^3)).
\end{cases}
\]
We observe that $f^{(1)}(0.5)=f^{(2)}(0.5)=f^{(3)}(0.5)=0$ and $f^{(4)}(0.5)=(384\ln(2) - 192)-32>0$. Since $f^{(4)}(x)$ is a continuous function on $(0,1)$ there exists $\epsilon >0$ such that $f^{(4)}(x) > 0$ for all $x \in (0.5,0.5+\epsilon)$. This implies that $f^{(3)}(\cdot)$ is a strictly increasing function on $\left[0.5,0.5+\epsilon \right)$. Hence $f^{(3)}(x)>f^{(3)}(0.5)=0$ for all $x \in (0.5,0.5+\epsilon)$. Repeating the same argument two more times we obtain that $f^{(1)}(x) > 0$ for all $x \in (0.5,0.5+\epsilon)$. However we note that $\lim_{x \xrightarrow{} 1}f^{(1)}(x) = -\infty$. Since $f^{(1)}(x)$ is a continuous function, by intermediate value property we can say that there exists $0.5<y<1$ such that $f^{(1)}(y)=0$. Let $0.5<z<1$ be the first root of $f^{(1)}(x)=0$. $f^{(1)}(\cdot)$ is continuous on the compact interval $[0.5,z]$ and hence attains maximum in this interval. The maximum is not attained at end-points as $f^{(1)}(0.5+\epsilon/2)>f^{(1)}(0.5)=f^{(1)}(z)=0$. Suppose $f^{(1)}(\cdot)$ attains the maximum at $u \in (0.5,z)$. This implies that $f^{(2)}(u) = 0$. Since $f^{(2)}(\cdot)$ is symmetric about $0.5$, we also have $f^{(2)}(1-u)=0$. Moreover $0.5$ is a double root of $f^{(2)}(x)=0$ because of the symmetry of $f^{(2)}(\cdot)$ about $0.5$ and $f^{(2)}(0.5)=0$. Therefore we have got four real roots of $f^{(2)}(x)=0$ viz $0.5,0.5,u,1-u$. We now note that the equation $f^{(2)}(x)=0$ is essentially the following biquadratic equation,
\[
4x(1-x)+ (192\ln(2) - 96)(x - 0.5)^2x(1-x)-1 = 0,
\]
and hence $0.5,0.5,u,1-u$ is the exhaustive set of roots of $f^{(2)}(x)=0$. If there exists $z<z'<1$ such that $f^{(1)}(z')=0$, then proceeding as before we can obtain $u' \in (z,z')$ such that $f^{(2)}(u')=0$. However this is not possible as $f^{(2)}(\cdot)$ can admit at most four real roots. Hence $f^{(1)}(x)=0$ has only two roots in $[0.5,1]$, one at $0.5$ and the other at $z$. We can say the following about $f^{(1)}(\cdot)$,
\[
\begin{cases}
f^{(1)}(x) & \geq 0 \quad \mbox{on}\quad x \in [0.5,z],\\
f^{(1)}(x) &< 0 \quad \mbox{on} \quad x \in \left(z,1\right].
\end{cases}
\]
Therefore $f(\cdot)$ increases from $0.5$ to $z$ and thereafter decreases from $z$ to $1$. For $x \in [0.5,z]$, $f(x) \geq f(0.5) = 0$. For $x \in \left(z,1\right]$, $f(x) \geq f(1) = 0.5 + (1/16)(16\ln(2)-8)-\ln(2)=0$. This proves the upper bound for $H(x,1/2)$.

For proving the lower bound on $H(x,1/2)$, we proceed in the same way as for proving the upper bound. We define $f(\cdot)$ the same way as before with $(16\ln(2)-8)$ replaced by $4/3$,
\[
f(x) = 2(x - 0.5)^2 + (4/3)(x - 0.5)^4 - \ln(2) - x\ln(x) - (1-x)\ln(1-x).
\]
Note that this time we have, $f^{(1)}(0.5)=f^{(2)}(0.5)=f^{(3)}(0.5)=f^{(4)}(0.5)=f^{(5)}(0.5)=0$. We also have $f^{(6)}(0.5)=-24(32+32)<0$. Therefore like before, we can find $\epsilon>0$ such that $f^{(1)}(x)<0$ for all $x \in (0.5,0.5+\epsilon)$. From the previous analysis, we know that $f^{(1)}(\cdot)$ can not have more that one root in $\left(0.5,1\right]$. If $f^{(1)}(x)$ does not have any root in $\left(0.5,1\right]$, we can say that $f^{(1)}(x) \leq 0$ for all $x \in [0.5,1]$ (as otherwise by intermediate value property we can find a root of $f^{(1)}(\cdot)$ in $\left(0.5,1\right]$). Therefore $f(\cdot)$ is a decreasing function in $[0.5,1]$ implying that $f(x) \leq f(0.5) = 0$ for all $x \in [0.5,1]$. If $f^{(1)}(x)$ has a root $z \in \left(0.5,1\right]$, we observe that $f^{(1)}(x)<0$ in both the intervals $(0.5,z)$ and $\left(z,1\right]$. Therefore $f^{(1)}(x) \leq 0$ for all $x \in [0.5,1]$ implying that $f(x) \leq f(0.5) = 0$ for all $x \in [0.5,1]$. This completes the proof of the lemma.  
\end{proof}

\begin{prop}
\label{prop:hoeff}
$X_1,\cdots,X_n$ are independently distributed random variables from distributions $F_1,\cdots, F_n$ (respectively) with median $\theta_0$ i.e.\ $\mathbb P(X_i \geq \theta_0) \geq 1/2$ and $\mathbb P(X_i \leq \theta_0) \geq 1/2$ for all $i \in [n]$. Then we have the following for all sample sizes $n \geq 1$ and for all  distribution functions $\{F_i\}_{i = 1}^n$ with median $\theta_0$, 
\begin{equation}
    \mathbb{P}\left(\theta_0\in\widehat{\mathrm{CI}}_{n,\alpha}^{\mathtt{Hoeff}}\right) \ge 1 - \alpha,
\end{equation}
where,
\begin{equation}\label{eq:DKW-CI-median}
\begin{split}
\widehat{\mathrm{CI}}_{n,\alpha}^{\mathtt{Hoeff}} &:= \left\{\theta\in\mathbb{R}:\,\sum_{i = 1}^n \textbf{1}\{X_i \leq \theta \} \ge \frac{n}{2} -  \sqrt{\frac{n\log(2/\alpha)}{2}} \right\}  \\
& \quad \quad \bigcap \left\{\theta\in\mathbb{R}:\,\sum_{i = 1}^n \textbf{1}\{X_i \geq \theta\} \ge  \frac{n}{2} - \sqrt{\frac{n\log(2/\alpha)}{2}} \right\},\\
\end{split}
\end{equation}
\end{prop}
\begin{proof}[Proof of Proposition~\ref{prop:hoeff}]
Hoeffding's inequality implies the following two bounds for all sample sizes $n \geq 1$ and for any $\lambda > 0$,
\begin{equation*}
    \begin{split}
        \mathbb{P}\left(n^{1/2}\left(\mathbb P(X \leq \theta_0) - \frac{1}{n}\sum_{i = 1}^n \textbf{1}\{X_i \leq \theta_0 \} \right) \ge \lambda\right) \leq & e^{-2\lambda^2}, \\
        \mathbb{P}\left(n^{1/2}\left(\mathbb P(X \geq \theta_0) - \frac{1}{n}\sum_{i = 1}^n \textbf{1}\{X_i \geq \theta_0 \} \right) \ge \lambda\right) \leq & e^{-2\lambda^2}. \\
    \end{split}
\end{equation*}
Setting $\lambda = \sqrt{\log(2/\alpha)/2}$ and using $\mathbb P(X \leq\theta_0)\geq1/2$ and $\mathbb P(X \leq\theta_0)\geq1/2$ we get that,
\begin{equation*}
    \begin{split}
        \mathbb{P}\left(\sum_{i = 1}^n \textbf{1}\{X_i \leq \theta_0 \} \le   (n/2) - \sqrt{n\log(2/\alpha)/2} \right) \leq & \alpha/2, \\
        \mathbb{P}\left(\sum_{i = 1}^n \textbf{1}\{X_i \geq \theta_0 \} \le  (n/2) -\sqrt{n\log(2/\alpha)/2} \right)  \leq & \alpha/2. \\
    \end{split}
\end{equation*}
Combining the above two probability inequalities through union bound we have the following, 
\begin{equation*}
    \begin{split}
        \mathbb{P}\left(\left\{\sum_{i = 1}^n \textbf{1}\{X_i \leq \theta_0 \} \le   (n/2) - \sqrt{n\log(2/\alpha)/2} \right\} \bigcup \left\{ \sum_{i = 1}^n \textbf{1}\{X_i \geq \theta_0 \} \le   (n/2) - \sqrt{n\log(2/\alpha)/2}\right\}\right) \leq \alpha . 
    \end{split}
\end{equation*}
This implies the following upper bound on the miscoverage probability, 
\[
\mathbb{P}\left(\theta_0\notin\widehat{\mathrm{CI}}_{n,\alpha}^{\mathtt{Hoeff}}\right) \leq \alpha.
\]
This completes the proof of the proposition. 
\end{proof}

\begin{prop}
\label{prop:mills_ineq}
We have the following inequality for $0<\alpha<1$,
\begin{equation}
    z_{\alpha/2} \leq \sqrt{2\log(2/\alpha)}.
\end{equation}
\end{prop}
\begin{proof}[Proof of Proposition~\ref{prop:mills_ineq}]
To prove this, we shall first show that, $x(1-\Phi(x)) \leq \phi(x)$. Consider the function, $f(x)=-(\phi(x)/x)$. We observe that,
\begin{equation}
    \frac{df}{dx} = \left( 1+ \frac{1}{x^2}\right)\phi(x).
\end{equation}
We perform the following computation,
\begin{equation}
    \begin{split}
        1-\Phi(x) &= \int_{t=x}^\infty \phi(x) dx\\
        & \leq \int_{t=x}^\infty \left( 1+ \frac{1}{x^2}\right)\phi(x) dx \\
        & = \int_{t=x}^\infty \frac{df}{dx} dx \\
        &= \lim_{t \xrightarrow{} 0}f(t) - f(x) \\
        &= -f(x) \\
        &= \frac{\phi(x)}{x}.
    \end{split}
\end{equation}
Thus we have shown that,
\begin{equation}
    1-\Phi(x) \leq \frac{\phi(x)}{x}.
\end{equation}
Substituting $\sqrt{2\log(2/\alpha)}$ in place of $x$ in the above equation, we have,
\begin{equation}
    \begin{split}
        1-\Phi(\sqrt{2\log(2/\alpha)}) &\leq \frac{\phi(\sqrt{2\log(2/\alpha)})}{\sqrt{2\log(2/\alpha)}} \\
        &= \frac{1}{2\sqrt{\pi\log(2/\alpha)}}\exp (-\log (2/\alpha)) \\
        &< \alpha/2 \\
        &= 1-\Phi(z_{\alpha/2}).
    \end{split}
\end{equation}
The above deduction implies that $\Phi(\sqrt{2\log(2/\alpha)}) \geq \Phi(z_{\alpha/2})$ and since $\Phi(.)$ is a monotonically increasing function, we can say that $ z_{\alpha/2} \leq \sqrt{2\log(2/\alpha)}$. This completes the proof. 
\end{proof}

\begin{prop}
\label{prop:holder}
Let $\rho\ge 1$ and define
\[
h(x):=|x|^{1/\rho}\,\mbox{sgn}(x), \qquad x\in\mathbb{R}.
\]
Then for all $x,y\in\mathbb{R}$,
\[
|h(x)-h(y)|
\;\le\;
2^{\,1-\frac{1}{\rho}}\,
|x-y|^{\frac{1}{\rho}}.
\]
Equivalently, $h$ is H\"older continuous of order $1/\rho$ with H\"older constant
$2^{\,1-\frac{1}{\rho}}$.
\end{prop}

\begin{proof}[Proof of Proposition~\ref{prop:holder}]
Set $p:=1/\rho\in(0,1]$. We show that for all $x,y\in\mathbb{R}$,
\[
|\,|x|^{p}\mbox{sgn}(x)-|y|^{p}\mbox{sgn}(y)\,|
\le 2^{1-p}|x-y|^{p}.
\]
We split into two cases.

\medskip
\noindent\textit{Case} 1: $xy\ge 0$ (same sign, or one is zero).
Then $\mbox{sgn}(x)=\mbox{sgn}(y)$ (possibly $0$), and therefore
\[
|h(x)-h(y)|=\big||x|^{p}-|y|^{p}\big|.
\]
For $p\in(0,1]$, the function $t\mapsto t^{p}$ is concave on $[0,\infty)$, hence
subadditive: for all $u,v\ge 0$,
\begin{equation}\label{eq:subadd}
(u+v)^{p}\le u^{p}+v^{p}.
\end{equation}
Assume without loss of generality that $|x|\ge |y|$. Applying
\eqref{eq:subadd} with $u=|x|-|y|$ and $v=|y|$ gives
\[
|x|^{p} = \big((|x|-|y|)+|y|\big)^{p}\le (|x|-|y|)^{p}+|y|^{p},
\]
so $\big||x|^{p}-|y|^{p}\big|\le (|x|-|y|)^{p}$. Using
$||x|-|y||\le |x-y|$ yields
\[
|h(x)-h(y)|
\le
\big||x|-|y|\big|^{p}
\le
|x-y|^{p}
\le
2^{1-p}|x-y|^{p},
\]
since $2^{1-p}\ge 1$.

\medskip
\noindent\textit{Case} 2: $xy<0$ (opposite signs).
Without loss of generality, let $x\ge 0$ and $y\le 0$. Then
$h(x)=x^{p}$ and $h(y)=-|y|^{p}$, hence
\[
|h(x)-h(y)|=x^{p}+|y|^{p}.
\]
Again by concavity of $t\mapsto t^{p}$ on $[0,\infty)$, Jensen's inequality gives
\[
\frac{x^{p}+|y|^{p}}{2}
\le
\left(\frac{x+|y|}{2}\right)^{p}.
\]
Multiplying both sides by $2$ yields
\[
x^{p}+|y|^{p}\le 2^{1-p}(x+|y|)^{p}.
\]
Since $|x-y|=x+|y|$ in this case, we conclude that
\[
|h(x)-h(y)|\le 2^{1-p}|x-y|^{p}.
\]

\medskip
Combining the two cases completes the proof.
\end{proof}

\begin{theorem}
    \label{thm:quantile_coverage}
Suppose $X_1,X_2,\cdots,X_n \stackrel{iid}{\sim} F$ and let $\theta_h$ be the $(0.5+h)$-th population quantile of the distribution $F$ i.e.\ if $X \sim F$ then $\mathbb{P}(X \leq \theta_h) \geq 0.5 +h$ and $\mathbb{P}(X \geq \theta_h) \geq 0.5-h$. If $\widehat{\mathrm{CI}}_{n,\alpha}$ is the confidence interval returned by \Cref{alg:proposed-conf-int} then
\[
\inf_{h\in[-\varepsilon, \varepsilon]}\,\mathbb{P}(\theta_h \in \widehat{\mathrm{CI}}_{n,\alpha})  ~\geq~ 1- \alpha - 2\alpha n(n-1)(1+2\varepsilon)^{n-2}\varepsilon^2.
\]

\end{theorem}

\begin{proof}[Proof of \Cref{thm:quantile_coverage}]
We have $X_1,X_2,\cdots,X_n \stackrel{iid}{\sim} F$. We recall the following from \Cref{alg:proposed-conf-int},
    \[
    \begin{cases}
      \widehat{\mathrm{CI}}_{1,n,\alpha} := \left\{\theta\in\mathbb{R}:\, \sum_{i=1}^n \mathbf{1}\{X_i \le \theta\} \ge \left\lfloor \frac{n}{2}\right\rfloor - c_{n,\alpha}\right\}, \\
    \widehat{\mathrm{CI}}_{2,n,\alpha} := \left\{\theta\in\mathbb{R}:\, \sum_{i=1}^n \mathbf{1}\{X_i \ge \theta\} \ge \left\lfloor \frac{n}{2}\right\rfloor - c_{n,\alpha}\right\}.     
    \end{cases}
    \]
Moreover the $100(1-\alpha)\%$ confidence interval that we work with is $\widehat{\mathrm{CI}}_{n,\alpha}=\widehat{\mathrm{CI}}_{1,n,\alpha} \cap \widehat{\mathrm{CI}}_{2,n,\alpha}$. We want to compute the probability that our confidence interval for the median contains $\theta_h$. We observe the following, 
\begin{equation*}
\begin{split}
    \mathbb{P}(\theta_h \in \widehat{\mathrm{CI}}_{1,n,\alpha}) &= \mathbb{P}(\sum_{i=1}^n \mathbf{1}\{X_i \le \theta\} \ge \left\lfloor n/2 \right\rfloor - c_{n,\alpha})    \\
    &\geq \mathbb{P}(W_n \geq \left\lfloor n/2 \right\rfloor - c_{n,\alpha}) \quad \mathrm{where} \quad W_n \sim \mathrm{Bin}(n,0.5+h).
\end{split}
\end{equation*}
Similarly we have, 
\begin{equation*}
\begin{split}
    \mathbb{P}(\theta_h \in \widehat{\mathrm{CI}}_{2,n,\alpha}) &= \mathbb{P}(\sum_{i=1}^n \mathbf{1}\{X_i \ge \theta\} \ge \left\lfloor n/2 \right\rfloor - c_{n,\alpha})    \\
    &\geq \mathbb{P}(Z_n \geq \left\lfloor n/2 \right\rfloor - c_{n,\alpha}) \quad \mbox{where $Z_n \sim \mbox{Bin}(n,0.5-h)$}.
\end{split}
\end{equation*}
Suppose $Y_n \sim \mbox{Bin}(n,0.5)$. We consider the following functions $\{f_k(\cdot)\}_{k=0}^n$,
\[
f_k(h) = \frac{\mathbb{P}(W_n=k)}{\mathbb{P}(Y_n=k)} = (1+2h)^k(1-2h)^{n-k} \quad \mbox{for $h \in (-0.5,0.5)$ and $k \in \{0,1,\cdots,n\}$}. 
\]
We can check that $f_k(0)=1$ and $f_k'(0)=2(2k-n)$ (the derivative is taken w.r.t.\ $h$) for $k=0,\cdots,n$. We also have the following bound on $f_k''(h)$ for $h \in [-\varepsilon,\varepsilon]$ ,
\begin{equation*}
    \begin{split}
        f_k''(h) &= 4k(k-1)(1+2h)^{k-2}(1-2h)^{n-k} - 8k(n-k)(1+2h)^{k-1}(1-2h)^{n-k-1} \\
        &\quad+4(n-k)(n-k-1)(1+2h)^k(1-2h)^{n-k-2} \\
        & \leq 4(1+2\varepsilon)^{n-2}\{k(k-1) -2k(n-k) +(n-k)(n-k-1) \} \\
        &= 4(1+2\varepsilon)^{n-2}\{(2k-n)^2 -n \}.
    \end{split}
\end{equation*}
Therefore if $k < \left\lfloor n/2 \right\rfloor - c_{n,\alpha}$ then $f_k''(h) \leq u_{n,\varepsilon}$ for $h \in (-\varepsilon,\varepsilon)$ where $u_{n,\varepsilon} = 4n(n-1)(1+2\varepsilon)^{n-2}$. For $h\in [-\varepsilon,\varepsilon]$ we have the following bounds on $\mathbb{P}(W_n < \left\lfloor n/2 \right\rfloor - c_{n,\alpha})$ and $\mathbb{P}(Z_n < \left\lfloor n/2 \right\rfloor - c_{n,\alpha})$,
\begin{equation*}
\begin{split}
    \mathbb{P}(W_n < \left\lfloor n/2 \right\rfloor - c_{n,\alpha}) &= \sum_{k<\left\lfloor n/2 \right\rfloor - c_{n,\alpha}} \mathbb{P}(W_n=k) \\
    &= \sum_{k < \left\lfloor n/2 \right\rfloor - c_{n,\alpha}} f_k(h) \mathbb{P}(Y_n =k) \\
    &\leq \sum_{k < \left\lfloor n/2 \right\rfloor - c_{n,\alpha}} [f_k(0) +f_k'(0)h + (1/2)u_{n,\varepsilon}\varepsilon^2]\mathbb{P}(Y_n =k) \\
    &= \mathbb{P}(Y_n < \left\lfloor n/2 \right\rfloor - c_{n,\alpha}) + \sum_{k<\left\lfloor n/2 \right\rfloor - c_{n,\alpha}}2(2k-n)h\mathbb{P}(Y_n=k) \\
    & \quad +\frac{u_{n,\varepsilon}}{2}\varepsilon^2 \mathbb{P}(Y_n < \left\lfloor n/2 \right\rfloor - c_{n,\alpha})\\
    & \leq \frac{\alpha}{2} + \sum_{k<\left\lfloor n/2 \right\rfloor - c_{n,\alpha}}2(2k-n)h\mathbb{P}(Y_n=k) + \frac{u_{n,\varepsilon} \alpha}{4}\varepsilon^2.
\end{split}
\end{equation*}
Similarly we have, 
\begin{equation*}
    \begin{split}
        \mathbb{P}(Z_n < \left\lfloor n/2 \right\rfloor - c_{n,\alpha}) &= \sum_{k<\left\lfloor n/2 \right\rfloor - c_{n,\alpha}} \mathbb{P}(Z_n=k) \\
    &= \sum_{k <\left\lfloor n/2 \right\rfloor - c_{n,\alpha}} f_k(-h) \mathbb{P}(Y_n =k) \\
    &\leq \frac{\alpha}{2} - \sum_{k<\left\lfloor n/2 \right\rfloor - c_{n,\alpha}}2(2k-n)h\mathbb{P}(Y_n=k) + \frac{u_{n,\varepsilon} \alpha}{4}\varepsilon^2.
    \end{split}
\end{equation*}
Combining the above derivations we have the following for $h \in [-\varepsilon,\varepsilon]$,
\begin{equation*}
\begin{split}
    \mathbb{P}(\theta_h \notin \widehat{\mathrm{CI}}_{n,\alpha}) &= \mathbb{P}(\theta_h \in (\widehat{\mathrm{CI}}_{1,n,\alpha} \cap \widehat{\mathrm{CI}}_{2,n,\alpha})^c ) \\
    &= \mathbb{P}(\theta_h \in \widehat{\mathrm{CI}}_{1,n,\alpha}^c \cup \widehat{\mathrm{CI}}_{2,n,\alpha}^c )  \\
    &\leq \mathbb{P}(\theta_h \in \widehat{\mathrm{CI}}_{1,n,\alpha}^c) + \mathbb{P}(\theta_h \in \widehat{\mathrm{CI}}_{2,n,\alpha}^c) \\
    &\leq \mathbb{P}(W_n < \left\lfloor n/2 \right\rfloor - c_{n,\alpha}) + \mathbb{P}(Z_n < \left\lfloor n/2 \right\rfloor - c_{n,\alpha}) \\
    &\leq \alpha + \frac{u_{n,\varepsilon} \alpha}{2}\varepsilon^2. 
\end{split}
\end{equation*}
Thus for $h \in [-\varepsilon,\varepsilon]$ we can say that $\mathbb{P}(\theta_h \in \widehat{\mathrm{CI}}_{n,\alpha}) \geq 1- \alpha -(u_{n,\varepsilon}\alpha/2)\varepsilon^2$. We note that for fixed sample size $n$ as $\varepsilon \xrightarrow{} 0$, $u_{n,\varepsilon}\xrightarrow{} 4n(n-1)$. Thus the coverage probability of $\theta_h$, $\mathbb{P}(\theta_h \in \widehat{\mathrm{CI}}_{n,\alpha}) \geq 1-\alpha + o(h^2)$ as $h \xrightarrow{} 0$. 
\end{proof}

\begin{lemma}
\label{lem:conc_uniform}
If $U_1,\cdots,U_n$ are i.i.d.\ random variables from $\text{Uniform}(0,1)$ distribution, we have the following concentration inequality for the $k$th order statistic $U_{k:n}$,
\begin{equation}\label{eq:best-Dumbgun-version}
\mathbb{P}\left(\left|U_{k:n} - \frac{k}{n + 1}\right| \le \sqrt{\frac{2k}{n + 1}\left(1 - \frac{k}{n+1}\right)\frac{2\log(n)}{n + 1}} + \left|1 - \frac{2k}{n + 1}\right|\frac{2\log(n)}{n + 1}\right) \ge 1 - 2n^{-2}.
\end{equation}
The following simplified inequalities suffice, respectively, for $k \ll (n + 1)/2$ and $k\approx (n+1)/2$. But they are, nonetheless, valid for all $k\in\{1, 2, \ldots, n\}$.
\begin{equation}
\begin{split}
\mathbb{P}\left(\left|U_{k:n} - \frac{k}{n + 1}\right| \le \sqrt{\frac{4k\log(n)}{(n + 1)^2}} + \frac{2\log(n)}{n + 1}\right) &\ge 1 - 2n^{-2},\\
\mathbb{P}\left(\left|U_{k:n} - \frac{k}{n+1}\right| \le \sqrt{\frac{\log n}{n + 1}} + \left|1 - \frac{2k}{n + 1}\right|\frac{2\log(n)}{n+1}\right) &\ge 1 - 2 n^{-2}.
\end{split}
\end{equation}
\end{lemma}

\begin{proof}[Proof of \Cref{lem:conc_uniform}]
It is well-known that $U_{k:n} \sim \mbox{Beta}(k, n-k+1)$. Here $U_{k:n}$ represents the $k$-th smallest among $n$ iid uniform random variables $U_1, \ldots, U_n$. Proposition 2.1 of \cite{dumbgen1998new} implies that
\begin{equation}\label{eq:tail-bounds-KL}
\begin{split}
\mathbb{P}(U_{k:n} \ge x) ~&\le~ \exp(-(n+1)\Psi(x, k/(n+1))), \quad\mbox{if }x \ge k/(n + 1),\\
\mathbb{P}(U_{k:n} \le x) ~&\le~ \exp(-(n+1)\Psi(x, k/(n+1))), \quad\mbox{if }x \le k/(n + 1),
\end{split}
\end{equation}
where
\[
\Psi(x, p) = p\log\left(\frac{p}{x}\right) + (1 - p)\log\left(\frac{1-p}{1-x}\right).
\]
Note that 
\[
\frac{\partial}{\partial x}\Psi(x, p) = -\frac{p}{x} + \frac{1 - p}{1 - x} = \frac{x - p}{x(1- x)}.
\]
This implies that $x\mapsto \Psi(x, p)$ is a decreasing function for $x \le p$ and an increasing function for $x > p$.

Consider now the event
\[
\left\{\Psi(U_{k:n}, k/(n + 1)) \ge \eta, U_{k:n} \ge k/(n+1)\right\}.
\]
The function $x\mapsto \Psi(x, k/(n+1))$ is increasing on $[k/(n+1), 1]$ and in particular, there is a unique solution $x_{\eta} \ge k/(n+1)$ to the equation $\Psi(x, k/(n+1)) = \eta$. Therefore, 
\[
U_{k:n} \ge \frac{k}{n+1}\mbox{ and }\Psi(U_{k:n}, k/(n+1)) \ge \eta\quad\Rightarrow\quad U_{k:n} \ge x_{\eta}.
\]
Hence,
\begin{equation}
    \label{eq:larger-than-k/n}
    \begin{split}
        \mathbb{P}(\Psi(U_{k:n}, k/(n + 1)) \ge \eta, U_{k:n} \ge k/(n+1)) &\le \mathbb{P}(U_{k:n} \ge x_{\eta})\\ 
        &\le \exp(-(n+1)\Psi(x_{\eta}, k/(n+1))) \\
        &= \exp(-(n+1)\eta).
    \end{split}
\end{equation}
Here the second inequality follows from~\eqref{eq:tail-bounds-KL} because $x_{\eta} \ge k/(n+1)$. A similar argument corresponding to the event $\{\Psi(U_{k:n}, k/(n+1)) \ge \eta, U_{k:n} \le k/(n+1)\}$ implies that
\begin{equation}
    \label{eq:less-than-k/n}
    \mathbb{P}(\Psi(U_{k:n}, k/(n + 1)) \ge \eta, U_{k:n} \le k/(n+1)) \le \exp(-(n+1)\eta). 
\end{equation}
Combining~\eqref{eq:larger-than-k/n} and~\eqref{eq:less-than-k/n}, we get that
\[
\mathbb{P}(\Psi(U_{k:n}, k/(n+1)) \ge \eta) \le 2\exp(-(n+1)\eta)\quad\mbox{for all}\quad \eta > 0.
\]
Replacing $\eta$ with $\eta/(n+1)$, we get
\begin{equation}\label{eq:final-KL-bound-Order-statistics}
\mathbb{P}((n+1)\Psi(U_{k:n}, k/(n+1)) \ge \eta) \le 2\exp(-\eta)\quad\mbox{for any}\quad \eta > 0.
\end{equation}
Therefore, with probability at least $1 - 2n^{-2}$, 
\[
\Psi(U_{k:n}, k/(n + 1)) \le \frac{2\log(n)}{n + 1}.
\]
Proposition 2.1 of \cite{dumbgen1998new} again implies then that with probability at least $1 - 2n^{-2},$
\begin{align*}
-&\sqrt{\frac{2k}{n + 1}\left(1 - \frac{k}{n+1}\right)\frac{2\log(n)}{n + 1}} - \left(1 - \frac{2k}{n+1}\right)_-\frac{2\log(n)}{n + 1}\\
&\le 
U_{k:n} - \frac{k}{n + 1}\\ 
&\le \sqrt{\frac{2k}{n + 1}\left(1 - \frac{k}{n+1}\right)\frac{2\log(n)}{n + 1}} + \left(1 - \frac{2k}{n+1}\right)_+\frac{2\log(n)}{n + 1},
\end{align*}
where $(x)_+ = \max\{0, x\}$. This can be succinctly written as
\begin{equation}
\mathbb{P}\left(\left|U_{k:n} - \frac{k}{n + 1}\right| \le \sqrt{\frac{2k}{n + 1}\left(1 - \frac{k}{n+1}\right)\frac{2\log(n)}{n + 1}} + \left|1 - \frac{2k}{n + 1}\right|\frac{2\log(n)}{n + 1}\right) \ge 1 - 2n^{-2}.
\end{equation}
For $k\ll (n+1)/2$, we can use $(1 - k/(n+1)) \in [0, 1]$ and $|1 - 2k/(n+1)| \le 1$, and further simplify this to write
\[
\mathbb{P}\left(\left|U_{k:n} - \frac{k}{n + 1}\right| \le \sqrt{\frac{4k\log(n)}{(n + 1)^2}} + \frac{2\log(n)}{n + 1}\right) \ge 1 - 2n^{-2}.
\]
This completes the proof of the lemma. 
\end{proof}

\begin{lemma}
\label{lem:width_distr}
Let $X_1,X_2,...,X_n \stackrel{iid}{\sim} F $. Suppose that $F$ is a continuous CDF Then for every sample size $n\geq \log_2(2/\alpha)$, with probability at least $1 - 2n^{-2}$,
\[
 \left|(F(X_{(\ceil{{n}/{2}}+c_{n,\alpha}+1)})-F(X_{(\floor{{n}/{2}}-c_{n,\alpha})}))-\frac{z_{\alpha/2}}{\sqrt{n}} \right| ~\leq~ \frac{5.18+0.25z_{\alpha/2}^2+2\log n + 2(4 + \sqrt{n}z_{\alpha/2})^{1/2}\sqrt{\log n}}{n+1}.
\]    
\end{lemma}

{
\color{blue}

}
\begin{proof}[Proof of \Cref{lem:width_distr}]
We shall use \Cref{lem:conc_uniform} to perform finite sample analysis of $F(X_{(\ceil{{n}/{2}}+c_{n,\alpha}+1 )})-F(X_{(\floor{{n}/{2}}-c_{n,\alpha})})$ where $F$ is assumed to be a continuous CDF with median $\theta_0$. Under the assumption that $F$ is a continuous CDF, $Y_i=F(X_i)$ are i.i.d.\ from $U \stackrel{d}{=} U(0,1)$ for $i=1,2,\cdots,n$. Note that, we have $Y_{(\ceil{{n}/{2}}+c_{n,\alpha}+1)}=F(X_{(\ceil{{n}/{2}}+c_{n,\alpha}+1 )})$ and $Y_{(\floor{{n}/{2}}-c_{n,\alpha})}=F(X_{(\floor{{n}/{2}}-c_{n,\alpha})})$. Theorem 1.6.7 of~\cite{reiss2012approximate} implies that the spacing $Y_{(r)}-Y_{(s)}$ has the same distribution as $Y_{(r-s)}$. Thus we obtain that
\[
Y_{(\ceil{{n}/{2}}+c_{n,\alpha}+1)}-Y_{(\floor{{n}/{2}}-c_{n,\alpha})} ~\overset{d}{=}~ 
\begin{cases}
Y_{(2c_{n,\alpha}+1)},  &\mbox{if }n\mbox{ is even},\\
Y_{(2c_{n,\alpha} + 2)}, &\mbox{if }n\mbox{ is odd.}
\end{cases}
\]
We now apply the concentration inequality with $k = 2c_{n,\alpha}+1$ and $k = 2c_{n,\alpha} + 2$ for $n$ even and $n$ odd, respectively. For notational convenience, set $k_{n,\alpha} = 2c_{n,\alpha} +1+ \mathbbm{1}\{n\mbox{ odd}\}$. We thus obtain that,
\begin{equation}
\label{eq:modified_kl}
    \mathbb{P}\left(\left|Y_{(\ceil{{n}/{2}}+c_{n,\alpha}+1)}-Y_{(\floor{{n}/{2}}-c_{n,\alpha})}-\frac{k_{n,\alpha}}{n+1}\right| \le \sqrt{\frac{4k_{n,\alpha}\log n}{(n+1)^2}}+ \frac{2\log n}{n+1}\right) \geq 1-2n^{-2}.
\end{equation}
\Cref{thm:assym_of_cn} states that,
\begin{equation*}
  -\frac{1}{2\sqrt{n}}\max\left\{\frac{z_{\alpha/2}^3}{5},\,\sqrt{1 + \frac{2\ln(2) - 1}{n^2}}\right\} - 1.5 \le c_{n,\alpha} - \frac{\sqrt{n}z_{\alpha/2}}{2} \le 1.   
\end{equation*}
We observe that,
\begin{equation*}
    \begin{split}
  -\frac{1}{2\sqrt{n}}\max\left\{\frac{z_{\alpha/2}^3}{5},\,\sqrt{1 + \frac{2\ln(2) - 1}{n^2}}\right\} - 1.5 &=- \frac{z_{\alpha/2}^3}{10\sqrt{n}}  -\frac{1}{2\sqrt{n}}\sqrt{1 + \frac{2\ln(2) - 1}{n^2}} -1.5 \\
  &\geq - \frac{z_{\alpha/2}^3}{10\sqrt{n}} - \sqrt{\frac{\ln(2)}{2}}-1.5 \\
  &\geq - \frac{z_{\alpha/2}^3}{10\sqrt{n}} - 2.1.
    \end{split}
\end{equation*}
Hence we can say that,
\[
- \frac{z_{\alpha/2}^3}{10\sqrt{n}} - 2.1\le c_{n,\alpha} - \frac{\sqrt{n}z_{\alpha/2}}{2} \le 1.
\]
Note that we are concerned with $n \geq \log_2(2/\alpha) = \log(2/\alpha)/\log2$. Hence $z_{\alpha/2}\leq \sqrt{2\log(2/\alpha)}\leq \sqrt{2\log(2) n}$. Using this inequality we have, $z_{\alpha/2}^3/(5\sqrt{n}) \leq z_{\alpha/2}^2/4$. Using this and the above simplified bounds obtained from \Cref{thm:assym_of_cn}, we have,
\[
\left|\frac{k_{n,\alpha}-\sqrt{n}z_{\alpha/2}}{n+1} \right| \leq \frac{0.25z_{\alpha/2}^2+4}{n+1}.
\]
We also have,
\begin{equation}
    \begin{split}
        \left|\frac{\sqrt{n}z_{\alpha/2}}{n+1} - \frac{z_{\alpha/2}}{\sqrt{n}} \right| &= \left|\frac{z_{\alpha/2}}{\sqrt{n}(n+1)} \right| \leq \left|\frac{\sqrt{2\log(2)}}{n+1} \right| \leq \frac{1.18}{n+1}.
    \end{split}
\end{equation}
Combining the above three inequalities we obtain that with probability greater than or equal to $1-2n^{-2}$ the following inequality holds true,
\begin{equation}
    \begin{split}
        \left|(Y_{(\ceil{{n}/{2}}+c_{n,\alpha}+1)}-Y_{(\floor{{n}/{2}}-c_{n,\alpha})})-(z_{\alpha/2}/\sqrt{n}) \right| &\leq \frac{0.25z_{\alpha/2}^2+5.18}{n+1} + \sqrt{\frac{4k_{n,\alpha}\log n}{(n+1)^2}}+ \frac{2\log n}{n+1} \\
        &\leq \frac{0.25z_{\alpha/2}^2+5.18}{n+1} + \sqrt{\frac{4(4+\sqrt{n}z_{\alpha/2})\log n}{(n+1)^2}}+ \frac{2\log n}{n+1}.
    \end{split}
\end{equation}
Thus, we have shown that for $n \geq \log_2(2/\alpha)$, the following event occurs with probability greater than or equal to $1-2n^{-2}$,
\[
\left|(F(X_{(\ceil{{n}/{2}}+c_{n,\alpha}+1)})-F(X_{(\floor{{n}/{2}}-c_{n,\alpha})}))-(z_{\alpha/2}/\sqrt{n}) \right| \leq \frac{0.25z_{\alpha/2}^2+5.18}{n+1} + \sqrt{\frac{4(4+\sqrt{n}z_{\alpha/2})\log n}{(n+1)^2}}+ \frac{2\log n}{n+1}.
\]
This completes the proof of the lemma. Note that we can further simplify the bounds as follows,
\begin{equation*}
    \begin{split}
    & \frac{0.25z_{\alpha/2}^2+5.18}{n+1} + \sqrt{\frac{4(4+\sqrt{n}z_{\alpha/2})\log n}{(n+1)^2}}+ \frac{2\log n}{n+1} \\
    \leq &  \frac{0.25z_{\alpha/2}^2+5.18}{n+1} + \frac{4\sqrt{\log n}}{n+1} + \frac{2\sqrt{\sqrt{n}z_{\alpha/2}\log n}}{n+1}+\frac{2\log n}{n+1} \\
    \leq & \frac{0.25z_{\alpha/2}^2+5.18}{n+1} + \frac{6+2\sqrt{z_{\alpha/2}}}{(n+1)^{3/4}}(\log(n)+1) \\
    \leq & \frac{0.5\log(2/\alpha)+5.18}{n+1} + \frac{6+2(2\log(2/\alpha))^{1/4}}{(n+1)^{3/4}}(\log(n)+1).    
    \end{split}
\end{equation*}
Therefore we can say that for $n \geq \log_2(2/\alpha)$, the following event occurs with probability greater than or equal to $1-2n^{-2}$,
\begin{equation*}
    \begin{split}
   &  \left|(F(X_{(\ceil{{n}/{2}}+c_{n,\alpha}+1)})-F(X_{(\floor{{n}/{2}}-c_{n,\alpha})}))-(z_{\alpha/2}/\sqrt{n}) \right|  \\
   \leq & \frac{0.5\log(2/\alpha)+5.18}{n+1} + \frac{6+2(2\log(2/\alpha))^{1/4}}{(n+1)^{3/4}}(\log(n)+1),    
    \end{split}
\end{equation*}
which can be also written as,
\begin{equation*}
    \begin{split}
& \left|\sqrt{n}(F(X_{(\ceil{{n}/{2}}+c_{n,\alpha}+1)})-F(X_{(\floor{{n}/{2}}-c_{n,\alpha})}))-z_{\alpha/2} \right| \\
& \leq \frac{0.5\log(2/\alpha)+5.18}{\sqrt{n}} + \frac{6+2(2\log(2/\alpha))^{1/4}}{n^{1/4}}(\log(n)+1) .   
    \end{split}
\end{equation*}
\end{proof}

\begin{lemma}
\label{lem:conc_anbn}
Let $X_1,X_2,\cdots,X_n \stackrel{iid}{\sim} F $. Suppose that $F$ is a continuous CDF with median $\theta_0$. Then for every sample size $n\geq \log_2(2/\alpha)$, we have the following, 
\[
\begin{cases}
    \mathbb{P}\left(\left|F(X_{(k_{n,\alpha}+1)})-F(\theta_0) \right| \leq A_n \right) &\geq 1- 2n^{-2},\\
    \mathbb{P}\left(\left|F(X_{(n-k_{n,\alpha})})-F(\theta_0) \right| \leq B_n \right) &\geq 1- 2n^{-2},
\end{cases}
\]
where $k_{n,\alpha}=G_n^{-1}(1-(\alpha/2))$ and $A_n,B_n$ are defined as follows,
\begin{equation}
\begin{split}
    A_n &= \frac{z_{\alpha/2}}{2\sqrt{n}}+\frac{2}{n}+\sqrt{\frac{\log (n)}{n+1}} + \left(\frac{z_{\alpha/2}}{\sqrt{n}} + \frac{4}{n}\right)\frac{2\log (n)}{n+1},\\
    B_n &= \frac{z_{\alpha/2}}{2\sqrt{n}}+\frac{3}{2n}+\sqrt{\frac{\log (n)}{n+1}} + \left(\frac{z_{\alpha/2}}{\sqrt{n}} + \frac{3}{n}\right)\frac{2\log (n)}{n+1}.
\end{split}
\end{equation}
\end{lemma}

\begin{proof}[Proof of \Cref{lem:conc_anbn}]
Using \Cref{lem:conc_uniform} with $k=k_{n,\alpha}+1$, we have the following inequality with probability greater than or equal to $1-2n^{-2}$ for all sample sizes $n \geq \log_2(2/\alpha)$,
\begin{equation*}
    \begin{split}
    \left| F(X_{(k_{n,\alpha}+1)}) - \frac{k_{n,\alpha}+1}{n+1}\right| &\leq \sqrt{\frac{2(k_{n,\alpha}+1)}{n+1}\left( 1- \frac{k_{n,\alpha}+1}{n+1}\right)\frac{2\log(n)}{n+1}}+\left|1-\frac{2(k_{n,\alpha}+1)}{n+1} \right|\frac{2\log(n)}{n+1} \\
     & \leq \sqrt{\frac{1}{4}\frac{4\log(n)}{n+1}}+\left(\frac{2(k_{n,\alpha}+1)}{n+1} -1\right)\frac{2\log(n)}{n+1}\\
    &\leq \sqrt{\frac{\log(n)}{n+1}} +\left(\frac{2((n/2)+(\sqrt{n}z_{\alpha/2}/2)+2)}{n+1} -1\right)\frac{2\log(n)}{n+1} \\
    &\leq \sqrt{\frac{\log (n)}{n+1}} + \left(\frac{z_{\alpha/2}}{\sqrt{n}} + \frac{4}{n}\right)\frac{2\log (n)}{n+1}.
    \end{split}
\end{equation*}
The above follows from the application of the lower and upper bound on $k_{n,\alpha}$ mentioned in \eqref{eq:hd_new_bound}. In the above computation we also used the fact that $2(k_{n,\alpha}+1)\geq n+1$. This holds because $k_{n,\alpha}=G_n^{-1}(1-(\alpha/2))\geq (n-1)/2$ which implies that $2(k_{n,\alpha}+1)\geq n+1$. By another application of \eqref{eq:hd_new_bound} we obtain,
\begin{equation*}
    \begin{split}
        \left|\frac{k_{n,\alpha}+1}{n+1} - \frac{1}{2} \right| &= \left( \frac{k_{n,\alpha}+1}{n+1} - \frac{1}{2}\right) \\
        &\leq \left(\frac{(n/2)+(\sqrt{n}z_{\alpha/2}/2)+2}{n} - \frac{1}{2}\right)\\
        &= \frac{z_{\alpha/2}}{2\sqrt{n}}+\frac{2}{n}.
    \end{split}
\end{equation*}
Combining the above two inequalities using triangle inequality we obtain that,
\[
\mathbb{P}\left(\left|F(X_{(k_{n,\alpha}+1)})-F(\theta_0) \right| \leq A_n \right) \geq 1- 2n^{-2},
\]
where $A_n$ is as defined in the lemma. Note that since $F(\theta_0)=1/2$, we have replaced $1/2$ with $F(\theta_0)$. We proceed in a similar manner to prove the other inequality. Using \Cref{lem:conc_uniform} with $k=n-k_{n,\alpha}$, we have the following inequality with probability greater than or equal to $1-2n^{-2}$ for all sample sizes $n \geq \log_2(2/\alpha)$,
\begin{equation*}
\begin{split}
    \left| F(X_{(n-k_{n,\alpha})}) - \frac{n-k_{n,\alpha}}{n+1}\right| &\leq \sqrt{\frac{2(n-k_{n,\alpha})}{n+1}\left( 1- \frac{n-k_{n,\alpha}}{n+1}\right)\frac{2\log(n)}{n+1}}+\left|1-\frac{2(n-k_{n,\alpha})}{n+1} \right|\frac{2\log(n)}{n+1} \\
    &\leq \sqrt{\frac{1}{4}\frac{4\log(n)}{n+1}}+\left(1-\frac{2(n-k_{n,\alpha})}{n+1} \right)\frac{2\log(n)}{n+1} \\
    &\leq \sqrt{\frac{\log(n)}{n+1}}+\left(1-\frac{2(n-((n/2)+(\sqrt{n}z_{\alpha/2}/2)+1))}{n+1} \right)\frac{2\log(n)}{n+1} \\
    &\leq \sqrt{\frac{\log (n)}{n+1}} + \left(\frac{z_{\alpha/2}}{\sqrt{n}} + \frac{3}{n}\right)\frac{2\log (n)}{n+1}.
\end{split}    
\end{equation*}
We showed earlier that $2(k_{n,\alpha}+1)\geq n+1$. This implies that $2(n-k_{n,\alpha}) \leq n+1$. We have used this in the above derivation to remove the modulus. Using \eqref{eq:hd_new_bound} we obtain that,
\begin{equation*}
    \begin{split}
        \left|\frac{n-k_{n,\alpha}}{n+1} - \frac{1}{2} \right| &= \left( \frac{1}{2} - \frac{n-k_{n,\alpha}}{n+1}\right) \\
        &\leq \left(\frac{1}{2} - \frac{n-((n/2)+(\sqrt{n}z_{\alpha/2}/2)+1)}{n+1} \right)\\
        & = \frac{\sqrt{n}z_{\alpha/2}+3}{n+1} \\
        &= \frac{z_{\alpha/2}}{2\sqrt{n}}+\frac{3}{2n}.
    \end{split}
\end{equation*}
Combining the above two inequalities using triangle inequality we obtain that,
\[
\mathbb{P}\left(\left|F(X_{(n-k_{n,\alpha})})-F(\theta_0) \right| \leq B_n \right) \geq 1- 2n^{-2},
\]
where $B_n$ is as defined in the lemma. We note that $\ceil{n/2}+c_{n,\alpha}+1=k_{n,\alpha}+1$ and $\floor{n/2}-c_{n,\alpha}=n-k_{n,\alpha}$. Thus the confidence interval of $\theta_0$ for $n\geq \log_2(2/\alpha)$, $[X_{(\floor{n/2}-c_{n,\alpha})},X_{(\ceil{n/2}+c_{n,\alpha}+1)}]$ is same as the confidence interval $[X_{(n-k_{n,\alpha})},X_{(k_{n,\alpha}+1)}]$. 
\end{proof}

\begin{prop}
 \label{thm:sukhatme}   
 Let $X_1,X_2,...,X_n \stackrel{iid}{\sim} F $. Assume that $F$ is a continuous CDF whose $p$-th quantile is $\eta_p$ (i.e.\ $F^{-1}(p)=\eta_p$). Let $\{k_n\}$ be a sequence of positive integers such that $k_n=np+O(n^{1/2})$. Then we have,
\begin{equation}
    F(X_{(k_n)})=\frac{k_n}{n}-F_n(\eta_p)+F(\eta_p)+R_n,
\end{equation}
where $\sqrt{n}R_n \xrightarrow{P} 0$. 
\end{prop}
\begin{proof}[Proof of Proposition~\ref{thm:sukhatme}]
We observe that if $F$ is a continuous distribution function then $F(X_1), F(X_2),\cdots, F(X_n)$ are i.i.d.\ $\text{Uniform}(0,1)$ random variables. Let $Y_i=F(X_i)$ for $i=1,2,\cdots,n$. The $Y_i$'s are i.i.d.\ from $U \stackrel{d}{=} \text{Uniform}(0,1)$. Applying Theorem-$1$ of \cite{ghosh1971new} on $Y_1,Y_2,...,Y_n \stackrel{iid}{\sim} U $ and noting that $p_n-p=(k_n/n)-p=O(n^{-1/2})$ and $F(\eta_p)=p$, we obtain the desired result. 
\end{proof}

\begin{lemma}
\label{lem:joint_distribution}
   $(Q_{1,\boldsymbol{X},n},Q_{2,\boldsymbol{X},n})$ has the following asymptotic joint distribution,
   \[
   \sqrt{n}(Q_{1,\boldsymbol{X},n},Q_{2,\boldsymbol{X},n}) \stackrel{d}{\xrightarrow{}}  N\left(\left(\frac{z_{\alpha/2}}{2}, -\frac{z_{\alpha/2}}{2}\right), \frac{1}{4}\textbf{1}\textbf{1}^{T} \right).
   \]
\end{lemma}

\begin{proof}[Proof of \Cref{lem:joint_distribution}]
We decompose $Q_{1,\boldsymbol{X},n}$ and $Q_{2,\boldsymbol{X},n}$ as follows,
\begin{equation}
    \begin{split}
        Q_{1,\boldsymbol{X},n} &= P_{n} + S_{n}, \\
        Q_{2,\boldsymbol{X},n} &= R_{n} + T_{n},
    \end{split}
\end{equation}
where,
\begin{equation}
    \begin{split}
        P_n &= \left(\frac{k_{n,\alpha}+1}{n} - \frac{1}{2}\right), \\
        R_n &= \left(\frac{n-k_{n,\alpha}}{n} - \frac{1}{2}\right), \\
        S_n &= \frac{1}{n}\sum_{i=1}^n \mathbf{1}\{F(X_i) \le (k_{n,\alpha}+1)/n\} - \frac{k_{n,\alpha}+1}{n},\\
        T_n &= \frac{1}{n}\sum_{i=1}^n \mathbf{1}\{F(X_i) \le (n-k_{n,\alpha})/n\} - \frac{n-k_{n,\alpha}}{n}.
    \end{split}
\end{equation}
We shall obtain the asymptotic joint distribution of $(S_n,T_n)$ by deriving the asymptotic distributions of the linear combinations $l_1S_n + l_2T_n$ where $l_1,l_2 \in \mathbb{R}$. We shall use Lyapounov CLT for this purpose. We define $Y_{ni}$ for $i=1,\cdots,n$ as follows,
\begin{equation}
    Y_{ni} = \frac{l_1(\mathbf{1}\{F(X_i) \le p_{1n} \}-p_{1n}) + l_2(\mathbf{1}\{F(X_i) \le p_{2n} \}-p_{2n})}{\sqrt{n(l_1^2p_{1n}(1-p_{1n})+l_2^2p_{2n}(1-p_{2n})+2l_1l_2p_{2n}(1-p_{1n}))}},
\end{equation}
where $p_{1n}=(k_{n,\alpha}+1)/n$ and $p_{2n}=(n-k_{n,\alpha})/n$. Note that for each $n$, $Y_{ni}$ for $i=1,\cdots,n$ are mutually independent. The following conditions can be easily verified:
\begin{alignat}{2}
   &\mathbb{E}(Y_{ni}) &&= 0 \quad\quad \mbox{for all}\quad n,i,\\
   &\sum_{i=1}^n \mathbb{E}(Y_{ni}^2) &&= 1 \quad\quad \mbox{for all}\quad n, \\
  & \lim_{n\xrightarrow{}0}\sum_{i=1}^n \mathbb{E}(Y_{ni}^4) &&= 0.
\end{alignat}
Thus using Lyapounov CLT we obtain that $\sum_{i=1}^n Y_{ni} \stackrel{d}{\xrightarrow{}} N(0,1)$. In other words, for any $l_1,l_2 \in \mathbb{R}$ we have,
\begin{equation}
    \sqrt{n}(l_1S_n + l_2T_n) \stackrel{d}{\xrightarrow{}} N\left(0,\frac{1}{4}(l_1+l_2)^2\right).
\end{equation}
To obtain above, we used the fact that both $p_{1n}$ and $p_{2n}$ converge to $1/2$ as $n$ goes to infinity. Hence the asymptotic joint distribution of $(S_n,T_n)$ is as follows,
\begin{equation}
    \sqrt{n}(S_n,T_n) \stackrel{d}{\xrightarrow{}} N\left(\left(0,0\right), \frac{1}{4}\textbf{1}\textbf{1}^{T} \right).
\end{equation}
From earlier computations we know that $\sqrt{n}P_n \xrightarrow{} z_{\alpha/2}/2$ and $\sqrt{n}R_n \xrightarrow{} -z_{\alpha/2}/2$. Applying Slutsky's theorem we get that,
\begin{equation}
    \sqrt{n}(Q_{1,\boldsymbol{X},n},Q_{2,\boldsymbol{X},n}) \stackrel{d}{\xrightarrow{}} (W_1,W_2),
\end{equation}
where, 
\begin{equation}
    (W_1,W_2) \sim N\left(\left(\frac{z_{\alpha/2}}{2}, -\frac{z_{\alpha/2}}{2}\right), \frac{1}{4}\textbf{1}\textbf{1}^{T} \right).
\end{equation}
   
This completes the proof of the lemma.
\end{proof}

\begin{prop}
    \label{prop:boundedness_of_limit}
Suppose $\mathscr{G}(a, b) = |a|^{1/\rho}\mbox{sgn}(a) - |a - b|^{1/\rho}\mbox{sgn}(a - b)$. The following holds true for any $\alpha \in (0,1)$ and for any $\rho \geq 1$, 
\[
 \mathbb{P}\left(\frac{1}{2}\mathscr{G}((Z/z_{\alpha/2})+1, 2)  \leq 2^{ (1/\rho) - 1} \right)  = \alpha \quad \mbox{where} \quad Z\sim N(0,1). 
\]
\end{prop}

\begin{proof}[Proof of Proposition~\ref{prop:boundedness_of_limit}]
We begin the proof by showing the following equivalence for any $a \in \mathbb{R}$ and $\rho > 1$,
\begin{equation}
\label{eq:imp_equiv}
    \mathscr{G}(a,2)\le 2^{1/\rho}
\quad\Longleftrightarrow\quad
a\le 0 \ \text{or}\ a\ge 2.
\end{equation}
For $\rho \geq 1$, the function $h(x)=x^{1/\rho}$ is concave on $[0,\infty)$. Hence, for any $y\ge 0$, the increment
\[
\Delta_y(x)\;:=\;h(x+y)-h(x)
\]
is nonincreasing in $x\ge 0$ (since $h'$ is nonincreasing). Therefore,
\begin{equation}
   \label{eq:concave_prop} 
   h(x+y)-h(x)\le h(0+y)-h(0)=y^{1/\rho}
\qquad\text{for all }x,y\ge 0.
\end{equation}
We now consider three cases.

\medskip\noindent
\textit{Case} 1: $a\ge 2$. Here $\mbox{sgn}(a)=\mbox{sgn}(a-2)=1$ and hence, 
\[
\mathscr{G}(a,2)=a^{1/\rho}-(a-2)^{1/\rho}.
\]
Writing $a=(a-2)+2$ and applying \eqref{eq:concave_prop} with $x=a-2$ and $y=2$ yields
\[
\mathscr{G}(a,2)=\big((a-2)+2\big)^{1/\rho}-(a-2)^{1/\rho}\le 2^{1/\rho}.
\]

\medskip\noindent
\textit{Case} 2: $a\le 0$. Then $\mbox{sgn}(a)=\mbox{sgn}(a-2)=-1$, and using $|a|=-a$, $|a-2|=2-a=(-a)+2$ gives
\[
\mathscr{G}(a,2)=|a-2|^{1/\rho}-|a|^{1/\rho}=(2-a)^{1/\rho}-(-a)^{1/\rho}.
\]
We apply \eqref{eq:concave_prop} with $x=-a$ and $y=2$ to obtain
\[
\mathscr{G}(a,2)=\big((-a)+2\big)^{1/\rho}-(-a)^{1/\rho}\le 2^{1/\rho}.
\]
\medskip\noindent
\textit{Case} 3: $0<a<2$. Then $\mbox{sgn}(a)=1$ and $\mbox{sgn}(a-2)=-1$, hence
\[
\mathscr{G}(a,2)=a^{1/\rho}+(2-a)^{1/\rho}.
\]
By concavity of $h(x)=x^{1/\rho}$ and Jensen's inequality,
\[
\frac{a^{1/\rho}+(2-a)^{1/\rho}}{2}
\ge \left(\frac{a+(2-a)}{2}\right)^{1/\rho}
=1,
\]
This implies that $\mathscr{G}(a,2)\ge 2$. If $\rho>1$, then $2>2^{1/\rho}$, and thus
$\mathscr{G}(a,2)>2^{1/\rho}$ for all $a\in(0,2)$. Combining the three cases shows that $\mathscr{G}(a,2)\le 2^{1/\rho}$ holds for
$a\le 0$ or $a\ge 2$, while for $\rho>1$ it fails on $(0,2)$, proving the
claimed equivalence. (At the boundary points, $\mathscr{G}(0,2)=\mathscr{G}(2,2)=2^{1/\rho}$
under the usual convention $\mbox{sgn}(0)=0$.) 

Using the above equivalence we can derive the following probability statement, 
\begin{equation*}
    \begin{split}
        \mathbb{P}\left(\frac{1}{2}\mathscr{G}((Z/z_{\alpha/2})+1, 2)  \leq 2^{ (1/\rho) - 1} \right) = &  \mathbb{P}\left(\mathscr{G}((Z/z_{\alpha/2})+1, 2)  \leq 2^{ 1/\rho } \right) \\
        \stackrel{\eqref{eq:concave_prop}}{=} & \mathbb P\left( (Z/z_{\alpha/2})+1 \notin (0,2)\right)  \\
        = & \mathbb P \left(Z \leq - z_{\alpha/2} \right) + \mathbb P \left( Z \geq z_{\alpha/2}\right) \\
        = & \alpha. 
    \end{split}
\end{equation*}
\end{proof}

\begin{prop}
\label{prop:unequal_limits_bound_G}
Let $\rho\ge 1$, $M_->0$, $M_+>0$, and for $a,b\in\mathbb{R}$ define
\[
\overline{\mathscr{G}}(a,b)
=
|a|^{1/\rho}\,\mbox{sgn}(a)
\left[
\frac{\mathbf{1}\{a<0\}}{M_-^{1/\rho}}
+
\frac{\mathbf{1}\{a>0\}}{M_+^{1/\rho}}
\right]
-
|a-b|^{1/\rho}\,\mbox{sgn}(a-b)
\left[
\frac{\mathbf{1}\{a<b\}}{M_-^{1/\rho}}
+
\frac{\mathbf{1}\{a>b\}}{M_+^{1/\rho}}
\right].
\]
Then, for every $a\in\mathbb{R}$,
\[
\overline{\mathscr{G}}(a,z_{\alpha/2})
\le
2^{\,1-1/\rho}\, z_{\alpha/2}^{1/\rho}\,
\max\!\left\{M_-^{-1/\rho},\,M_+^{-1/\rho}\right\}.
\]

\end{prop}

\begin{proof}
Write $b=z_{\alpha/2}>0$. We consider three cases depending on the value of $a$.

\medskip
\noindent\textit{Case} 1: $a\le 0$. Then $\mathbf{1}\{a<0\}=1$, $\mathbf{1}\{a>0\}=0$, and also $a<b$ since $b>0$.
Hence
\[
\overline{\mathscr{G}}(a,b)
=
\frac{|a|^{1/\rho}\,\mbox{sgn}(a)}{M_-^{1/\rho}}
-
\frac{|a-b|^{1/\rho}\,\mbox{sgn}(a-b)}{M_-^{1/\rho}}
=
\frac{(b-a)^{1/\rho}-(-a)^{1/\rho}}{M_-^{1/\rho}},
\]
because $\mbox{sgn}(a)=\mbox{sgn}(a-b)=-1$ and $|a|=-a$, $|a-b|=b-a$.
Since $\rho\ge 1$, the function $x\mapsto x^{1/\rho}$ is concave on $[0,\infty)$,
so the increment $x\mapsto (x+b)^{1/\rho}-x^{1/\rho}$ is nonincreasing in $x\ge 0$.
Taking $x=-a\ge 0$ yields
\[
(b-a)^{1/\rho}-(-a)^{1/\rho}
=
\big((-a)+b\big)^{1/\rho}-(-a)^{1/\rho}
\le
b^{1/\rho}.
\]
Therefore,
\[
\overline{\mathscr{G}}(a,b)\le b^{1/\rho}M_-^{-1/\rho}
\le 2^{\,1-1/\rho} b^{1/\rho}\max\!\{M_-^{-1/\rho},M_+^{-1/\rho}\}.
\]

\medskip
\noindent\textit{Case} 2: $a\ge b$. Then $\mathbf{1}\{a<b\}=0$, $\mathbf{1}\{a>b\}=1$, and also $\mbox{sgn}(a)=\mbox{sgn}(a-b)=1$.
Thus
\[
\overline{\mathscr{G}}(a,b)
=
\frac{a^{1/\rho}}{M_+^{1/\rho}}
-
\frac{(a-b)^{1/\rho}}{M_+^{1/\rho}}
=
\frac{a^{1/\rho}-(a-b)^{1/\rho}}{M_+^{1/\rho}}.
\]
By the same concavity argument applied to the increment $x\mapsto (x+b)^{1/\rho}-x^{1/\rho}$,
with $x=a-b\ge 0$, we have $a^{1/\rho}-(a-b)^{1/\rho}\le b^{1/\rho}$. Hence
\[
\overline{\mathscr{G}}(a,b)\le b^{1/\rho}M_+^{-1/\rho}
\le 2^{\,1-1/\rho} b^{1/\rho}\max\!\{M_-^{-1/\rho},M_+^{-1/\rho}\}.
\]

\medskip
\noindent\textit{Case} 3: $0<a<b$. Then $\mbox{sgn}(a)=1$ and $\mbox{sgn}(a-b)=-1$, and the indicators select
$M_+$ in the first bracket and $M_-$ in the second bracket. Therefore
\[
\overline{\mathscr{G}}(a,b)
=
\frac{a^{1/\rho}}{M_+^{1/\rho}}
+
\frac{(b-a)^{1/\rho}}{M_-^{1/\rho}}
\le
\max\!\{M_-^{-1/\rho},M_+^{-1/\rho}\}\,\Big(a^{1/\rho}+(b-a)^{1/\rho}\Big).
\]
Using concavity of $x\mapsto x^{1/\rho}$ again, Jensen's inequality gives
\[
\frac{a^{1/\rho}+(b-a)^{1/\rho}}{2}
\le
\left(\frac{a+(b-a)}{2}\right)^{1/\rho}
=
\left(\frac{b}{2}\right)^{1/\rho},
\]
so $a^{1/\rho}+(b-a)^{1/\rho}\le 2^{\,1-1/\rho} b^{1/\rho}$. Hence
\[
\overline{\mathscr{G}}(a,b)
\le
2^{\,1-1/\rho} b^{1/\rho}\max\!\{M_-^{-1/\rho},M_+^{-1/\rho}\}.
\]

\medskip
Combining the three cases proves the stated bound for all $a\in\mathbb{R}$.
\end{proof}
\section{Proof of Theorem~\ref{thm:assym_of_cn}}
\label{appendix:theorem_asym_cn}
At first we shall compute the value of $c_{n,\alpha}$ which has been defined in Step 1 of Algorithm~\Cref{alg:proposed-conf-int}. Let $G_n(\cdot)$ denote the cumulative distribution function of $\mbox{Bin}(n, 1/2)$ and let 
\[
G_n^{-1}(p) = \inf\{x\in\mathbb{R}:\,G_n(x)\ge p\},
\] 
denote the quantile function (inverse CDF) of $\mbox{Bin}(n, 1/2)$ distribution.
Let us first consider the case that $n$ is even. In this case,
\begin{equation}
\begin{split}
    c_{n,\alpha} &=\inf \{x: \mathbb{P}( Y_n \geq (n/2)-x) \geq 1-\alpha/2 \} \text{  where  }Y_n \sim \text{Bin}(n,1/2) \\
    &= \inf \left\{ x: 1-\mathbb{P}\left(Y_n < \frac{n}{2}-x\right) \geq 1-\alpha/2 \right\} \text{  where  }Y_n \sim \text{Bin}(n,1/2) \\
    &= \inf \{ x: 1-G_n((n/2)-x-1) \geq 1-\alpha/2\} \\
    &= \inf \{ x: G_n((n/2)+x) \geq 1-\alpha/2 \} \\
    &= G_n^{-1}(1-(\alpha/2))-(n/2).
\end{split}
\end{equation}
The second last inequality follows from the following computation,
\begin{equation}
\begin{split}
    G_n\big(\frac{n}{2}-c_{n,\alpha}-1\big) &= \frac{1}{2^n}\sum_{t=0}^{\frac{n}{2}-c_{n,\alpha}-1} \dbinom{n}{t} \\
    &= \frac{1}{2^n}\sum_{t=0}^{\frac{n}{2}-c_{n,\alpha}-1} \dbinom{n}{n-t} \\
    &= \frac{1}{2^n}\sum_{t=\frac{n}{2}+c_{n,\alpha}+1}^{n} \dbinom{n}{t} \\
    &= 1- \frac{1}{2^n}\sum_{t=0}^{\frac{n}{2}+c_{n,\alpha}} \dbinom{n}{t} \\
    &= 1- G_n\big(\frac{n}{2}+c_{n,\alpha}\big).
\end{split}
\end{equation}
We approach in a similar manner for the case when $n$ is odd.
\begin{equation}
    \begin{split}
        c_{n,\alpha} &=\inf \left\{x: \mathbb{P}\left( Y_n \geq \floor{\frac{n}{2}}-x\right) \geq 1-\alpha/2 \right\} \text{  where  }Y_n \sim \text{Bin}(n,1/2) \\
    &= \inf \left\{ x: 1-\mathbb{P}\left(Y_n < \floor{\frac{n}{2}}-x\right) \geq 1-\alpha/2 \right\} \text{  where  }Y_n \sim \text{Bin}(n,1/2) \\
    &= \inf \left\{ x: 1-G_n\left(\floor{\frac{n}{2}}-x-1\right) \geq 1-\alpha/2 \right\} \\
    &= \inf \{ x: G_n(\floor{{n}/{2}}+x+1) \geq 1-\alpha/2 \} \\
    &= G_n^{-1}(1-(\alpha/2))-\ceil{n/2}.
    \end{split}
\end{equation}
The second last inequality again follows from the following computation,
\begin{equation}
\begin{split}
    G_n\big(\floor{\frac{n}{2}}-c_{n,\alpha}-1\big) &= \frac{1}{2^n}\sum_{t=0}^{\floor{\frac{n}{2}}-c_{n,\alpha}-1} \dbinom{n}{t} \\
    &= \frac{1}{2^n}\sum_{t=0}^{\floor{\frac{n}{2}}-c_{n,\alpha}-1} \dbinom{n}{n-t} \\
    &= \frac{1}{2^n}\sum_{t=\floor{\frac{n}{2}}+c_{n,\alpha}+2}^{n} \dbinom{n}{t} \\
    &= 1- \frac{1}{2^n}\sum_{t=0}^{\floor{\frac{n}{2}}+c_{n,\alpha}+1} \dbinom{n}{t} \\
    &= 1- G_n\big(\floor{\frac{n}{2}}+c_{n,\alpha}+1\big).
\end{split}    
\end{equation}
Hence combining both the cases (for even and odd sample sizes) we conclude that, for any $n \in \mathbb{N}$, 
\begin{equation}\label{eq:formula-for-c_{n,alpha}}
c_{n,\alpha}~=~G_n^{-1}(1-(\alpha/2))-\ceil{(n/2)}.
\end{equation}
\paragraph{Proof of~\ref{conclu:coverage-guarantee}}
From the derivation of the value of $c_{n,\alpha}$, we obtain that $$c_{n,\alpha}\in\{x:\,\mathbb{P}( Y_n \geq \floor{n/2}-x ) \ge 1 - \alpha/2\}.$$ This holds because $G_n(G_n^{-1}(1-(\alpha/2))) \geq 1-(\alpha/2)$. Hence we can say the following,
\begin{equation}
    \mathbb{P}( Y_n \geq \floor{n/2}-c_{n,\alpha}) \ge 1 - \alpha/2\,\quad\mbox{where}\quad Y_n \sim \text{Bin}(n,1/2).
\end{equation}
We now note that under the assumption that $X_1, X_2, \ldots, X_n$ are independent random variables from distributions with median $\theta_0$. Thus, we obtain,
\begin{equation}
    \begin{split}
        \mathbb{P}(\theta_0\in\widehat{\mathrm{CI}}_{1,n,\alpha}) &= \mathbb{P} \left\{ \sum_{i=1}^n \mathbf{1}\{X_i \le \theta_0\} \ge \lfloor n/2\rfloor - c_{n,\alpha}\right\} \\
        &\ge \mathbb{P} \left\{ Y_n \ge \lfloor n/2\rfloor - c_{n,\alpha}\right\} \text {   where  } Y_n \sim \text{Bin}(n,1/2)\\
        & \ge 1 - \alpha/2.
    \end{split}
\end{equation}
Similarly we have that,
\begin{equation}
    \begin{split}
        \mathbb{P}(\theta_0\in\widehat{\mathrm{CI}}_{2,n,\alpha}) &= \mathbb{P} \left\{ \sum_{i=1}^n \mathbf{1}\{X_i \ge \theta_0\} \ge \lfloor n/2\rfloor - c_{n,\alpha}\right\} \\
        &\ge \mathbb{P} \left\{ Y_n \ge \lfloor n/2\rfloor - c_{n,\alpha}\right\} \text {   where  } Y_n \sim \text{Bin}(n,1/2)\\
        & \ge 1 - \alpha/2.
    \end{split}
\end{equation}
Combining the above two equations, the coverage probability of $\widehat{\mathrm{CI}}_{n,\alpha}=\widehat{\mathrm{CI}}_{1,n,\alpha} \cap \widehat{\mathrm{CI}}_{2,n,\alpha}$ is obtained as,
\begin{equation}
    \begin{split}
      \mathbb{P}(\theta_0\in\widehat{\mathrm{CI}}_{n,\alpha}) &=  \mathbb{P}(\theta_0\in\widehat{\mathrm{CI}}_{1,n,\alpha} \cap \widehat{\mathrm{CI}}_{2,n,\alpha}) \\
      &= \mathbb{P}(\theta_0\in\widehat{\mathrm{CI}}_{1,n,\alpha})+\mathbb{P}(\theta_0\in\widehat{\mathrm{CI}}_{2,n,\alpha})-\mathbb{P}(\theta_0\in\widehat{\mathrm{CI}}_{1,n,\alpha} \cup \widehat{\mathrm{CI}}_{2,n,\alpha}) \\
      &\ge (1 - \alpha/2) + (1 - \alpha/2) -1 \\
      &= 1-\alpha.    
    \end{split}
\end{equation}
This proves the first part of the theorem. 

\paragraph{Proof of~\ref{conclu:form-of-conf-interval}}
We observe that if $\lfloor n/2\rfloor - c_{n,\alpha} =0$, then 
\[
\widehat{\mathrm{CI}}_{n,\alpha} = \widehat{\mathrm{CI}}_{1,n,\alpha} \cap \widehat{\mathrm{CI}}_{2,n,\alpha} = \mathbb{R} \cap \mathbb{R} = \mathbb{R},
\]
because both $\sum_{i=1}^n \mathbf{1}\{X_i \le \theta\}$ and $\sum_{i=1}^n \mathbf{1}\{X_i \ge \theta\}$ are greater than or equal to $0$ no matter what the value of $\theta$ is. Hence we shall obtain a non-trivial confidence region only when $\lfloor n/2\rfloor - c_{n,\alpha} >0$. This condition can be further simplified as follows,
\begin{equation}
    \begin{split}
        \floor{n/2} - c_{n,\alpha} > 0 & \iff \floor{n/2} - (G_n^{-1}(1-(\alpha/2))-\ceil{n/2}) > 0\\
        & \iff G_n^{-1}(1-(\alpha/2)) < \floor{n/2}+ \ceil{n/2} \\
        & \iff G_n^{-1}(1-(\alpha/2)) < n \\
        & \iff G_n^{-1}(1-(\alpha/2)) \leq n-1 \\
        & \iff G_n(n-1) \geq 1-(\alpha/2) \\
        & \iff 1-\frac{1}{2^n} \geq 1-(\alpha/2)  \\
        & \iff n \geq \log_2(2/\alpha).
    \end{split}
\end{equation}
Therefore we conclude that if $n < \log_2(2/\alpha)$, then $\widehat{\mathrm{CI}}_{n,\alpha}=\mathbb{R}$. If $n \geq \log_2(2/\alpha)$, then the confidence interval is non-trivial as $\lfloor n/2\rfloor - c_{n,\alpha} >0$. 

We now show that if $n\ge\log_2(2/\alpha)$, then $\widehat{\mathrm{CI}}_{n,\alpha}$ can be represented in terms of the order statistics. Recall that
\[
\widehat{\mathrm{CI}}_{n,\alpha} = \widehat{\mathrm{CI}}_{1,n,\alpha} \cap \widehat{\mathrm{CI}}_{2,n,\alpha},
\]
where,
 \[
    \begin{cases}
      \widehat{\mathrm{CI}}_{1,n,\alpha} = \left\{\theta\in\mathbb{R}:\, \sum_{i=1}^n \mathbf{1}\{X_i \le \theta\} \ge \left\lfloor \frac{n}{2}\right\rfloor - c_{n,\alpha}\right\}, \\
    \widehat{\mathrm{CI}}_{2,n,\alpha} = \left\{\theta\in\mathbb{R}:\, \sum_{i=1}^n \mathbf{1}\{X_i \ge \theta\} \ge \left\lfloor \frac{n}{2}\right\rfloor - c_{n,\alpha}\right\}.     
    \end{cases}
 \]
Observe that if $\lfloor n/2\rfloor - c_{n,\alpha} > 0$, then
\begin{equation}
\begin{split}
    X_{(\floor{n/2}-c_{n,\alpha})} \leq \theta & \iff \text{At least $\floor{n/2}-c_{n,\alpha}$ many $X_i$'s are less than or equal to $\theta$} \\
    & \iff \sum_{i=1}^n \mathbf{1}\{X_i \le \theta\} \geq \floor{n/2}-c_{n,\alpha}.
\end{split}
\end{equation}
This implies that,
\begin{equation}
    \begin{split}
      \widehat{\mathrm{CI}}_{1,n,\alpha} &= \left\{\theta\in\mathbb{R}:\, \sum_{i=1}^n \mathbf{1}\{X_i \le \theta\} \ge \left\lfloor \frac{n}{2}\right\rfloor - c_{n,\alpha}\right\} \\
      &= \left\{\theta\in\mathbb{R}:\,  X_{(\floor{n/2}-c_{n,\alpha})} \leq \theta\right\} \\
      &= \left[ X_{(\floor{n/2}-c_{n,\alpha})}, \infty \right).
    \end{split}
\end{equation}
We also note that under the condition $\lfloor n/2\rfloor - c_{n,\alpha} > 0$, the following happens,
\begin{equation}
    \begin{split}
        X_{(\ceil{{n}/{2}}+c_{n,\alpha}+1)} \ge \theta & \iff X_{(\ceil{{n}/{2}}+c_{n,\alpha}+1)}, \cdots, X_{(n)} \ge \theta \\
        & \iff \text{At least $n-(\ceil{{n}/{2}}+c_{n,\alpha}+1)+1= \floor{n/2}-c_{n,\alpha}$ many $X_i$'s are $\ge \theta$} \\
        & \iff \sum_{i=1}^n \mathbf{1}\{X_i \ge \theta\} \geq \floor{n/2}-c_{n,\alpha}.
    \end{split}
\end{equation}
This implies that,
\begin{equation}
    \begin{split}
      \widehat{\mathrm{CI}}_{2,n,\alpha} &= \left\{\theta\in\mathbb{R}:\, \sum_{i=1}^n \mathbf{1}\{X_i \ge \theta\} \ge \left\lfloor \frac{n}{2}\right\rfloor - c_{n,\alpha}\right\} \\
      &= \left\{\theta\in\mathbb{R}:\,  X_{(\ceil{{n}/{2}}+c_{n,\alpha}+1)} \ge \theta\right\} \\
      &= \left( -\infty, X_{(\ceil{{n}/{2}}+c_{n,\alpha}+1)}\right].
    \end{split}
\end{equation}
This shows that, 
\begin{equation}
    \begin{split}
      \widehat{\mathrm{CI}}_{n,\alpha} &= \widehat{\mathrm{CI}}_{1,n,\alpha} \cap \widehat{\mathrm{CI}}_{2,n,\alpha} \\
      &= \left[ X_{(\floor{n/2}-c_{n,\alpha})}, \infty \right) \cap \left( -\infty, X_{(\ceil{{n}/{2}}+c_{n,\alpha}+1)}\right] \\
      &= \left[X_{(\floor{n/2}-c_{n,\alpha})}, X_{(\ceil{{n}/{2}}+c_{n,\alpha}+1)}\right].
    \end{split}
\end{equation}
This completes the proof of the second part of the theorem. 
\paragraph{Proof of~\ref{conclu:c_{n,alpha}_bounds}}
We shall now establish the relation between $c_{n,\alpha}$ and the normal quantile, $z_{\alpha/2}$. We have already seen that, for any $n \in \mathbb{N}$, $c_{n,\alpha}=G_n^{-1}(1-(\alpha/2))-\ceil{(n/2)}$. Suppose $G_n^{-1}(1-(\alpha/2))=k$. We know that $G_n^{-1}(1-(\alpha/2))=k$ iff the following two inequalities hold true:
\begin{equation}\label{eq:defining_equation_k}
    \begin{split}
        G_n(k-1) < 1 - \frac{\alpha}{2} \le G_n(k).
    \end{split}
\end{equation}
Note that $1 - \alpha/2 \ge 1 - 2^{-n}$ (or equivalently, $n \le \log_2(2/\alpha)$) if and only if $k = n$ because $G_n(n-1) = 1 - 2^{-n}$ and $G_n(n) = 1$. Therefore, for $\alpha\in(0, 1)$ satisfying $\log_2(2/\alpha) < n$, $k \leq n-1$. 
For such $k$, the universal inequalities for binomial distribution mentioned in Theorem $1$ of \cite{zubkov2012full} imply that,
\[
C_{n}(k-1) \le G_n(k-1) \le C_{n}(k) \le G_n(k) \le C_n(k+1)\quad\mbox{for all}\quad k\in\{1, 2, \ldots, n-1\}, 
\]
where for $1 \le k\le n-1$,
\[
C_n(k) = \Phi\left( \mathrm{sgn}\left(\frac{k}{n} - \frac{1}{2}\right)\sqrt{2nH\left(\frac{k}{n},\,\frac{1}{2}\right)}\right),\quad\mbox{with}\quad H(x, 1/2) = x\ln(2x) + (1-x)\ln(2 - 2x).
\]
Suppose $k'\in\{1, 2, \ldots, n\}$ is such that $C_n(k') < 1 - \alpha/2 \le C_n(k'+1)$, then
\[
G_n(k' - 1) ~\le~ C_n(k') ~<~ 1 - \frac{\alpha}{2} ~\le~ C_n(k'+1) ~\le~ G_n(k'+1).
\]
This implies $k\in\{k', k' + 1\}$ or equivalently, $k' \le k \le k' + 1$. Note that for $\alpha < 1$,
\[
C_n(\lfloor n/2\rfloor) = \Phi\left(-\sqrt{2nH\left(\frac{\lfloor n/2\rfloor}{n},\,\frac{1}{2}\right)}\right) \le \frac{1}{2} < 1 - \frac{\alpha}{2}\quad\mbox{and}\quad C_n(\lceil n/2\rceil) = \Phi\left(\sqrt{2nH\left(\frac{\lceil n/2\rceil}{n},\,\frac{1}{2}\right)}\right) > \frac{1}{2}.
\]
Hence, $k' = \lfloor n/2\rfloor$ if and only if 
\[
\Phi\left(\sqrt{2nH\left(\frac{\lceil n/2\rceil}{n},\,\frac{1}{2}\right)}\right) \ge 1 - \frac{\alpha}{2}\quad\Leftrightarrow\quad H\left(\frac{\lceil n/2\rceil}{n},\,\frac{1}{2}\right) \ge \frac{z_{\alpha/2}^2}{2n}.
\]
If $n$ is even, this cannot never occur as $H(\lceil n/2\rceil/n, 1/2) = 0$ in this case. If $n$ is odd, then $\lceil n/2\rceil = (n + 1)/2$ and the inequality $H(\lceil n/2\rceil/n, 1/2) \ge z_{\alpha/2}^2/(2n)$ becomes
\begin{equation}\label{eq:inequality-k'-equals-n/2}
\left(\frac{1}{2} + \frac{1}{2n}\right)\log\left(1 + \frac{1}{n}\right) + \left(\frac{1}{2} - \frac{1}{2n}\right)\log\left(1 - \frac{1}{n}\right) ~=~ H\left(\frac{1}{2} + \frac{1}{2n},\,\frac{1}{2}\right) \ge \frac{z_{\alpha/2}^2}{2n}.
\end{equation}
By \Cref{lem:kl_bound} we get $H(x, 1/2) \ge 2(x - 1/2)^2 + (4/3)(x - 1/2)^4$. This implies that inequality~\eqref{eq:inequality-k'-equals-n/2} holds only if
\[
2\left(\frac{1}{2n}\right)^2 + \frac{4}{3}\left(\frac{1}{2n}\right)^4 \ge \frac{z_{\alpha/2}^2}{2n}\quad\Leftrightarrow\quad \frac{1}{n} + \frac{1}{6n^3} \ge z_{\alpha/2}^2.
\]
Assume $n$ is large enough so that $H(\lceil n/2\rceil/n, 1/2) < z_{\alpha/2}^2/(2n)$. Then $k' > n/2$ and hence,
\[
C_n(k') < 1 - \alpha/2\quad\Leftrightarrow\quad H\left(\frac{k'}{n},\,\frac{1}{2}\right) \le \frac{z_{\alpha/2}^2}{2n}.
\]
Similarly, $C_n(k'+1) \ge 1 - \alpha/2$ becomes
\[
H\left(\frac{k' + 1}{n},\,\frac{1}{2}\right) \ge \frac{z_{\alpha/2}^2}{2n}.
\]
Hence, the defining inequality for $k'$ is
\[
H\left(\frac{k'}{n},\,\frac{1}{2}\right) \le \frac{z_{\alpha/2}^2}{2n} \le H\left(\frac{k'+1}{n},\,\frac{1}{2}\right).
\]
From \Cref{lem:kl_bound} we know that,
\[
2(x - 1/2)^2 + (4/3)(x - 1/2)^4 \le H(x, 1/2) \le 2(x - 1/2)^2 + 3.2(x - 1/2)^4\quad\mbox{for all}\quad x\in[0, 1].
\]
Therefore, $k'$ satisfies
\[
2\left(\frac{k'}{n} - \frac{1}{2}\right)^2 + (4/3)\left(\frac{k'}{n} - \frac{1}{2}\right)^4 \le \frac{z_{\alpha/2}^2}{2n} \le 2\left(\frac{k' + 1}{n} - \frac{1}{2}\right)^2 + 3.2\left(\frac{k' + 1}{n} - \frac{1}{2}\right)^4.
\]
These inequalities are equivalent to
\[
\frac{k'}{n} - \frac{1}{2} \le \frac{1}{2}\sqrt{\sqrt{9 + \frac{6z_{\alpha/2}^2}{n}} - 3}\quad\Leftrightarrow\quad k' \le \frac{n}{2} + \frac{\sqrt{n}z_{\alpha/2}}{2}\sqrt{\frac{1}{(1/2)\sqrt{1 + (2/3)z_{\alpha/2}^2/n} + 1/2}},
\]
and
\[
\frac{k' + 1}{n} - \frac{1}{2} \ge \frac{1}{4}\sqrt{\sqrt{25 + \frac{80z_{\alpha/2}^2}{2n}} - 5}\quad\Leftrightarrow\quad k' \ge \frac{n}{2} + \frac{\sqrt{n}z_{\alpha/2}}{2}\sqrt{\frac{1}{(1/2)\sqrt{1 + (8/5)z_{\alpha/2}^2/n} + (1/2)}} - 1.
\]
Therefore, we conclude that if $z_{\alpha/2}^2 \ge 2nH(\lceil n/2\rceil/n, 0.5)$,
\begin{equation}\label{eq:inequalities-z-large}
\frac{n}{2}+\frac{\sqrt{n}z_{\alpha/2}}{2}\sqrt{\frac{2}{\sqrt{1 + (8/5)z_{\alpha/2}^2/n} + 1}} - 1 ~\le~ k ~\le~ \frac{n}{2} + \frac{\sqrt{n}z_{\alpha/2}}{2}\sqrt{\frac{2}{\sqrt{1 + (2/3)z_{\alpha/2}^2/n} + 1}} + 1.    
\end{equation}
Both upper and lower bounds are asymptotically the same as $n/2 + \sqrt{n}z_{\alpha/2}/2$. 
The right-hand side is less than $n/2 + \sqrt{n}z_{\alpha/2}/2 + 1$ for all $n\ge1$. To simplify the left-hand side, define $f(c) = \sqrt{2/(1 + \sqrt{1 + c})} - 1$. Clearly, $f''(c) \ge 0, f''(c) \le 2^{-3/2}/4 + (3/8)2^{-5/2}$ for all $c \ge 0$ and hence,
\[
-\frac{1}{8}c = f(0) + f'(0)c \le f(c) \le f(0) = 0.
\]
This implies
\begin{align*}
\frac{\sqrt{n}z_{\alpha/2}}{2}\left(\sqrt{\frac{2}{1 + \sqrt{1 + (8/5)z_{\alpha/2}^2/n}}} - 1\right) &\ge \frac{\sqrt{n}z_{\alpha/2}}{2}\left(-\frac{1}{8}\frac{8}{5}\frac{z_{\alpha/2}^2}{n}\right) = -\frac{1}{10}\frac{z_{\alpha/2}^3}{\sqrt{n}}.
\end{align*}
Therefore, we conclude that if $z_{\alpha/2}^2 \ge 2nH(\lceil n/2\rceil/n, 0.5)$, then
\begin{equation}\label{eq:final_ineq_n-large}
\frac{n}{2} + \frac{\sqrt{n}z_{\alpha/2}}{2} - \frac{z_{\alpha/2}^3}{10\sqrt{n}} - 1 \le k \le \frac{n}{2} + \frac{\sqrt{n}z_{\alpha/2}}{2} + 1.
\end{equation}
If, on the other hand, $z_{\alpha/2}^2 \le 2nH(\lceil n/2\rceil/n, 0.5)$, then
\[
\lfloor n/2\rfloor \le k \le \lfloor n/2\rfloor + 1\quad\Rightarrow\quad \frac{n}{2} - 1 \le k \le \frac{n}{2} + 1.
\]
From \Cref{lem:kl_bound} and the fact that $\lceil n/2\rceil \le (n + 1)/2$, we get
\[
H\left(\frac{\lceil n/2\rceil}{n},\,\frac{1}{2}\right) \le 2\left(\frac{1}{2} + \frac{1}{2n} - \frac{1}{2}\right)^2 + (16\ln(2) - 8)\left(\frac{1}{2} + \frac{1}{2n} - \frac{1}{2}\right)^4 \le \frac{1}{2n^2} + (\ln(2) - 1/2)\frac{1}{n^4}.
\]
Therefore, $z_{\alpha/2}^2 \le 2nH(\lceil n/2\rceil/n, 0.5)$ implies $nz_{\alpha/2}^2 \le n^{-1} + (2\ln(2) - 1)n^{-3}$. This further implies that if $z_{\alpha/2}^2 \le 2nH(\lceil n/2\rceil/n, 0.5)$, then
\begin{equation}\label{eq:inequalities-z-small}
\frac{n}{2} + \frac{\sqrt{n}z_{\alpha/2}}{2} - \frac{\sqrt{n^{-1} + (2\ln(2) - 1)n^{-3}}}{2} - 1 \le k \le \frac{n}{2} + \frac{\sqrt{n}z_{\alpha/2}}{2} + 1.    
\end{equation}
Combining inequalities~\eqref{eq:final_ineq_n-large} and~\eqref{eq:inequalities-z-small}, we get for all $n\ge\log_2(2/\alpha)$,
\begin{equation}
\label{eq:hd_new_bound}
-\frac{1}{2\sqrt{n}}\max\left\{\frac{z_{\alpha/2}^3}{5},\,\sqrt{1 + \frac{2\ln(2) - 1}{n^2}}\right\} - 1 \le k - \left(\frac{n}{2} + \frac{\sqrt{n}z_{\alpha/2}}{2}\right) \le 1.
\end{equation}

For $n \le \log_2(2/\alpha)$, $k = n$. In this case, the confidence interval is anyway $\mathbb{R}$ and there is no need to compare $k$ to the quantile from normal approximation. Using the fact that $c_{n,\alpha}=k-\lceil{n/2\rceil}$, we have for all $n\geq \log_2(2/\alpha)$, 
\[
-\frac{1}{2\sqrt{n}}\max\left\{\frac{z_{\alpha/2}^3}{5},\,\sqrt{1 + \frac{2\ln(2) - 1}{n^2}}\right\} - 1.5 \le c_{n,\alpha} - \frac{\sqrt{n}z_{\alpha/2}}{2} \le 1.
\]
This completes the proof of the theorem. 

\section{Proof of Theorem~\ref{thm:bahadur_fin_sample}}
\label{appendix:theorem_bahadur_finsample}
We state and prove a stronger version of the result stated in \Cref{thm:bahadur_fin_sample}.
\begin{theorem}
\label{thm:bahadur_fin_sample_1}
Let $X_1,X_2,\cdots,X_n \stackrel{iid}{\sim} F $. Suppose that $F$ is a continuous CDF with median $\theta_0$. We assume that the following holds for the distribution function $F$,
\begin{equation*}
\label{asump:fin_bahadur_1}
|F(\theta_0+h)-F(\theta_0)-Mh| \leq C|h|^{1+\delta} \quad\mbox{for all}\quad |h|<\eta,    
\end{equation*}
where $0<M,C,\delta,\eta<\infty$. Define
\begin{equation*}
    \begin{split}
     \zeta &:=(M/2)\min\{\eta,(M/2C)^{1/\delta}\}  , \\
  A_n &:= \frac{z_{\alpha/2}}{2\sqrt{n}}+\frac{2}{n}+\sqrt{\frac{\log (n)}{n+1}} + \left(\frac{z_{\alpha/2}}{\sqrt{n}} + \frac{4}{n}\right)\frac{2\log (n)}{n+1}.   
    \end{split}
\end{equation*}
Then for every sample size $n\geq \log_2(2/\alpha)$, we have the following with probability greater than or equal to $(1-6n^{-2})\mathbbm{1}\{A_n<\zeta\}$, 
\begin{align*}
    \left|\mathrm{Width}(\widehat{\mathrm{CI}}_{n,\alpha}) - \frac{z_{\alpha/2}}{\sqrt{n}M}\right| &\leq \frac{1}{M}\left(\frac{0.25z_{\alpha/2}^2+5.18}{n+1} + \sqrt{\frac{4(4+\sqrt{n}z_{\alpha/2})\log n}{(n+1)^2}}+ \frac{2\log n}{n+1}\right)\\ 
    &+ \left(\frac{2}{M}\right)^{2+\delta}CA_n^{1+\delta} .
\end{align*}
\end{theorem}
Let us see how the simpler version presented in \Cref{thm:bahadur_fin_sample} follows from this more general statement. We shall first obtain a bound on $A_n$. 
\begin{equation*}
    \begin{split}
        A_n &= \frac{z_{\alpha/2}}{2\sqrt{n}}+\frac{2}{n}+\sqrt{\frac{\log (n)}{n+1}} + \left(\frac{z_{\alpha/2}}{\sqrt{n}} + \frac{4}{n}\right)\frac{2\log (n)}{n+1} \\
        &= \frac{z_{\alpha/2}}{\sqrt{n}}\left( \frac{1}{2} + \frac{2 \log(n)}{n+1} \right) + \sqrt{\frac{log n}{n+1}} \left(\frac{2}{n}\sqrt{\frac{n+1}{\log n}} + 1 + \frac{8}{n}\sqrt{\frac{\log n}{n+1}} \right) \\
        &\leq 1.1 \frac{z_{\alpha/2}}{\sqrt{n}} + 5.1 \sqrt{\frac{\log n}{n + 1}} \\
        & \leq \sqrt{\frac{\log(2n/\alpha)}{n}}\left( \frac{1.1 z_{\alpha/2}}{\sqrt{\log(2n/\alpha)}} + 5.1 \sqrt{\frac{\log n}{\log(2n/\alpha}} \right) \\
        &\leq \sqrt{\frac{\log(2n/\alpha)}{n}} \left( \frac{1.1 \sqrt{2 \log(2/\alpha)}}{\sqrt{\log(4/\alpha)}} + 5.1 \sqrt{\frac{\log n}{\log(2n}} \right)   \\
        & \leq \sqrt{\frac{\log(2n/\alpha)}{n}} (1.1\sqrt{2} + 5.1 ) \\
        &\leq 7 \sqrt{\frac{\log(2n/\alpha)}{n}}.
    \end{split}
\end{equation*}
Thus the condition $A_n < \zeta$ will be satisfied if $7\sqrt{\log(2n/\alpha)/n} < \zeta$ i.e.\ $n \geq (49/\zeta^2)\log(2n/\alpha)$. Thus if $n \geq \max\{\log_2(2/\alpha), (49/\zeta^2)\log(2n/\alpha) \}$, then the following holds with probability at-least $1 - 6n^{-2}$, 
\begin{equation*}
    \begin{split}
        &\left|\frac{\sqrt{n}M}{z_{\alpha/2}}\mathrm{Width}(\widehat{\mathrm{CI}}_{n,\alpha}) - 1\right| \\
        \leq& \frac{\sqrt{n}}{z_{\alpha/2}}\left(\frac{0.25z_{\alpha/2}^2+5.18}{n+1} + \sqrt{\frac{4(4+\sqrt{n}z_{\alpha/2})\log n}{(n+1)^2}}+ \frac{2\log n}{n+1}\right) + \frac{\sqrt{n}}{z_{\alpha/2}}\frac{2^{2+\delta}}{M^{1+\delta}}CA_n^{1+\delta} \\
        \leq & \frac{1}{\sqrt{n}z_{\alpha/2}} \left\{n^{1/4}z_{\alpha/2}\left( 1+ \frac{0.25z_{\alpha/2}}{n^{1/4}} \right) + n^{1/4}\log(n)\left( 1+ \frac{2\log(n) + 4\sqrt{\log(n)} + 5.18}{n^{1/4}\log(n)} \right) \right\} \\
        & + \frac{\sqrt{n}}{z_{\alpha/2}}\frac{2^{2+\delta}}{M^{1+\delta}}CA_n^{1+\delta} \\
        \leq & \frac{1}{n^{1/4}} \left( 1 + \frac{0.25 z_{\alpha/2}}{n^{1/4}} + \frac{13.1 \log(n)}{z_{\alpha/2}} \right) + \frac{\sqrt{n}}{z_{\alpha/2}}\frac{2^{2+\delta}}{M^{1+\delta}}C \left(7 \sqrt{\frac{\log(2n/\alpha)}{n}} \right)^{1+\delta} \\
        \leq & \frac{1 + 14(\log(n)/z_{\alpha/2})}{n^{1/4}} + \sqrt{\frac{\log(2/\alpha)}{8n}} + \frac{2C (14)^{1+\delta}(\log(2n/\alpha))^{(1+\delta)/2}}{z_{\alpha/2}M^{1+\delta}n^{\delta/2}}.
    \end{split}
\end{equation*}
This is the simpler version which has been stated in \Cref{thm:bahadur_fin_sample}. 

\begin{proof}[Proof of \Cref{thm:bahadur_fin_sample_1}]
We first state and prove the following lemma. 
\begin{lemma}
\label{lem:order_stat_behav}
Let $X_1,X_2,\cdots,X_n \stackrel{iid}{\sim} F $. Suppose that $F$ is a continuous CDF with median $\theta_0$. We assume that the following holds for the distribution function $F$,
\[
|F(\theta_0+h)-F(\theta_0)-Mh| \leq C|h|^{1+\delta} \quad\mbox{for all}\quad |h|<\eta,
\]
where $0<M,C,\delta, \eta<\infty$. Then for every sample size $n\geq \log_2(2/\alpha)$, we have the following, 
\begin{equation}
\begin{split}
    \mathbb{P}\left(\left|X_{(k_{n,\alpha}+1)}-\theta_0 \right| \leq \frac{2A_n}{M} \right) &\geq (1- 2n^{-2})\mathbbm{1}\{A_n < \zeta\},\\
    \mathbb{P}\left(\left|X_{(n-k_{n,\alpha})}-\theta_0 \right| \leq \frac{2B_n}{M} \right) &\geq (1- 2n^{-2})\mathbbm{1}\{B_n < \zeta \},
\end{split}
\end{equation}
where $\zeta=(M/2)\min\{\eta,(M/2C)^{1/\delta}\}$ and $A_n,B_n$ are as defined in \Cref{lem:conc_anbn}.
\end{lemma}

\begin{proof}[Proof of \Cref{lem:order_stat_behav}]
We start by observing that for $0<\epsilon<\eta$ the following inequality holds,
\begin{equation*}
    \begin{split}
        \mathbb{P}\left(\left|X_{(k_{n,\alpha}+1)}-\theta_0 \right| \geq \epsilon \right) &= \mathbb{P}\left(X_{(k_{n,\alpha}+1)} \geq \theta_0+\epsilon \right) + \mathbb{P}\left(X_{(k_{n,\alpha}+1)} \leq \theta_0-\epsilon \right) \\
        &= \mathbb{P}\left(F(X_{(k_{n,\alpha}+1)}) \geq F(\theta_0+\epsilon) \right) + \mathbb{P}\left(F(X_{(k_{n,\alpha}+1)}) \leq F(\theta_0-\epsilon) \right) \\
        &= \mathbb{P}\left(F(X_{(k_{n,\alpha}+1)})-F(\theta_0) \geq F(\theta_0+\epsilon)-F(\theta_0) \right) \\
        & \quad + \mathbb{P}\left(F(X_{(k_{n,\alpha}+1)})-F(\theta_0) \leq F(\theta_0-\epsilon)-F(\theta_0) \right) \\
        &\leq \mathbb{P}\left(\left|F(X_{(k_{n,\alpha}+1)})-F(\theta_0) \right| \geq \min\{F(\theta_0+\epsilon)-F(\theta_0),F(\theta_0-\epsilon)-F(\theta_0)  \} \right)\\
        &\leq \mathbb{P}\left(\left|F(X_{(k_{n,\alpha}+1)})-F(\theta_0) \right| \geq M\epsilon -C\epsilon^{1+\delta} \right).
    \end{split}
\end{equation*}
We observe that if $\epsilon<(M/2C)^{1/\delta}$ then $M\epsilon -C\epsilon^{1+\delta} > M\epsilon - M\epsilon/2 = M\epsilon/2$. Thus if $0<\epsilon<\min\{\eta, (M/2C)^{1/\delta} \}$ we can say the following,
\[
\mathbb{P}\left(\left|X_{(k_{n,\alpha}+1)}-\theta_0 \right| \leq \epsilon \right) \geq \mathbb{P}\left(\left|F(X_{(k_{n,\alpha}+1)})-F(\theta_0) \right| \leq M\epsilon/2 \right).
\]
Setting $M\epsilon/2=A_n$ (i.e.,\ $\epsilon=2A_n/M$) we obtain the following provided $n\geq \log_2(2/\alpha)$ and $A_n<\zeta$ where $\zeta=(M/2)\min\{\eta,(M/2C)^{1/\delta}\}$,
\begin{equation*}
\begin{split}
\mathbb{P}\left(\left|X_{(k_{n,\alpha}+1)}-\theta_0 \right| \leq \frac{2A_n}{M} \right) &\geq \mathbb{P}\left(\left|F(X_{(k_{n,\alpha}+1)})-F(\theta_0) \right| \leq A_n \right)    \\
&\geq 1-2n^{-2}.
\end{split}    
\end{equation*}
The last inequality in the above derivation follows from \Cref{lem:conc_anbn}. This completes the proof of the first concentration inequality. The proof of the second concentration inequality follows exactly the same path.
\end{proof}
From \Cref{lem:order_stat_behav} we know that with probability greater than or equal to $(1-2n^{-2})\mathbbm{1}\{A_n<\zeta\}$, $\left|X_{(k_{n,\alpha}+1)}-\theta_0 \right| \leq 2A_n/M$. We observe that $A_n<\zeta$ implies that $A_n<(M/2)\eta$ (i.e.\ $2A_n/M<\eta$) which in turn implies that with probability greater than or equal to $(1-2n^{-2})\mathbbm{1}\{A_n<\zeta\}$, $\left|X_{(k_{n,\alpha}+1)}-\theta_0 \right| \leq \eta$. Therefore using the assumption on $F$ (for $|h|<\eta$) mentioned in the theorem we obtain the following for all $n\geq \log_2(2/\alpha)$,
\[
\begin{cases}
    \mathbb{P}\left(\left|F(X_{(k_{n,\alpha}+1)})-F(\theta_0)- M(X_{(k_{n,\alpha}+1)}-\theta_0) \right| \leq C\left|X_{(k_{n,\alpha}+1)}-\theta_0 \right|^{1+\delta}\right) \geq (1-2n^{-2})\mathbbm{1}\{A_n<\zeta\},\\
    \mathbb{P}\left(\left|F(X_{(n-k_{n,\alpha})})-F(\theta_0)- M(X_{(n-k_{n,\alpha})}-\theta_0) \right| \leq C\left|X_{(n-k_{n,\alpha})}-\theta_0 \right|^{1+\delta}\right) \geq (1-2n^{-2})\mathbbm{1}\{A_n<\zeta\}.
\end{cases}
\]
This implies that with probability greater than or equal to $(1-2n^{-2})\mathbbm{1}\{A_n<\zeta\}$ we have for all $n\geq \log_2(2/\alpha)$,
\begin{equation*}
    \begin{split}
& \frac{1}{M}\{F(X_{(k_{n,\alpha}+1)})-F(\theta_0) \}- \frac{C}{M}\left|X_{(k_{n,\alpha}+1)}-\theta_0 \right|^{1+\delta} \\
\leq & X_{(k_{n,\alpha}+1)}-\theta_0 \\
\leq & \frac{1}{M}\{F(X_{(k_{n,\alpha}+1)})-F(\theta_0)  \} +\frac{C}{M}\left|X_{(k_{n,\alpha}+1)}-\theta_0 \right|^{1+\delta} ,       
    \end{split}
\end{equation*}
and with probability greater than or equal to $(1-2n^{-2})\mathbbm{1}\{B_n<\zeta\}$ we have for all $n\geq \log_2(2/\alpha)$,
\begin{equation*}
    \begin{split}
& \frac{1}{M}\{F(X_{(n-k_{n,\alpha})})-F(\theta_0) \}- \frac{C}{M}\left|X_{(n-k_{n,\alpha})}-\theta_0 \right|^{1+\delta} \\
\leq & X_{(k_{n,\alpha}+1)}-\theta_0 \\
\leq & \frac{1}{M}\{F(X_{(n-k_{n,\alpha})})-F(\theta_0)\} +\frac{C}{M}\left|X_{(n-k_{n,\alpha})}-\theta_0 \right|^{1+\delta}.      
    \end{split}
\end{equation*}
Subtracting the second equation from the first and further subtracting $z_{\alpha/2}/(\sqrt{n}M)$ from all sides we obtain that with probability $(1-4n^{-2})\mathbbm{1}\{ \max\{A_n,B_n \}<\zeta\}$ (i.e.\ with probability $(1-4n^{-2})\mathbbm{1}\{A_n <\zeta\}$ as $B_n<A_n$) the following holds for all $n\geq \log_2(2/\alpha)$,
\begin{equation*}
    \begin{split}
&\frac{1}{M}\{F(X_{(k_{n,\alpha}+1)})-F(X_{(n-k_{n,\alpha})}) -\frac{z_{\alpha/2}}{\sqrt{n}} \} - \frac{C}{M}\left|X_{(k_{n,\alpha}+1)}-\theta_0 \right|^{1+\delta}-\frac{C}{M}\left|X_{(n-k_{n,\alpha})}-\theta_0 \right|^{1+\delta}\\
 \leq & X_{(k_{n,\alpha}+1)}-X_{(n-k_{n,\alpha})} -\frac{z_{\alpha/2}}{\sqrt{n}M}\\
  \leq &\frac{1}{M}\{F(X_{(k_{n,\alpha}+1)})-F(X_{(n-k_{n,\alpha})}) -\frac{z_{\alpha/2}}{\sqrt{n}} \} +\frac{C}{M}\left|X_{(k_{n,\alpha}+1)}-\theta_0 \right|^{1+\delta}  +\frac{C}{M}\left|X_{(n-k_{n,\alpha})}-\theta_0 \right|^{1+\delta}.      
    \end{split}
\end{equation*}
We take modulus and apply \Cref{lem:width_distr} and \Cref{lem:order_stat_behav}. We get that the following event occurs with probability $(1-6n^{-2})\mathbbm{1}\{A_n <\zeta\}$ for all $n\geq \log_2(2/\alpha)$,
\begin{equation*}
    \begin{split}
       &\quad \left|X_{(k_{n,\alpha}+1)}-X_{(n-k_{n,\alpha})} -\frac{z_{\alpha/2}}{\sqrt{n}M} \right| \\
       &\leq \frac{1}{M}\left| F(X_{(k_{n,\alpha}+1)})-F(X_{(n-k_{n,\alpha})}) -\frac{z_{\alpha/2}}{\sqrt{n}} \right|+\frac{C}{M}\left|X_{(k_{n,\alpha}+1)}-\theta_0 \right|^{1+\delta}  +\frac{C}{M}\left|X_{(n-k_{n,\alpha})}-\theta_0 \right|^{1+\delta} \\
       &\leq \frac{1}{M}\left(\frac{0.25z_{\alpha/2}^2+5.18}{n+1} + \sqrt{\frac{4(4+\sqrt{n}z_{\alpha/2})\log n}{(n+1)^2}}+ \frac{2\log n}{n+1}\right) + \frac{C}{M}\left(\left( \frac{2A_n}{M}\right)^{1+\delta}+ \left( \frac{2B_n}{M}\right)^{1+\delta}\right) \\
       &\leq \frac{1}{M}\left(\frac{0.25z_{\alpha/2}^2+5.18}{n+1} + \sqrt{\frac{4(4+\sqrt{n}z_{\alpha/2})\log n}{(n+1)^2}}+ \frac{2\log n}{n+1}\right)+ \left(\frac{2}{M}\right)^{2+\delta}CA_n^{1+\delta}.
    \end{split}
\end{equation*}
Note that the probability of occurrence of the above event is  $(1-6n^{-2})\mathbbm{1}\{A_n <\zeta\}$ because of the concentration inequality mentioned in \Cref{lem:width_distr} and because of the fact $\mathbb{P}(A \cap B)=\mathbb{P}(A)+\mathbb{P}(B)-\mathbb{P}(A \cup B) \geq \mathbb{P}(A)+\mathbb{P}(B)- 1$. This completes the proof of the theorem.
\end{proof}
We also provide an asymptotic version of \Cref{thm:bahadur_fin_sample}. \Cref{thm:bahadur_asump} stated below summarizes the asymptotic behavior of $\mathrm{WR}_{n,\alpha}$ as $n \rightarrow \infty$ under standard Bahadur assumptions.
\begin{theorem}
\label{thm:bahadur_asump}
Let $X_1,X_2,...,X_n \stackrel{iid}{\sim} F$. Assume that $F$ is continuously differentiable at the population median $\theta_0$ with $F'(\theta_0)>0$. Then for any $\alpha\in[0, 1]$, as $n\to\infty$,
\begin{equation}
    \mathrm{WR}_{n,\alpha}= 1 +o_p(1).
\end{equation}
Hence, the width of $\widehat{\mathrm{CI}}_{n,\alpha}$ defined in \eqref{eq:rewritten-CI} is asymptotically equal to the width of the Wald confidence interval of the median.
\end{theorem}
\begin{proof}[Proof of \Cref{thm:bahadur_asump}]
\label{appendix:theorem_bahadur_asump}
We have already seen that if the sample size $n \geq \log_2(2/\alpha)$, the confidence interval described in \Cref{alg:proposed-conf-int} is $\widehat{\mathrm{CI}}_{n,\alpha}=\left[X_{(\floor{{n}/{2}}-c_{n,\alpha})}, X_{(\ceil{{n}/{2}}+c_{n,\alpha}+1)}\right]$. We analyse the width of this confidence interval as follows,
\begin{equation}
\label{eq:analysing_width}
    \begin{split}
    \mathrm{Width}(\widehat{\mathrm{CI}}_{n,\alpha}) &= X_{(\ceil{\frac{n}{2}}+c_{n,\alpha}+1)} - X_{(\floor{\frac{n}{2}}-c_{n,\alpha})} \\
     &= F_n^{-1}\bigg(\frac{1}{2} + \frac{c_{n,\alpha}+1}{n} \bigg) - F_n^{-1}\bigg(\frac{1}{2} - \frac{c_{n,\alpha}}{n} \bigg) \\
    &= F^{-1}(1/2) + \frac{c_{n,\alpha}/n}{F'(F^{-1}(1/2))} + \frac{(1/2)-F_n(F^{-1}(1/2)}{F'(F^{-1}(1/2))} \\
    & \text{          } -F^{-1}(1/2) + \frac{c_{n,\alpha}/n}{F'(F^{-1}(1/2))} - \frac{(1/2)-F_n(F^{-1}(1/2)}{F'(F^{-1}(1/2))}+o_p(n^{-1/2}) \\
    &= \frac{2c_{n,\alpha}}{nF'(F^{-1}(1/2))}+o_p(n^{-1/2}).
    \end{split}
\end{equation}
\Cref{thm:assym_of_cn} states that,
\[
 -\frac{1}{2\sqrt{n}}\max\left\{\frac{z_{\alpha/2}^3}{5},\,\sqrt{1 + \frac{2\ln(2) - 1}{n^2}}\right\} - 1.5 \le c_{n,\alpha} - \frac{\sqrt{n}z_{\alpha/2}}{2} \le 1. 
\]
This implies that,
\begin{equation*}
    \begin{split}
& -\frac{1}{n^{3/2}F'(F^{-1}(1/2))}\max\left\{\frac{z_{\alpha/2}^3}{5},\,\sqrt{1 + \frac{2\ln(2) - 1}{n^2}}\right\} - \frac{3}{nF'(F^{-1}(1/2))} \\
\leq & \frac{2c_{n,\alpha}}{nF'(F^{-1}(1/2))} - \frac{z_{\alpha/2}}{\sqrt{n}F'(F^{-1}(1/2))} \\
\leq & \frac{2}{nF'(F^{-1}(1/2))}.        
    \end{split}
\end{equation*}
Thus we can say that,
\[
\frac{2c_{n,\alpha}}{nF'(F^{-1}(1/2))}=\frac{z_{\alpha/2}}{\sqrt{n}F'(F^{-1}(1/2))}+o_p(n^{-1/2}).
\]
Using this in \eqref{eq:analysing_width} and the fact that $F^{-1}(1/2) = \theta_0$, we obtain,
\begin{equation}
     \mathrm{Width}(\widehat{\mathrm{CI}}_{n,\alpha})= \frac{z_{\alpha/2}}{\sqrt{n}F'(\theta_0)}+o_p(n^{-1/2}).
\end{equation}
Note that $z_{\alpha/2}/(\sqrt{n}F'(\theta_0))$ is the width of the Wald confidence interval. This completes the proof of this theorem.
\end{proof}

\section{Proof of Theorem~\ref{thm:asym_result_irreg_case}}
\label{appendix_thm_asym_irreg}
We state and prove two theorems which are stronger versions of the result stated in \Cref{thm:asym_result_irreg_case}.
\begin{theorem}
\label{thm:fin_sample_non_bahadur}
Let $X_1,X_2,\cdots,X_n \stackrel{iid}{\sim} F $. Suppose that $F$ is a continuous CDF with median $\theta_0$. We assume that the following holds for the distribution function $F$,
\begin{equation}
\label{asump:fin_non_bahadur}
|F(\theta_0+h)-F(\theta_0)-M|h|^{\rho}\mathrm{sgn}(h)| \leq C|h|^{\rho+\Delta} \quad\mbox{for all}\quad |h|<\eta,   
\end{equation}
where $0<M,C,\rho,\Delta,\eta<\infty$. We define the following quantities,
\begin{equation*}
    \begin{split}
        k_{n,\alpha} &= c_{n,\alpha} + \ceil{n/2},\\
        \delta &=\Delta/\rho, \\
        A_n  &= \frac{z_{\alpha/2}}{2\sqrt{n}}+\frac{2}{n}+\sqrt{\frac{\log (n)}{n+1}} + \left(\frac{z_{\alpha/2}}{\sqrt{n}} + \frac{4}{n}\right)\frac{2\log (n)}{n+1},\\
       \zeta &=(M/2)\min\{\eta^{\rho},(M/2C)^{1/\delta}\}, \\
       \kappa(k/n,2,n) &= \max \left\{\sqrt{21\frac{k}{n}\left( 1-\frac{k}{n}\right)},21\sqrt{\frac{\log(n)}{n}} \right\},\\
  \Tilde{\kappa}(k/n,2,n) &= \max \left\{\sqrt{12\frac{k}{n+1}\left( 1-\frac{k}{n+1}\right)},12\sqrt{\frac{\log(n)}{n}} \right\} + \frac{1}{n}, \\
  K_n &= \{\kappa((k_{n,\alpha}+1)/n,2,n) >\Tilde{\kappa}((k_{n,\alpha}+1)/n,2,n) \} \\
  & \quad \cap \{ \kappa((n-k_{n,\alpha})/n,2,n) >\Tilde{\kappa}((n-k_{n,\alpha})/n,2,n)\}.
    \end{split}
\end{equation*}
Then we can say that with probability greater than or equal to $(1-(4+2D_2)n^{-2})\mathbbm{1}\{A_n < \zeta \}\mathbbm{1}\{K_n\}$ the following event occurs for all $n \geq \log_2(2/\alpha)$,
\begin{equation*}
    \begin{split}
 &\quad\frac{1}{M^{1/\rho}}\{|Q_{1,\boldsymbol{X},n}-\gamma_n|^{1/\rho}\mathrm{sgn}(Q_{1,\boldsymbol{X},n}-\gamma_n) -  |Q_{2,\boldsymbol{X},n}+\gamma_n|^{1/\rho}\mathrm{sgn}(Q_{2,\boldsymbol{X},n}+\gamma_n)     \} \\
 &\leq \mathrm{Width}(\widehat{\mathrm{CI}}_{n,\alpha}) \\
 &\leq \frac{1}{M^{1/\rho}}\{|Q_{1,\boldsymbol{X},n}+\gamma_n|^{1/\rho}\mathrm{sgn}(Q_{1,\boldsymbol{X},n}+\gamma_n)- |Q_{2,\boldsymbol{X},n}-\gamma_n|^{1/\rho}\mathrm{sgn}(Q_{2,\boldsymbol{X},n}-\gamma_n)\},
    \end{split}
\end{equation*}   
where,
\begin{equation*}
    \begin{split}
 Q_{1,\boldsymbol{X},n} &= \left(\frac{k_{n,\alpha}+1}{n} - \frac{1}{2}\right) - \frac{1}{n}\sum_{i=1}^n \mathbf{1}\{F(X_i) \le (k_{n,\alpha}+1)/n\} + \frac{k_{n,\alpha}+1}{n}, \\
 Q_{2,\boldsymbol{X},n}&= \left(\frac{n-k_{n,\alpha}}{n} - \frac{1}{2}\right) - \frac{1}{n}\sum_{i=1}^n \mathbf{1}\{F(X_i) \le (n-k_{n,\alpha})/n\} + \frac{n-k_{n,\alpha}}{n},\\
 D_2 &= 2+ e^{13/2}, \\
 \gamma_n &=\max\left\{\frac{35}{\sqrt{2}},\,\, 35\sqrt{\frac{35\log(n)}{n}} \right\}\left(\frac{\log(n)}{n}\right)^{3/4}+C\left(\frac{2A_n}{M}\right)^{1+\delta}.
    \end{split}
\end{equation*}
\end{theorem}
\begin{proof}[Proof of \Cref{thm:fin_sample_non_bahadur}]
We define $H(t)$ for $t \in \mathbb{R}$ as follows,
\[
H(t)=F(\theta_0+|t|^{1/\rho} \mathrm{sgn}(t)).
\]
It is easy to see that $H(\cdot)$ is a distribution function with median $0$ i.e.\ $H(0)=F(\theta_0)=1/2$. Therefore by substituting $|t|^{1/\rho} \mathrm{sgn}(t)$ in place of $h$ in assumption \eqref{asump:fin_non_bahadur}, the assumption on $F(\cdot)$ translates into the following assumption on $H(\cdot)$, 
\begin{equation}
    \label{eq:asump_H}
|H(t)-H(0)-Mt| \leq C|t|^{1+\delta} \quad\mbox{for all}\quad |t|<\eta^{\rho},      
\end{equation}
where $0<\delta=\Delta/\rho<\infty$ and $0<M,C,\rho,\eta<\infty$. Assumption \eqref{eq:asump_H} is identical to the assumption made in \Cref{lem:order_stat_behav}. We note that if $X \sim F$ then $H^{-1}(F(X)) \sim H$. Since $H^{-1}(y)=|F^{-1}(y)-\theta_0|^{\rho} \mathrm{sgn}(F^{-1}(y)-\theta_0)$, we obtain that $H^{-1}(F(X))=|X-\theta_0|^{\rho} \mathrm{sgn}(X-\theta_0)$. Therefore if $X \sim F$ then $|X-\theta_0|^{\rho} \mathrm{sgn}(X-\theta_0) \sim H$. We now apply \Cref{lem:order_stat_behav} for the distribution function $H(\cdot)$ on the transformed random variables $|X_i-\theta_0|^{\rho} \mathrm{sgn}(X_i-\theta_0)$ for $i=1,\cdots,n$ keeping in mind that this is an increasing transformation (which implies that $(|X-\theta_0|^{\rho} \mathrm{sgn}(X-\theta_0))_{(k)}=|X_{(k)}-\theta_0|^{\rho} \mathrm{sgn}(X_{(k)}-\theta_0)$),
\begin{equation}
\label{eq:non_bahadur_order_stat}
    \begin{split}
    \mathbb{P}\left(\left|X_{(k_{n,\alpha}+1)}-\theta_0 \right|^{\rho} \leq \frac{2A_n}{M} \right) &\geq (1- 2n^{-2})\mathbbm{1}\{A_n < \zeta\} \quad\quad \mbox{for} \quad\quad n\geq \log_2(2/\alpha),\\
   \mathbb{P}\left(\left|X_{(n-k_{n,\alpha})}-\theta_0 \right|^{\rho} \leq \frac{2B_n}{M} \right) &\geq (1- 2n^{-2})\mathbbm{1}\{B_n < \zeta\} \quad\quad \mbox{for} \quad\quad n\geq \log_2(2/\alpha)  ,  
    \end{split}
\end{equation}
where $\zeta=(M/2)\min\{\eta^{\rho},(M/2C)^{1/\delta}\},A_n,B_n$ are as defined in \Cref{lem:order_stat_behav}. We set $Y_i = F(X_i), 1\le i\le n$ which implies that $Y_{(1)}, \ldots, Y_{(n)}$ are uniform order statistics. Theorem $6.3.1$ of \cite{reiss2012approximate} implies that,
\[
\mathbb{P}\left(\left|Y_{(k)} - \frac{k}{n} + \frac{1}{n}\sum_{i=1}^n \mathbf{1}\{Y_i \le k/n\} - \frac{k}{n}\right| \ge \left(\frac{\log n}{n}\right)^{3/4}\Tilde{\delta}\left(\frac{k}{n},2,n\right)\right) \le \frac{C(2,n)}{n^{2}},
\]
where,
\begin{equation*}
    \begin{split}
\Tilde{\delta}\left(\frac{k}{n},2,n\right) &= 35\max\left\{ \left( \frac{k}{n}\left( 1-\frac{k}{n}\right)\right)^{1/4},\,\, \sqrt{35\frac{\log(n)}{n}}\right\} \\
 &\leq 35\max\left\{\frac{1}{4^{1/4}},\,\,\sqrt{35\frac{\log(n)}{n}}  \right\} \\
 &= \max\left\{\frac{35}{\sqrt{2}},\,\, 35\sqrt{\frac{35\log(n)}{n}} \right\} \\
 &= D_{1,n} \quad \mbox{(say)},
    \end{split}
\end{equation*}
and,
\begin{equation*}
    \begin{split}
        \frac{C(2,n)}{n^2} &= \frac{A(2,n)}{n^2}+\frac{B(2,n)}{n^2} \\
        &= (n+2)^2\exp\{-5\log(n)+(3/4)\log(n)+(13/2)\} + \left(\frac{2}{n^2}-1\right)\mathbbm{1}\{\kappa(k/n,2,n) >\Tilde{\kappa}(k/n,2,n) \} +1 \\
        &\leq \frac{e^{13/2}}{n^2}\left(1+ \left(\frac{(n+2)^2}{n^{9/4}} -1\right)\mathbbm{1}\{n<8\} \right)  + \left(\frac{2}{n^2}-1\right)\mathbbm{1}\{\kappa(k/n,2,n) >\Tilde{\kappa}(k/n,2,n) \} +1 \\
        &\leq 1+ \left\{\frac{1}{n^2}\left(2+ e^{13/2}\left(1+ \left(\frac{(n+2)^2}{n^{9/4}} -1\right)\mathbbm{1}\{n<8\} \right) \right)-1\right\}\mathbbm{1}\{\kappa(k/n,2,n) >\Tilde{\kappa}(k/n,2,n) \},
    \end{split}
\end{equation*}
where,
\begin{equation*}
    \begin{split}
  \kappa(k/n,2,n) &= \max \left\{\sqrt{21\frac{k}{n}\left( 1-\frac{k}{n}\right)},21\sqrt{\frac{\log(n)}{n}} \right\},\\
  \Tilde{\kappa}(k/n,2,n) &= \max \left\{\sqrt{12\frac{k}{n+1}\left( 1-\frac{k}{n+1}\right)},12\sqrt{\frac{\log(n)}{n}} \right\} + \frac{1}{n}.
    \end{split}
\end{equation*}
Therefore we have,
\begin{equation*}
    \begin{split}
 & \quad \mathbb{P}\left(\left|Y_{(k)} - \frac{k}{n} + \frac{1}{n}\sum_{i=1}^n \mathbf{1}\{Y_i \le k/n\} - \frac{k}{n}\right| \ge \left(\frac{\log n}{n}\right)^{3/4}D_{1,n}\right) \\
 &\le 1+ \left\{\frac{1}{n^2}\left(2+ e^{13/2}\left(1+ \left(\frac{(n+2)^2}{n^{9/4}} -1\right)\mathbbm{1}\{n<8\} \right) \right)-1 \right\}\mathbbm{1}\{\kappa(k/n,2,n) >\Tilde{\kappa}(k/n,2,n) \},       
    \end{split}
\end{equation*}
which is the same as saying,
\begin{equation}
    \label{eq:reiss_fin_bahadur}
\begin{split}
& \mathbb{P}\left(\left|Y_{(k)} - \frac{k}{n} + \frac{1}{n}\sum_{i=1}^n \mathbf{1}\{Y_i \le k/n\} - \frac{k}{n}\right| \ge \left(\frac{\log n}{n}\right)^{3/4}D_{1,n}\right) \\
\le & 1+ \left(\frac{D_2}{n^2}-1\right)\mathbbm{1}\{\kappa(k/n,2,n) >\Tilde{\kappa}(k/n,2,n) \},
\end{split}
\end{equation}
where $D_2 = 2+ e^{13/2}$.
Rewritten in terms of $X_i$'s and using $F(\theta_0) = 1/2$, this is equivalent to
\begin{equation*}\label{eq:Finite-sample-Bahadur-representation}
\begin{split}
&\quad\mathbb{P}\left(\left|F(X_{(k)}) - F(\theta_0) - \left(\frac{k}{n} - \frac{1}{2}\right) + \frac{1}{n}\sum_{i=1}^n \mathbf{1}\{F(X_i) \le k/n\} - \frac{k}{n}\right| \ge D_{1,n}\left(\frac{\log n}{n}\right)^{3/4}\right) \\
&\le 1+ \left(\frac{D_2}{n^2}-1 \right)\mathbbm{1}\{\kappa(k/n,2,n) >\Tilde{\kappa}(k/n,2,n) \}.    
\end{split}
\end{equation*}
We apply this result on the transformed order statistics $T_{(k)}= |X_{(k)}-\theta_0|^{\rho} \mathrm{sgn}(X_{(k)}-\theta_0)$ when the underlying distribution function is $H(\cdot)$,
\begin{equation*}
    \begin{split}
        &\quad \mathbb{P}\left(\left|H(T_{(k)}) - H(0) - \left(\frac{k}{n} - \frac{1}{2}\right) + \frac{1}{n}\sum_{i=1}^n \mathbf{1}\{H(T_i) \le k/n\} - \frac{k}{n}\right| \ge D_{1,n}\left(\frac{\log n}{n}\right)^{3/4}\right) \\
        &\le 1+ \left(\frac{D_2}{n^2}-1\right)\mathbbm{1}\{\kappa(k/n,2,n) >\Tilde{\kappa}(k/n,2,n) \}.
    \end{split}
\end{equation*}
Since $H(T_i) \le k/n \iff F(X_i) \le k/n$, we can say that the following event happens with probability greater than or equal to $(1-(D_2/n^2))\mathbbm{1}\{\kappa(k/n,2,n) >\Tilde{\kappa}(k/n,2,n) \}$,
\[
\left|H(T_{(k)}) - H(0) - \left(\frac{k}{n} - \frac{1}{2}\right) + \frac{1}{n}\sum_{i=1}^n \mathbf{1}\{F(X_i) \le k/n\} - \frac{k}{n}\right| \le D_{1,n}\left(\frac{\log n}{n}\right)^{3/4}.
\]
In particular for $k=k_{n,\alpha}+1$ we have with probability greater than or equal to $(1-(D_2/n^2))\mathbbm{1}\{\kappa((k_{n,\alpha}+1)/n,2,n) >\Tilde{\kappa}((k_{n,\alpha}+1)/n,2,n) \}$,
\[
\left|H(T_{(k_{n,\alpha}+1)}) - H(0) - Q_{1,\boldsymbol{X},n}\right| \le D_{1,n}\left(\frac{\log n}{n}\right)^{3/4},
\]
where,
\[
Q_{1,\boldsymbol{X},n}= \left(\frac{k_{n,\alpha}+1}{n} - \frac{1}{2}\right) - \frac{1}{n}\sum_{i=1}^n \mathbf{1}\{F(X_i) \le (k_{n,\alpha}+1)/n\} + \frac{k_{n,\alpha}+1}{n}.
\]
Note that $Q_{1,\boldsymbol{X},n}$ is same as $Q$ defined in \Cref{thm:asym_result_irreg_case}. From the assumption \eqref{eq:asump_H} we obtain that,
\[
|H(T_{(k)})-H(0)-Mt| \leq C|T_{(k)}|^{1+\delta} \quad\mbox{provided}\quad |T_{(k)}|<\eta^{\rho}.
\]
Therefore for $k=k_{n,\alpha}+1$, using \eqref{eq:non_bahadur_order_stat} we obtain that with probability greater than or equal to $(1-2n^{-2})\mathbbm{1}\{A_n < \min\{\zeta, (M\eta^{\rho})/2 \} \}$ (which is same as $(1-2n^{-2})\mathbbm{1}\{A_n < \zeta\}$ as $\zeta=(M/2)\min\{\eta^{\rho},(M/2C)^{1/\delta}\} \leq (M\eta^{\rho})/2$) the following event happens for all $n \geq \log_2(2/\alpha)$, 
\begin{equation*}
    \begin{split}
 |H(T_{(k_{n,\alpha}+1)})-H(0)-MT_{(k_{n,\alpha}+1)}| &\leq C|T_{(k_{n,\alpha}+1)}|^{1+\delta}        \\
 &\leq C\left( \frac{2A_n}{M} \right)^{1+\delta}.
    \end{split}
\end{equation*}
Using triangle inequality on the above two bounds we obtain that with probability greater than or equal to $(1-(2+D_2)n^{-2})\mathbbm{1}\{A_n < \zeta \}\mathbbm{1}\{\kappa((k_{n,\alpha}+1)/n,2,n) >\Tilde{\kappa}((k_{n,\alpha}+1)/n,2,n) \}$ the following event happens for all $n \geq \log_2(2/\alpha)$, 
\[
|MT_{(k_{n,\alpha}+1)} - Q_{1,\boldsymbol{X},n}| \leq D_{1,n}\left(\frac{\log n}{n}\right)^{3/4} + C\left( \frac{2A_n}{M} \right)^{1+\delta}.
\]
Similarly for $k=n-k_{n,\alpha}$ we obtain that  with probability greater than or equal to $(1-(2+D_2)n^{-2})\mathbbm{1}\{B_n < \zeta \}\mathbbm{1}\{\kappa((n-k_{n,\alpha})/n,2,n) >\Tilde{\kappa}((n-k_{n,\alpha})/n,2,n) \}$ the following event happens for all $n \geq \log_2(2/\alpha)$, 
\begin{equation*}
    \begin{split}
 |MT_{(n-k_{n,\alpha})} - Q_{2,\boldsymbol{X},n}| &\leq D_{1,n}\left(\frac{\log n}{n}\right)^{3/4} + C\left( \frac{2B_n}{M} \right)^{1+\delta} \\
 &\leq D_{1,n}\left(\frac{\log n}{n}\right)^{3/4} + C\left( \frac{2A_n}{M} \right)^{1+\delta},
    \end{split}
\end{equation*}
where,
\[
Q_{2,\boldsymbol{X},n}= \left(\frac{n-k_{n,\alpha}}{n} - \frac{1}{2}\right) - \frac{1}{n}\sum_{i=1}^n \mathbf{1}\{F(X_i) \le (n-k_{n,\alpha})/n\} + \frac{n-k_{n,\alpha}}{n}.
\]
Hence with probability greater than or equal to $(1-(4+2D_2)n^{-2})\mathbbm{1}\{A_n < \zeta \}\mathbbm{1}\{\kappa((k_{n,\alpha}+1)/n,2,n) >\Tilde{\kappa}((k_{n,\alpha}+1)/n,2,n) \}\mathbbm{1}\{\kappa((n-k_{n,\alpha})/n,2,n) >\Tilde{\kappa}((n-k_{n,\alpha})/n,2,n) \}$ both the above events hold true. If we let $K_n$ to be the event for which $\mathbbm{1}\{K_n\}=\mathbbm{1}\{\kappa((k_{n,\alpha}+1)/n,2,n) >\Tilde{\kappa}((k_{n,\alpha}+1)/n,2,n) \}\mathbbm{1}\{\kappa((n-k_{n,\alpha})/n,2,n) >\Tilde{\kappa}((n-k_{n,\alpha})/n,2,n) \}$, we obtain that with probability greater than or equal to $(1-(4+2D_2)n^{-2})\mathbbm{1}\{A_n < \zeta \}\mathbbm{1}\{K_n\}$ both the above events hold true. We let $\gamma_n=D_{1,n}(\log(n)/n)^{3/4}+C(2A_n/M)^{1+\delta}$. To sum up, with probability greater than or equal to $(1-(4+2D_2)n^{-2})\mathbbm{1}\{A_n < \zeta \}\mathbbm{1}\{K_n\}$ the following events occur for all $n \geq \log_2(2/\alpha)$,
\begin{equation*}
    \begin{split}
        \frac{1}{M}(Q_{1,\boldsymbol{X},n}-\gamma_n) \leq T_{(k_{n,\alpha}+1)} \leq \frac{1}{M}(Q_{1,\boldsymbol{X},n}+\gamma_n),\\
        \frac{1}{M}(Q_{2,\boldsymbol{X},n}-\gamma_n) \leq T_{(n-k_{n,\alpha})} \leq \frac{1}{M}(Q_{2,\boldsymbol{X},n}+\gamma_n).
    \end{split}
\end{equation*}
Using the monotonicity of the transformation we obtain that with probability greater than or equal to $(1-(4+2D_2)n^{-2})\mathbbm{1}\{A_n < \zeta \}\mathbbm{1}\{K_n\}$ the following events occur for all $n \geq \log_2(2/\alpha)$,
\begin{equation*}
    \begin{split}
        \frac{1}{M^{1/\rho}}|Q_{1,\boldsymbol{X},n}-\gamma_n|^{1/\rho} \mathrm{sgn}(Q_{1,\boldsymbol{X},n}-\gamma_n) \leq X_{(k_{n,\alpha}+1)}-\theta_0 \leq \frac{1}{M^{1/\rho}}|Q_{1,\boldsymbol{X},n}+\gamma_n|^{1/\rho} \mathrm{sgn}(Q_{1,\boldsymbol{X},n}+\gamma_n),\\
        \frac{1}{M^{1/\rho}}|Q_{2,\boldsymbol{X},n}-\gamma_n|^{1/\rho} \mathrm{sgn}(Q_{2,\boldsymbol{X},n}-\gamma_n) \leq X_{(n-k_{n,\alpha})}-\theta_0 \leq \frac{1}{M^{1/\rho}}|Q_{2,\boldsymbol{X},n}+\gamma_n|^{1/\rho} \mathrm{sgn}(Q_{2,\boldsymbol{X},n}+\gamma_n).
    \end{split}
\end{equation*}
We obtain that with probability greater than or equal to $(1-(4+2D_2)n^{-2})\mathbbm{1}\{A_n < \zeta \}\mathbbm{1}\{K_n\}$ the following happens for all $n \geq \log_2(2/\alpha)$,
\begin{equation*}
    \begin{split}
 &\quad\frac{1}{M^{1/\rho}}\{|Q_{1,\boldsymbol{X},n}-\gamma_n|^{1/\rho} \mathrm{sgn}(Q_{1,\boldsymbol{X},n}-\gamma_n) -  |Q_{2,\boldsymbol{X},n}+\gamma_n|^{1/\rho} \mathrm{sgn}(Q_{2,\boldsymbol{X},n}+\gamma_n)     \} \\
 &\leq X_{(k_{n,\alpha}+1)} - X_{(n-k_{n,\alpha})} \\
 &\leq \frac{1}{M^{1/\rho}}\{|Q_{1,\boldsymbol{X},n}+\gamma_n|^{1/\rho} \mathrm{sgn}(Q_{1,\boldsymbol{X},n}+\gamma_n)- |Q_{2,\boldsymbol{X},n}-\gamma_n|^{1/\rho} \mathrm{sgn}(Q_{2,\boldsymbol{X},n}-\gamma_n)\}.
    \end{split}
\end{equation*}
This completes the proof of the theorem.    
\end{proof}
\begin{theorem}
\label{thm:asym_result_irreg_case_1}
Let $X_1,X_2,\cdots,X_n \stackrel{iid}{\sim} F $. Suppose that $F$ is a continuous CDF with median $\theta_0$. We assume that the following holds for the distribution function $F$,
\begin{equation}
|F(\theta_0+h)-F(\theta_0)-M|h|^{\rho}\mathrm{sgn}(h)| \leq C|h|^{\rho+\Delta} \quad\mbox{for all}\quad |h|<\eta,   
\end{equation}
where $0<M,C,\Delta,\eta,\rho<\infty$. Then with probability greater than or equal to $(1-(6+2D_2)n^{-2})\mathbbm{1}\{A_n < \zeta \}\mathbbm{1}\{K_n\}$ the following event holds true for $1 \leq \rho < \infty$,
\begin{equation*}
    \begin{split}
  &\left|n^{1/(2\rho)}\mathrm{Width}(\widehat{\mathrm{CI}}_{n,\alpha}) -  \frac{n^{1/2\rho}}{M^{1/\rho}}\{|Q_{1,\boldsymbol{X},n}|^{1/\rho}\mathrm{sgn}(Q_{1,\boldsymbol{X},n})- |Q_{1,\boldsymbol{X},n}-(z_{\alpha/2}/\sqrt{n})|^{1/\rho}\mathrm{sgn}(Q_{1,\boldsymbol{X},n}-(z_{\alpha/2}/\sqrt{n})) \} \right|    \\
  & \leq \frac{4}{M^{1/\rho}}(\sqrt{n}\gamma_n)^{1/\rho} + \frac{2}{M^{1/\rho}}\left( \frac{z_{\alpha/2}^3}{5n} + \frac{4.2}{\sqrt{n}}+ \sqrt{\frac{6(3+\sqrt{n}z_{\alpha/2})\log(n)}{n}}\right)^{1/\rho}.
    \end{split}
\end{equation*}
With probability greater than or equal to $(1-(10+2D_2)n^{-2})\mathbbm{1}\{A_n < \zeta \}\mathbbm{1}\{K_n\}$ the following event holds true for $0< \rho < 1$,
\begin{equation*}
    \begin{split}
  &\left|n^{1/(2\rho)}\mathrm{Width}(\widehat{\mathrm{CI}}_{n,\alpha}) -  \frac{n^{1/2\rho}}{M^{1/\rho}}\{|Q_{1,\boldsymbol{X},n}|^{1/\rho}\mathrm{sgn}(Q_{1,\boldsymbol{X},n})- |Q_{1,\boldsymbol{X},n}-(z_{\alpha/2}/\sqrt{n})|^{1/\rho}\mathrm{sgn}(Q_{1,\boldsymbol{X},n}-(z_{\alpha/2}/\sqrt{n})) \} \right|    \\
  & \leq \frac{2C_{\rho, \alpha}}{M^{1/\rho}}\sqrt{n}\gamma_n + \frac{C_{\rho, \alpha}}{M^{1/\rho}}\left( \frac{z_{\alpha/2}^3}{5n} + \frac{4.2}{\sqrt{n}}+ \sqrt{\frac{6(3+\sqrt{n}z_{\alpha/2})\log(n)}{n}}\right),
    \end{split}
\end{equation*}
where $C_{\rho, \alpha}$ is a constant depending on $\rho$ (see \eqref{eq:crho} for details). We also have the following,
\begin{equation*}
    \begin{split}
        &\frac{n^{1/2\rho}}{M^{1/\rho}}\{|Q_{1,\boldsymbol{X},n}|^{1/\rho}\mathrm{sgn}(Q_{1,\boldsymbol{X},n})- |Q_{1,\boldsymbol{X},n}-(z_{\alpha/2}/\sqrt{n})|^{1/\rho}\mathrm{sgn}(Q_{1,\boldsymbol{X},n}-(z_{\alpha/2}/\sqrt{n})) \} \\
        & \stackrel{d}{\xrightarrow{}} \frac{1}{M^{1/\rho}}\{|W_1|^{1/\rho}\mathrm{sgn}(W_1)-|W_1 - z_{\alpha/2}|^{1/\rho}\mathrm{sgn}(W_1 - z_{\alpha/2}) \},
    \end{split}
\end{equation*}
where $W_1 \sim N(z_{\alpha/2}/2,1/4)$.
\end{theorem}

Let us see how \Cref{thm:asym_result_irreg_case_1} implies the simplified version of the result stated in \Cref{thm:asym_result_irreg_case}. We have already seen in \Cref{appendix:theorem_bahadur_finsample} that the constraint on sample-size $n$ $A_n < \zeta$ holds true if $n \geq 49 \log(2n/\alpha)/\zeta^2$. We need to find a sufficient condition on $n$ so that the event $K_n$ holds true. Let $p = k/(n+1)$ and $q = k/n$. Let $\delta = p/q = n/(n+1)$. We observe the following, 
\begin{equation*}
    \begin{split}
        \frac{p(1 - p)}{q(1 - q)} - 1 &= \frac{\delta(1 - q\delta)}{1 - q} - 1 \\
       &= \frac{\delta(1 - \delta + \delta(1 - q))}{1 - q} - 1 \\
       &= \frac{\delta(1 - \delta)}{1 - q} + \delta^2 - 1 \\
       &= \left(-\frac{\delta}{1 - q} + \delta + 1 \right)\left( \delta - 1 \right) \\
       &= \frac{1}{n+1}\left( \delta \left( \frac{q}{1 - q} \right) - 1 \right). 
    \end{split}
\end{equation*}
We have shown in \Cref{appendix:theorem_asym_cn} that, 
\[
\frac{1}{2} + \frac{z_{\alpha/2}}{2\sqrt{n}} - \frac{z_{\alpha/2}^3}{10n\sqrt{n}} - \frac{1.6}{n} \leq \frac{k_{n,\alpha}}{n} = q \leq \frac{1}{2} + \frac{z_{\alpha/2}}{2\sqrt{n}} + \frac{1}{n}. 
\]
In our case $q$ is either $k_{n,\alpha} + 1$ or $n - k_{n,\alpha}$. Therefore we can bound $q/(1 - q)$ in the following way, 
\begin{equation*}
    \begin{split}
        \frac{q}{1 - q} &\leq \max\left\{\frac{\frac{1}{2} + \frac{z_{\alpha/2}}{2\sqrt{n}} + \frac{2}{n}}{\frac{1}{2} - \frac{z_{\alpha/2}}{2\sqrt{n}} +\frac{z_{\alpha/2}^3}{10n\sqrt{n}} + \frac{0.6}{n}},  \frac{\frac{1}{2} - \frac{z_{\alpha/2}}{2\sqrt{n}} +\frac{z_{\alpha/2}^3}{10n\sqrt{n}} + \frac{1.6}{n}}{\frac{1}{2} + \frac{z_{\alpha/2}}{2\sqrt{n}} + \frac{1}{n}}\right\} \\
        &\leq \max\left\{ \frac{\frac{1}{2} + \frac{z_{\alpha/2}}{2\sqrt{n}} + \frac{2}{n}}{\frac{1}{2} - \frac{z_{\alpha/2}}{2\sqrt{n}} },  
        \frac{\frac{1}{2} - \frac{z_{\alpha/2}}{2\sqrt{n}} +\frac{z_{\alpha/2}^3}{10n\sqrt{n}} + \frac{1.6}{n}}{\frac{1}{2} + \frac{z_{\alpha/2}}{2\sqrt{n}} }\right\}\\
        &\leq \max \left\{1 + 2 \left( \frac{z_{\alpha/2}}{\sqrt{n}} + \frac{1}{n} \right), 1 - \frac{\frac{z_{\alpha/2}}{\sqrt{n}} - \frac{z_{\alpha/2}^3}{10n\sqrt{n}}}{\frac{1}{2} + \frac{z_{\alpha/2}}{2\sqrt{n}}} + \frac{0.8}{1/2} \right\} \\
        &\leq \max \left\{1 + 2 \left( \frac{z_{\alpha/2}}{\sqrt{n}} + \frac{1}{n} \right), 2.6 - \frac{\left(\frac{z_{\alpha/2}}{\sqrt{n}} - \frac{z_{\alpha/2}^3}{10 n \sqrt{n}}\right)}{\frac{1}{2} + \frac{z_{\alpha/2}}{2\sqrt{n}}} \right\} \\
        &\leq \max\{4, 2.6 \} \quad \mbox{(if $n \geq 4 z_{\alpha/2}^2$)} \\
        & \leq 4. 
    \end{split}
\end{equation*}
We use this bound in the previous derivation, 
\begin{equation*}
    \begin{split}
       \frac{p(1 - p)}{q(1 - q)} - 1 &= \frac{1}{n+1}\left( \delta \left( \frac{q}{1 - q} \right) - 1 \right) \\
       &\leq \frac{1}{n+1}\left( \frac{4n}{n+1}  - 1 \right) \\
       &\leq 0.56. 
    \end{split}
\end{equation*}
For $k = k_{n,\alpha} + 1, n - k_{n,\alpha}$ last equation can be re-written as, 
\begin{equation*}
    \begin{split}
        & \sqrt{1.56}\sqrt{\frac{k}{n} \left(1 - \frac{k}{n} \right)} \geq \sqrt{\frac{k}{n+1} \left(1 - \frac{k}{n+1} \right)} \\
        \implies & \sqrt{18.72}\sqrt{\frac{k}{n} \left(1 - \frac{k}{n} \right)} \geq \sqrt{12}\sqrt{\frac{k}{n+1} \left(1 - \frac{k}{n+1} \right)} \\
        \implies &\sqrt{18.72}\sqrt{\frac{k}{n} \left(1 - \frac{k}{n} \right)} + (\sqrt{21} - \sqrt{18.72})\sqrt{\frac{k}{n} \left(1 - \frac{k}{n} \right)} \geq \sqrt{12}\sqrt{\frac{k}{n+1} \left(1 - \frac{k}{n+1} \right)} + \frac{1}{n} .
    \end{split}
\end{equation*}
The following inequality holds true, 
\begin{equation*}
    \begin{split}
       & 9^2 \log(n) \geq \frac{1}{n} \\
     \implies & 21 \sqrt{\frac{\log n}{n}} \geq 12 \sqrt{\frac{\log n}{n}} + \frac{1}{n}.    
    \end{split}
\end{equation*}
Combining both the inequalities we obtain the following for $k = k_{n,\alpha} + 1, n - k_{n,\alpha}$, 
\begin{equation*}
    \begin{split}
        & \max \left\{\sqrt{21}\sqrt{\frac{k}{n} \left(1 - \frac{k}{n} \right)},  21 \sqrt{\frac{\log n}{n}}\right\} \geq \max\left\{ \sqrt{12}\sqrt{\frac{k}{n+1} \left(1 - \frac{k}{n+1} \right)} , 12 \sqrt{\frac{\log n}{n}}\right\} + \frac{1}{n} \\
        \implies & \kappa(k/n, 2, n) \geq \Tilde{\kappa}(k/n, 2, n).
    \end{split}
\end{equation*}
Thus the event $K_n$ holds when $n \geq 4 z_{\alpha/2}^2$. Therefore the result of \Cref{thm:asym_result_irreg_case_1} holds true when $n\geq \max\{\log_2(2/\alpha), 49\log(2n/\alpha)/\zeta^2, 4z_{\alpha/2}^2\}$. We have shown in \Cref{appendix:theorem_bahadur_finsample} that $A_n \leq 7\sqrt{\log(2n/\alpha)/n}$. We use this to bound $\gamma_n$. 
\begin{equation*}
    \begin{split}
    \gamma_n &=\max\left\{\frac{35}{\sqrt{2}},\,\, 35\sqrt{\frac{35\log(n)}{n}} \right\}\left(\frac{\log(n)}{n}\right)^{3/4}+C\left(\frac{2A_n}{M}\right)^{1+\delta} \\
    &\leq 208 \left( \frac{\log n}{n} \right)^{3/4} + C \left( \frac{14}{M}\right)^{1+\delta} \left(\frac{\log(2n/\alpha)}{n} \right)^{(1+\delta)/2}.
    \end{split}
\end{equation*}
Similarly we can bound the second term as follows.
\begin{equation*}
    \begin{split}
      & \frac{z_{\alpha/2}^3}{5n} + \frac{4.2}{\sqrt{n}}+ \sqrt{\frac{6(3+\sqrt{n}z_{\alpha/2})\log(n)}{n}} \\
      \leq & \frac{z_{\alpha/2}^3}{5n} + \frac{4.2}{\sqrt{n}} + \sqrt{\frac{18 \log(n)}{n}} + \sqrt{\frac{6 z_{\alpha/2}\log(n)}{\sqrt{n}}} \\
      \leq & \sqrt{z_{\alpha/2}} \left(\frac{\sqrt{6\log(n)}}{n^{1/4}} + \frac{z_{\alpha/2}^{5/2}}{5n} \right) + \sqrt{\frac{18 \log(n)}{n}} \left( 1+ \frac{4.2}{\sqrt{18 \log(n)}} \right) \\
    \leq & \sqrt{z_{\alpha/2}} \left(\frac{\sqrt{6\log(n)}}{n^{1/4}} + \frac{z_{\alpha/2}}{5n^{1/4}} \right) + \sqrt{\frac{18 \log(n)}{n}} \left( 1+ \frac{4.2}{\sqrt{18 \log(n)}} \right) \\   
    \leq & \frac{(2\log(2/\alpha))^{1/4}}{n^{1/4}} \left( \sqrt{6 \log(n)} + \sqrt{2\log(2/\alpha)} \right) + \sqrt{\frac{18 \log(n)}{n}} \left( 1+ \frac{4.2}{\sqrt{18 \log(n)}} \right) \\
    \leq & \frac{(2\log(2/\alpha))^{1/4}}{n^{1/4}} \sqrt{12 \log(2n/\alpha)} + 2.2 \sqrt{\frac{18 \log(n)}{n}} \\
    \leq & \frac{4.2 \log(2/\alpha)^{1/4} \sqrt{\log(2n/\alpha)}}{n^{1/4}} + 9.5\sqrt{\frac{\log n}{n}}.
    \end{split}
\end{equation*}
Combining all the above bounds and using the fact that $\mathrm{Width}(\widehat{\mathrm{CI}}_{n,\alpha}^{\texttt{Wald}}) = 2^{1 - (1/\rho)} n^{-1/(2\rho)}(z_{\alpha/2}/M)^{1/\rho}$ we obtain the result in \Cref{thm:asym_result_irreg_case}. Note that to obtain the simplified version (as stated in \Cref{thm:asym_result_irreg_case}), we used the identity $\mathscr{G}(W, z_{\alpha/2})/(2^{1 - (1/\rho)}z_{\alpha/2}^{1/\rho}) = \mathscr{G}((Z/z_{\alpha/2}) + 1, 2)/2$ where $W \stackrel{d}{=} (z_{\alpha/2}/2) + (1/2) Z$ and $Z \sim N(0, 1)$. 
\begin{proof}[Proof of \Cref{thm:asym_result_irreg_case_1}]

We shall first deal with the case $\rho \geq 1$. We shall start by proving that the function $h(x)=|x|^{1/\rho} \mathrm{sgn}(x)$ is a holder continuous function for $\rho \geq 1$. We note that the following holds for all $\rho \geq 1$ and for all $a,b>0$,
\begin{equation*}
    \begin{split}
       & \left(\frac{a}{a+b}\right)^{1/\rho} +  \left(\frac{b}{a+b}\right)^{1/\rho} \geq 1\\
       \implies & a^{1/\rho} + b ^{1/\rho} \geq (a+b)^{1/\rho}.
    \end{split}
\end{equation*}
This implies that $|h(x)-h(y)| \leq |x-y|^{1/\rho}$ when $x,y$ are of the same sign. Suppose w.l.o.g. $x<0<y$. In this case we can say the following,
\[
\left(\frac{|x|}{|x|+y}\right)^{1/\rho} +  \left(\frac{y}{|x|+y}\right)^{1/\rho} \leq 2.
\]
This implies that $|h(x)-h(y)| \leq 2|x-y|^{1/\rho}$ when $x,y$ are of the opposite sign. Combining both the cases we obtain that $h(x)$ is holder continuous and the following holds true,
\[
|h(x)-h(y)| \leq 2|x-y|^{1/\rho} \quad \mbox{for all  } x,y \in \mathbb{R}.
\]
Repeatedly using the holder continuity of the function $h(x)$ we obtain the following inequality,
\begin{equation}
\label{eq:app_holder_1}
\begin{split}
&\frac{n^{1/2\rho}}{M^{1/\rho}}|\{|Q_{1,\boldsymbol{X},n}-\gamma_n|^{1/\rho} \mathrm{sgn}(Q_{1,\boldsymbol{X},n}-\gamma_n) -  |Q_{2,\boldsymbol{X},n}+\gamma_n|^{1/\rho} \mathrm{sgn}(Q_{2,\boldsymbol{X},n}+\gamma_n)     \} \\
&- \{|Q_{1,\boldsymbol{X},n}|^{1/\rho} \mathrm{sgn}(Q_{1,\boldsymbol{X},n})- |Q_{1,\boldsymbol{X},n}-(z_{\alpha/2}/\sqrt{n})|^{1/\rho} \mathrm{sgn}(Q_{1,\boldsymbol{X},n}-(z_{\alpha/2}/\sqrt{n})) \} | \\
& \leq \frac{4}{M^{1/\rho}}(\sqrt{n}\gamma_n)^{1/\rho} + \frac{2}{M^{1/\rho}}|\sqrt{n}(Q_{1,\boldsymbol{X},n}-Q_{2,\boldsymbol{X},n})-z_{\alpha/2}|^{1/\rho}.   
\end{split}
\end{equation}
Similarly we also have,
\begin{equation}
\label{eq:app_holder_2}
\begin{split}
&\frac{n^{1/2\rho}}{M^{1/\rho}}|\{|Q_{1,\boldsymbol{X},n}+\gamma_n|^{1/\rho} \mathrm{sgn}(Q_{1,\boldsymbol{X},n}+\gamma_n) -  |Q_{2,\boldsymbol{X},n}-\gamma_n|^{1/\rho} \mathrm{sgn}(Q_{2,\boldsymbol{X},n}-\gamma_n)     \} \\
&- \{|Q_{1,\boldsymbol{X},n}|^{1/\rho} \mathrm{sgn}(Q_{1,\boldsymbol{X},n})- |Q_{1,\boldsymbol{X},n}-(z_{\alpha/2}/\sqrt{n})|^{1/\rho} \mathrm{sgn}(Q_{1,\boldsymbol{X},n}-(z_{\alpha/2}/\sqrt{n})) \} | \\
& \leq \frac{4}{M^{1/\rho}}(\sqrt{n}\gamma_n)^{1/\rho} + \frac{2}{M^{1/\rho}}|\sqrt{n}(Q_{1,\boldsymbol{X},n}-Q_{2,\boldsymbol{X},n})-z_{\alpha/2}|^{1/\rho}.   
\end{split}
\end{equation}
We shall now analyse the difference $(Q_{1,\boldsymbol{X},n}-Q_{2,\boldsymbol{X},n})$. We observe the following,
\[
Q_{1,\boldsymbol{X},n}-Q_{2,\boldsymbol{X},n} = \frac{2k_{n,\alpha}-n+1}{n} - \left\{\frac{1}{n}\sum_{i=1}^n \textbf{1}\left\{\frac{n-k_{n,\alpha}}{n} \leq F(X_i) \leq \frac{k_{n,\alpha}+1}{n} \right\} - \frac{2k_{n,\alpha}-n+1}{n} \right\}.
\]
Using part-$3$ of \Cref{thm:assym_of_cn} we have the following bound,
\[
\left|\sqrt{n}\left(\frac{2k_{n,\alpha}-n+1}{n}\right) - z_{\alpha/2} \right| \leq \frac{z_{\alpha/2}^3}{5n} + \frac{4.2}{\sqrt{n}}.
\]
Using Chernoff bound we have,
\begin{equation*}
    \begin{split}
        &\mathbb{P}\left\{\sqrt{n}\left|\frac{1}{n}\sum_{i=1}^n \textbf{1}\left\{\frac{n-k_{n,\alpha}}{n} \leq F(X_i) \leq \frac{k_{n,\alpha}+1}{n} \right\} - \frac{2k_{n,\alpha}-n+1}{n} \right| \geq \sqrt{\frac{6(3+\sqrt{n}z_{\alpha/2})\log(n)}{n}}\right\}\\
        \leq & \exp\left\{- \frac{n}{3(2k_{n,\alpha}-n+1)}\frac{6(3+\sqrt{n}z_{\alpha/2})\log(n)}{n}\right\} + \exp\left\{- \frac{n}{2(2k_{n,\alpha}-n+1)}\frac{6(3+\sqrt{n}z_{\alpha/2})\log(n)}{n}\right\} \\
        \leq & \exp\{-2\log(n)\} + \exp\{-3\log(n)\} \\
        \leq &\frac{2}{n^2}. 
    \end{split}
\end{equation*}
Thus we have the following concentration inequality for $(Q_{1,\boldsymbol{X},n}-Q_{2,\boldsymbol{X},n})$,
\begin{equation}
    \label{eq:chernoff_conc}
    \mathbb{P}\left\{|\sqrt{n}(Q_{1,\boldsymbol{X},n}-Q_{2,\boldsymbol{X},n})-z_{\alpha/2}| > \frac{z_{\alpha/2}^3}{5n} + \frac{4.2}{\sqrt{n}}+ \sqrt{\frac{6(3+\sqrt{n}z_{\alpha/2})\log(n)}{n}}\right\} \leq \frac{2}{n^2}.
\end{equation}
Using \eqref{eq:app_holder_1} and \eqref{eq:chernoff_conc} we conclude that the following event happens with probability greater than or equal to $(1-2n^{-2})$,
\begin{equation*}
    \begin{split}
&\frac{n^{1/2\rho}}{M^{1/\rho}}|\{|Q_{1,\boldsymbol{X},n}-\gamma_n|^{1/\rho} \mathrm{sgn}(Q_{1,\boldsymbol{X},n}-\gamma_n) -  |Q_{2,\boldsymbol{X},n}+\gamma_n|^{1/\rho} \mathrm{sgn}(Q_{2,\boldsymbol{X},n}+\gamma_n)     \} \\
&- \{|Q_{1,\boldsymbol{X},n}|^{1/\rho} \mathrm{sgn}(Q_{1,\boldsymbol{X},n})- |Q_{1,\boldsymbol{X},n}-(z_{\alpha/2}/\sqrt{n})|^{1/\rho} \mathrm{sgn}(Q_{1,\boldsymbol{X},n}-(z_{\alpha/2}/\sqrt{n})) \} | \\
& \leq \frac{4}{M^{1/\rho}}(\sqrt{n}\gamma_n)^{1/\rho} + \frac{2}{M^{1/\rho}}\left( \frac{z_{\alpha/2}^3}{5n} + \frac{4.2}{\sqrt{n}}+ \sqrt{\frac{6(3+\sqrt{n}z_{\alpha/2})\log(n)}{n}}\right)^{1/\rho}.
    \end{split}
\end{equation*}
Similarly using \eqref{eq:app_holder_2} and \eqref{eq:chernoff_conc} we conclude that the following event happens with probability greater than or equal to $(1-2n^{-2})$,
\begin{equation*}
    \begin{split}
&\frac{n^{1/2\rho}}{M^{1/\rho}}|\{|Q_{1,\boldsymbol{X},n}+\gamma_n|^{1/\rho} \mathrm{sgn}(Q_{1,\boldsymbol{X},n}+\gamma_n) -  |Q_{2,\boldsymbol{X},n}-\gamma_n|^{1/\rho} \mathrm{sgn}(Q_{2,\boldsymbol{X},n}-\gamma_n)     \} \\
&- \{|Q_{1,\boldsymbol{X},n}|^{1/\rho} \mathrm{sgn}(Q_{1,\boldsymbol{X},n})- |Q_{1,\boldsymbol{X},n}-(z_{\alpha/2}/\sqrt{n})|^{1/\rho} \mathrm{sgn}(Q_{1,\boldsymbol{X},n}-(z_{\alpha/2}/\sqrt{n})) \} | \\
& \leq \frac{4}{M^{1/\rho}}(\sqrt{n}\gamma_n)^{1/\rho} + \frac{2}{M^{1/\rho}}\left( \frac{z_{\alpha/2}^3}{5n} + \frac{4.2}{\sqrt{n}}+ \sqrt{\frac{6(3+\sqrt{n}z_{\alpha/2})\log(n)}{n}}\right)^{1/\rho}.
    \end{split}
\end{equation*}
Therefore with probability greater than or equal to $(1-2n^{-2})$ both the above events hold simultaneously. Finally, using \Cref{thm:fin_sample_non_bahadur} and the above two results we conclude that with probability greater than or equal to $(1-(6+2D_2)n^{-2})\mathbbm{1}\{A_n < \zeta \}\mathbbm{1}\{K_n\}$ the following event holds true for $\rho \geq 1$,
\begin{equation*}
    \begin{split}
  &\left|n^{1/(2\rho)}\mathrm{Width}(\widehat{\mathrm{CI}}_{n,\alpha}) -  \frac{n^{1/2\rho}}{M^{1/\rho}}\{|Q_{1,\boldsymbol{X},n}|^{1/\rho} \mathrm{sgn}(Q_{1,\boldsymbol{X},n})- |Q_{1,\boldsymbol{X},n}-(z_{\alpha/2}/\sqrt{n})|^{1/\rho} \mathrm{sgn}(Q_{1,\boldsymbol{X},n}-(z_{\alpha/2}/\sqrt{n})) \} \right|    \\
  & \leq \frac{4}{M^{1/\rho}}(\sqrt{n}\gamma_n)^{1/\rho} + \frac{2}{M^{1/\rho}}\left( \frac{z_{\alpha/2}^3}{5n} + \frac{4.2}{\sqrt{n}}+ \sqrt{\frac{6(3+\sqrt{n}z_{\alpha/2})\log(n)}{n}}\right)^{1/\rho}.
    \end{split}
\end{equation*}
Using \Cref{lem:joint_distribution} we know that $\sqrt{n}Q_{1,\boldsymbol{X},n} \stackrel{d}{\xrightarrow{}} N(z_{\alpha/2}/2,1/4)$. By continuous mapping theorem we obtain the following,
\begin{equation*}
    \begin{split}
        &\frac{n^{1/2\rho}}{M^{1/\rho}}\{|Q_{1,\boldsymbol{X},n}|^{1/\rho} \mathrm{sgn}(Q_{1,\boldsymbol{X},n})- |Q_{1,\boldsymbol{X},n}-(z_{\alpha/2}/\sqrt{n})|^{1/\rho} \mathrm{sgn}(Q_{1,\boldsymbol{X},n}-(z_{\alpha/2}/\sqrt{n})) \} \\
        & \stackrel{d}{\xrightarrow{}} \frac{1}{M^{1/\rho}}\{|W_1|^{1/\rho} \mathrm{sgn}(W_1)-|W_1 - z_{\alpha/2}|^{1/\rho} \mathrm{sgn}(W_1 - z_{\alpha/2}) \},
    \end{split}
\end{equation*}
where $W_1 \sim N(z_{\alpha/2}/2,1/4)$. This completes the proof of the first part of the theorem. Now we move to the case when $0<\rho<1$. We recall that,
\[
Q_{1,\boldsymbol{X},n} = \left(\frac{k_{n,\alpha}+1}{n} - \frac{1}{2} \right)- \left\{\frac{1}{n} \sum_{i=1}^n \textbf{1} \left\{\ F(X_i) \leq \frac{k_{n,\alpha}+1}{n} \right\}- \frac{k_{n,\alpha}+1}{n} \right\}.
\]
Using part-$3$ of \Cref{thm:assym_of_cn} we have the following bound,
\[
\left| \sqrt{n}\left(\frac{k_{n,\alpha}+1}{n} - \frac{1}{2} \right) - \frac{z_{\alpha/2}}{2\sqrt{n}}\right| \leq \frac{z_{\alpha/2}^3}{10n} + \frac{2}{\sqrt{n}}. 
\]
Using Chernoff bound we have,
\begin{equation*}
    \begin{split}
        &\mathbb{P}\left\{\sqrt{n} \left|\frac{1}{n} \sum_{i=1}^n \textbf{1} \left\{\ F(X_i) \leq \frac{k_{n,\alpha}+1}{n} \right\}- \frac{k_{n,\alpha}+1}{n} \right| \geq \sqrt{\frac{3(n+\sqrt{n}z_{\alpha/2}+4)\log(n)}{n}}\right\} \\
        \leq & \exp\left\{ -\frac{n}{3(k_{n,\alpha}+1)} \frac{3(n+\sqrt{n}z_{\alpha/2}+4)\log(n)}{n}\right\} + \exp\left\{ -\frac{n}{2(k_{n,\alpha}+1)} \frac{3(n+\sqrt{n}z_{\alpha/2}+4)\log(n)}{n}\right\} \\
        \leq & \exp\{-2 \log(n) \} +\exp\{-3\log(n) \} \\
        =& \frac{2}{n^2}.
    \end{split}
\end{equation*}
Combining the above two bounds we obtain the following concentration inequality for $Q_{1,\boldsymbol{X},n}$,
\[
\mathbb{P}\left\{ |\sqrt{n}Q_{1,\boldsymbol{X},n}-(z_{\alpha/2}/2)| > \frac{z_{\alpha/2}^3}{10n} + \frac{2}{\sqrt{n}} + \sqrt{\frac{3(n+\sqrt{n}z_{\alpha/2}+4)\log(n)}{n}} \right\} \leq \frac{2}{n^2}.
\]
Similarly we have the following concentration inequality for $Q_{2,\boldsymbol{X},n}$,
\[
\mathbb{P}\left\{ |\sqrt{n}Q_{2,\boldsymbol{X},n}+(z_{\alpha/2}/2)| > \frac{z_{\alpha/2}^3}{10n} + \frac{2}{\sqrt{n}} + \sqrt{\frac{3(n-\sqrt{n}z_{\alpha/2}+(z_{\alpha/2}^3/(5\sqrt{n}))+3.2)\log(n)}{n}} \right\} \leq \frac{2}{n^2}.
\]
We define the following constants,
\begin{equation*}
    \begin{split}
        E_{1n} &= \frac{z_{\alpha/2}^3}{10n} + \frac{2}{\sqrt{n}} + \sqrt{\frac{3(n+\sqrt{n}z_{\alpha/2}+4)\log(n)}{n}} \\
        E_{2n} &=  \frac{z_{\alpha/2}^3}{10n} + \frac{2}{\sqrt{n}} + \sqrt{\frac{3(n-\sqrt{n}z_{\alpha/2}+(z_{\alpha/2}^3/(5\sqrt{n}))+3.2)\log(n)}{n}} \\
        E_{3n} &=  \sqrt{n}\gamma_n,\\
        E_1 & = \sup_{n \in \mathbb{N}} E_{1n}\\
        E_2 &=  \sup_{n \in \mathbb{N}} E_{2n}\\
        E_3 &= \sup_{n \in \mathbb{N}} E_{3n} .
    \end{split}
\end{equation*}
We note that $0<E_1,E_2,E_3<\infty$. We see that with probability greater than or equal to $1-4n^{-2}$ both the events $\sqrt{n}Q_{1,\boldsymbol{X},n} \in [z_{\alpha/2}/2-E_1,z_{\alpha/2}/2+E_1]$ and $\sqrt{n}Q_{2,\boldsymbol{X},n} \in [-z_{\alpha/2}/2-E_2,-z_{\alpha/2}/2+E_2]$ hold true. Let us denote $E_0=z_{\alpha/2}/2+\max\{E_1,E_2\}+E_3$. We can show that for $0<\rho<1$, $h(\cdot)$ is holder-continuous in the compact set $[-E_0,E_0]$. For $x,y \in [-E_0,E_0]$ we have,
\begin{equation*}
    \begin{split}
        |h(x)-h(y)| &= |h'(x*)||x-y| \quad \mbox{where $x*\in [x,y]$} \\
        &\leq (|h'(x)| +|h'(y)|)|x-y| \\
        &\leq \left(\frac{1}{\rho}|x|^{(1/\rho)-1} \mathrm{sgn}(x) + \frac{1}{\rho}|y|^{(1/\rho)-1} \mathrm{sgn}(y) \right) |x-y| \\
        &\leq \frac{2}{\rho} E_0^{(1/\rho)-1} |x-y| \\
        &= C_{\rho, \alpha} |x-y|, 
    \end{split}
\end{equation*}
where, 
\begin{equation}
    \label{eq:crho}
    C_{\rho, \alpha} = \frac{2}{\rho} E_0^{(1/\rho)-1} \quad \mbox{where } E_0=z_{\alpha/2}/2+\max\{E_1,E_2\}+E_3.
\end{equation}
We know from earlier concentration inequalities that both $\sqrt{n}Q_{1,\boldsymbol{X},n}\pm \sqrt{n}\gamma_n,\sqrt{n}Q_{2,\boldsymbol{X},n} \pm \sqrt{n}\gamma_n$ lie in the interval $[-E_0,E_0]$ with probability greater than or equal to $1-4n^{-2}$. Using the holder-continuity of the function $h(\cdot)$ in the interval $[-E_0,E_0]$ we can say that with probability greater than or equal to $1-4n^{-2}$, the following event holds,
\begin{equation*}
    \begin{split}
     &\frac{n^{1/2\rho}}{M^{1/\rho}}|\{|Q_{1,\boldsymbol{X},n}-\gamma_n|^{1/\rho} \mathrm{sgn}(Q_{1,\boldsymbol{X},n}-\gamma_n) -  |Q_{2,\boldsymbol{X},n}+\gamma_n|^{1/\rho} \mathrm{sgn}(Q_{2,\boldsymbol{X},n}+\gamma_n)     \} \\
&- \{|Q_{1,\boldsymbol{X},n}|^{1/\rho} \mathrm{sgn}(Q_{1,\boldsymbol{X},n})- |Q_{1,\boldsymbol{X},n}-(z_{\alpha/2}/\sqrt{n})|^{1/\rho} \mathrm{sgn}(Q_{1,\boldsymbol{X},n}-(z_{\alpha/2}/\sqrt{n})) \} | \\
=&\frac{1}{M^{1/\rho}}|\{|\sqrt{n}Q_{1,\boldsymbol{X},n}-\sqrt{n}\gamma_n|^{1/\rho} \mathrm{sgn}(\sqrt{n}Q_{1,\boldsymbol{X},n}-\sqrt{n}\gamma_n) -  |\sqrt{n}Q_{2,\boldsymbol{X},n}+\sqrt{n}\gamma_n|^{1/\rho} \mathrm{sgn}(\sqrt{n}Q_{2,\boldsymbol{X},n}+\sqrt{n}\gamma_n)     \} \\
&- \{|\sqrt{n}Q_{1,\boldsymbol{X},n}|^{1/\rho} \mathrm{sgn}(\sqrt{n}Q_{1,\boldsymbol{X},n})- |\sqrt{n}Q_{1,\boldsymbol{X},n}-z_{\alpha/2}|^{1/\rho} \mathrm{sgn}(\sqrt{n}Q_{1,\boldsymbol{X},n}-z_{\alpha/2} \} | \\
&\leq \frac{2C_{\rho, \alpha}}{M^{1/\rho}}|\sqrt{n}\gamma_n| + \frac{C_{\rho, \alpha}}{M^{1/\rho}}|\sqrt{n}(Q_{1,\boldsymbol{X},n}-Q_{2,\boldsymbol{X},n})-z_{\alpha/2}|.
    \end{split}
\end{equation*}
Using \eqref{eq:chernoff_conc} and the above bound we conclude that the following event happens with probability greater than or equal to $(1-6n^{-2})$,
\begin{equation*}
    \begin{split}
&\frac{n^{1/2\rho}}{M^{1/\rho}}|\{|Q_{1,\boldsymbol{X},n}-\gamma_n|^{1/\rho} \mathrm{sgn}(Q_{1,\boldsymbol{X},n}-\gamma_n) -  |Q_{2,\boldsymbol{X},n}+\gamma_n|^{1/\rho} \mathrm{sgn}(Q_{2,\boldsymbol{X},n}+\gamma_n)     \} \\
&- \{|Q_{1,\boldsymbol{X},n}|^{1/\rho} \mathrm{sgn}(Q_{1,\boldsymbol{X},n})- |Q_{1,\boldsymbol{X},n}-(z_{\alpha/2}/\sqrt{n})|^{1/\rho} \mathrm{sgn}(Q_{1,\boldsymbol{X},n}-(z_{\alpha/2}/\sqrt{n})) \} | \\
& \leq \frac{2C_{\rho, \alpha}}{M^{1/\rho}}\sqrt{n}\gamma_n + \frac{C_{\rho, \alpha}}{M^{1/\rho}}\left( \frac{z_{\alpha/2}^3}{5n} + \frac{4.2}{\sqrt{n}}+ \sqrt{\frac{6(3+\sqrt{n}z_{\alpha/2})\log(n)}{n}}\right).
    \end{split}
\end{equation*}
Similarly we conclude that the following event happens with probability greater than or equal to $(1-6n^{-2})$,
\begin{equation*}
    \begin{split}
&\frac{n^{1/2\rho}}{M^{1/\rho}}|\{|Q_{1,\boldsymbol{X},n}+\gamma_n|^{1/\rho} \mathrm{sgn}(Q_{1,\boldsymbol{X},n}+\gamma_n) -  |Q_{2,\boldsymbol{X},n}-\gamma_n|^{1/\rho} \mathrm{sgn}(Q_{2,\boldsymbol{X},n}-\gamma_n)     \} \\
&- \{|Q_{1,\boldsymbol{X},n}|^{1/\rho} \mathrm{sgn}(Q_{1,\boldsymbol{X},n})- |Q_{1,\boldsymbol{X},n}-(z_{\alpha/2}/\sqrt{n})|^{1/\rho} \mathrm{sgn}(Q_{1,\boldsymbol{X},n}-(z_{\alpha/2}/\sqrt{n})) \} | \\
& \leq \frac{2C_{\rho, \alpha}}{M^{1/\rho}}\sqrt{n}\gamma_n + \frac{C_{\rho, \alpha}}{M^{1/\rho}}\left( \frac{z_{\alpha/2}^3}{5n} + \frac{4.2}{\sqrt{n}}+ \sqrt{\frac{6(3+\sqrt{n}z_{\alpha/2})\log(n)}{n}}\right).
    \end{split}
\end{equation*}
Therefore with probability greater than or equal to $(1-6n^{-2})$ both the above events hold simultaneously. Finally, using \Cref{thm:fin_sample_non_bahadur} and the above two results we conclude that with probability greater than or equal to $(1-(10+2D_2)n^{-2})\mathbbm{1}\{A_n < \zeta \}\mathbbm{1}\{K_n\}$ the following event holds true for $0< \rho < 1$,
\begin{equation*}
    \begin{split}
  &\left|n^{1/(2\rho)}\mathrm{Width}(\widehat{\mathrm{CI}}_{n,\alpha}) -  \frac{n^{1/2\rho}}{M^{1/\rho}}\{|Q_{1,\boldsymbol{X},n}|^{1/\rho} \mathrm{sgn}(Q_{1,\boldsymbol{X},n})- |Q_{1,\boldsymbol{X},n}-(z_{\alpha/2}/\sqrt{n})|^{1/\rho} \mathrm{sgn}(Q_{1,\boldsymbol{X},n}-(z_{\alpha/2}/\sqrt{n})) \} \right|    \\
  & \leq \frac{2C_{\rho, \alpha}}{M^{1/\rho}}\sqrt{n}\gamma_n + \frac{C_{\rho, \alpha}}{M^{1/\rho}}\left( \frac{z_{\alpha/2}^3}{5n} + \frac{4.2}{\sqrt{n}}+ \sqrt{\frac{6(3+\sqrt{n}z_{\alpha/2})\log(n)}{n}}\right).
    \end{split}
\end{equation*}
Finally by continuous mapping theorem we know that the following holds,
\begin{equation*}
    \begin{split}
        &\frac{n^{1/2\rho}}{M^{1/\rho}}\{|Q_{1,\boldsymbol{X},n}|^{1/\rho} \mathrm{sgn}(Q_{1,\boldsymbol{X},n})- |Q_{1,\boldsymbol{X},n}-(z_{\alpha/2}/\sqrt{n})|^{1/\rho} \mathrm{sgn}(Q_{1,\boldsymbol{X},n}-(z_{\alpha/2}/\sqrt{n})) \} \\
        & \stackrel{d}{\xrightarrow{}} \frac{1}{M^{1/\rho}}\{|W_1|^{1/\rho} \mathrm{sgn}(W_1)-|W_1 - z_{\alpha/2}|^{1/\rho} \mathrm{sgn}(W_1 - z_{\alpha/2}) \},
    \end{split}
\end{equation*}
where $W_1 \sim N(z_{\alpha/2}/2,1/4)$. Note that $W_1$ has the same distribution as $W$ in \Cref{thm:asym_result_irreg_case}.
\end{proof}

\section{Proof of Theorem~\ref{thm:nonreg_unequal_limits}}
\label{appendix:nonreg_unequal_limits}
We state and prove a stronger version of the result stated in \Cref{thm:nonreg_unequal_limits}.
\begin{theorem}
    \label{thm:nonreg_unequal_limits_1}
    Let $X_1,X_2,\cdots,X_n \stackrel{iid}{\sim} F $. Suppose that $F$ is a continuous CDF with median $\theta_0$. We assume that the following holds for the distribution function $F$,
\begin{equation}
|F(\theta_0+h)-F(\theta_0)-|h|^{\rho}\mathrm{sgn}(h)[L_{-1}\textbf{1}\{h<0\}+ L_1\textbf{1}\{h>0\}]| \leq C|h|^{\rho+\Delta} \quad\mbox{for all}\quad |h|<\eta,   
\end{equation}
where $0<L_{-1},L_1,C,\Delta,\eta,\rho<\infty$. We introduce the following notation, 
\begin{equation*}
\begin{split}
    \overline{\mathscr{G}}(a,b) = & |a|^{1/\rho} \mathrm{sgn}(a)\left[\frac{1}{L_{-1}^{1/\rho}}\textbf{1}\{a <0 \} + \frac{1}{L_{1}^{1/\rho}}\textbf{1}\{a >0 \} \right] \\
    & -|a-b|^{1/\rho} \mathrm{sgn}(a-b)\left[\frac{1}{L_{-1}^{1/\rho}}\textbf{1}\{a < b \} + \frac{1}{L_{1}^{1/\rho}}\textbf{1}\{a >b \} \right].
\end{split}
\end{equation*}
Then with probability greater than or equal to $(1-(6+2D_2)n^{-2})\mathbbm{1}\{A_n < \zeta \}\mathbbm{1}\{K_n\}$ the following event holds true for $\rho \geq 1$,
\begin{equation*}
\begin{split}
 & n^{1/(2\rho)}\left|\mathrm{Width}(\widehat{\mathrm{CI}}_{n,\alpha}) - \overline{\mathscr{G}}(Q_{1,\boldsymbol{X},n},(z_{\alpha/2}/\sqrt{n})) \right|\\
 \leq & \frac{4}{L_0^{1/\rho}}(\sqrt{n}\gamma_n)^{1/\rho} + \frac{2}{L_0^{1/\rho}}\left( \frac{z_{\alpha/2}^3}{5n} + \frac{4.2}{\sqrt{n}}+ \sqrt{\frac{6(3+\sqrt{n}z_{\alpha/2})\log(n)}{n}}\right)^{1/\rho}.
  \end{split}
\end{equation*}
With probability greater than or equal to $(1-(10+2D_2)n^{-2})\mathbbm{1}\{A_n < \zeta \}\mathbbm{1}\{K_n\}$ the following event holds true for $0< \rho < 1$,
\begin{equation*}
\begin{split}
& n^{1/(2\rho)}\left|\mathrm{Width}(\widehat{\mathrm{CI}}_{n,\alpha}) - \overline{\mathscr{G}}(Q_{1,\boldsymbol{X},n},(z_{\alpha/2}/\sqrt{n}))  \right|  \\
\leq & \frac{2C_{\rho, \alpha}}{L_0^{1/\rho}}\sqrt{n}\gamma_n + \frac{C_{\rho, \alpha}}{L_0^{1/\rho}}\left( \frac{z_{\alpha/2}^3}{5n} + \frac{4.2}{\sqrt{n}}+ \sqrt{\frac{6(3+\sqrt{n}z_{\alpha/2})\log(n)}{n}}\right).
\end{split}
\end{equation*}
where $C_{\rho, \alpha}$ is a constant depending on $\rho, \alpha$ (see \eqref{eq:crho} for details). The constant $M$ is replaced with $L_0 = \min\{L_{-1},L_1 \}$ in the definitions of $\gamma_n$ and $\zeta$ in \Cref{thm:fin_sample_non_bahadur}. All other notations are same as defined in \Cref{thm:fin_sample_non_bahadur}. We also have the following,
\[
n^{1/(2\rho)}\overline{\mathscr{G}}(Q_{1,\boldsymbol{X},n},(z_{\alpha/2}/\sqrt{n})) \stackrel{d}{\xrightarrow{}} \overline{\mathscr{G}}(W_1,z_{\alpha/2}),
\]
where $W_1 \sim N(z_{\alpha/2}/2,1/4)$.
\end{theorem}

\begin{proof}[Proof of \Cref{thm:nonreg_unequal_limits_1}]
The assumptions made in \Cref{thm:nonreg_unequal_limits_1} are same as in \Cref{thm:nonreg_unequal_limits} with $M_{-},M_{+},M$ replaced by $L_{-1},L_1,L_0$ respectively. The proof of this theorem follows along the same path as the proofs of \Cref{thm:fin_sample_non_bahadur} and \Cref{thm:asym_result_irreg_case}. Let us define $L_0=\min\{L_1, L_{-1} \}$. With a slight modification to the derivations in \Cref{lem:order_stat_behav} we have the following concentration inequalities for $X_{(k_{n,\alpha}+1)}$ and $X_{(n-k_{n,\alpha})}$,
\begin{equation}
\label{eq:non_bahadur_order_stat_1}
    \begin{split}
    \mathbb{P}\left(\left|X_{(k_{n,\alpha}+1)}-\theta_0 \right|^{\rho} \leq \frac{2A_n}{L_0} \right) &\geq (1- 2n^{-2})\mathbbm{1}\{A_n < \zeta\} \quad\quad \mbox{for} \quad\quad n\geq \log_2(2/\alpha),\\
   \mathbb{P}\left(\left|X_{(n-k_{n,\alpha})}-\theta_0 \right|^{\rho} \leq \frac{2B_n}{L_0} \right) &\geq (1- 2n^{-2})\mathbbm{1}\{B_n < \zeta\} \quad\quad \mbox{for} \quad\quad n\geq \log_2(2/\alpha)  ,  
    \end{split}
\end{equation}
where $\zeta=(L_0/2)\min\{\eta^{\rho},(L_0/2C)^{1/\delta}\},A_n,B_n$ are as defined in \Cref{lem:order_stat_behav}. The above follows by observing that $\min\{H(\epsilon)-H(0), H(0)-H(-\epsilon) \} \geq \min \{L_{-1}\epsilon-C\epsilon^{1+\delta}, L_1\epsilon-C\epsilon^{1+\delta} \}= L_0\epsilon-C\epsilon^{1+\delta}$ and accordingly modifying \Cref{lem:order_stat_behav}. Following the same steps as in \Cref{thm:fin_sample_non_bahadur} we obtain that, with probability greater than or equal to $(1-(4+2D_2)n^{-2})\mathbbm{1}\{A_n < \zeta \}\mathbbm{1}\{K_n\}$ the following events occur for all $n \geq \log_2(2/\alpha)$,
\begin{equation*}
    \begin{split}
        (Q_{1,\boldsymbol{X},n}-\gamma_n) \leq T_{(k_{n,\alpha}+1)}[L_{-1}\textbf{1}\{T_{(k_{n,\alpha}+1)}<0 \} + L_1\textbf{1}\{ T_{(k_{n,\alpha}+1)} >0\}] \leq (Q_{1,\boldsymbol{X},n}+\gamma_n),\\
        (Q_{2,\boldsymbol{X},n}-\gamma_n) \leq T_{(n-k_{n,\alpha})}[L_{-1}\textbf{1}\{T_{(n-k_{n,\alpha})}<0 \} + L_1\textbf{1}\{ T_{(n-k_{n,\alpha})} >0\}] \leq (Q_{2,\boldsymbol{X},n}+\gamma_n),
    \end{split}
\end{equation*}
where $\gamma_n=D_{1,n}(\log(n)/n)^{3/4}+C(2A_n/L_0)^{1+\delta}$. Note that $M$ is replaced by $L_0$ in the definition of $\gamma_n$. While repeating the derivations in the proofs of \Cref{thm:fin_sample_non_bahadur} and \Cref{thm:asym_result_irreg_case} we shall use this modified version of $\gamma_n$. Using monotonicity of the transformation we obtain that with probability greater than or equal to $(1-(4+2D_2)n^{-2})\mathbbm{1}\{A_n < \zeta \}\mathbbm{1}\{K_n\}$ the following events occur for all $n \geq \log_2(2/\alpha)$,
\begin{equation*}
    \begin{split}
        &|Q_{1,\boldsymbol{X},n}-\gamma_n|^{1/\rho} \mathrm{sgn}(Q_{1,\boldsymbol{X},n}-\gamma_n)\left[\frac{1}{L_{-1}^{1/\rho}}\textbf{1}\{Q_{1,\boldsymbol{X},n}-\gamma_n <0 \} + \frac{1}{L_{1}^{1/\rho}}\textbf{1}\{Q_{1,\boldsymbol{X},n}-\gamma_n >0 \} \right] \\
        \leq & X_{(k_{n,\alpha}+1)}-\theta_0 \\
        \leq & |Q_{1,\boldsymbol{X},n}+\gamma_n|^{1/\rho} \mathrm{sgn}(Q_{1,\boldsymbol{X},n}+\gamma_n)\left[\frac{1}{L_{-1}^{1/\rho}}\textbf{1}\{Q_{1,\boldsymbol{X},n}+\gamma_n <0 \} + \frac{1}{L_{1}^{1/\rho}}\textbf{1}\{Q_{1,\boldsymbol{X},n}+\gamma_n >0 \} \right],\\
        & \quad \mbox{and} \\
        &|Q_{2,\boldsymbol{X},n}-\gamma_n|^{1/\rho} \mathrm{sgn}(Q_{2,\boldsymbol{X},n}-\gamma_n)\left[\frac{1}{L_{-1}^{1/\rho}}\textbf{1}\{Q_{2,\boldsymbol{X},n}-\gamma_n <0 \} + \frac{1}{L_{1}^{1/\rho}}\textbf{1}\{Q_{2,\boldsymbol{X},n}-\gamma_n >0 \} \right] \\
        \leq & X_{(n-k_{n,\alpha})}-\theta_0 \\
        \leq &|Q_{2,\boldsymbol{X},n}+\gamma_n|^{1/\rho} \mathrm{sgn}(Q_{2,\boldsymbol{X},n}+\gamma_n)\left[\frac{1}{L_{-1}^{1/\rho}}\textbf{1}\{Q_{2,\boldsymbol{X},n}+\gamma_n <0 \} + \frac{1}{L_{1}^{1/\rho}}\textbf{1}\{Q_{2,\boldsymbol{X},n}+\gamma_n >0 \} \right].
    \end{split}
\end{equation*}
We obtain that with probability greater than or equal to $(1-(4+2D_2)n^{-2})\mathbbm{1}\{A_n < \zeta \}\mathbbm{1}\{K_n\}$ the following happens for all $n \geq \log_2(2/\alpha)$,
\begin{equation*}
    \begin{split}
 &|Q_{1,\boldsymbol{X},n}-\gamma_n|^{1/\rho} \mathrm{sgn}(Q_{1,\boldsymbol{X},n}-\gamma_n)\left[\frac{1}{L_{-1}^{1/\rho}}\textbf{1}\{Q_{1,\boldsymbol{X},n}-\gamma_n <0 \} + \frac{1}{L_{1}^{1/\rho}}\textbf{1}\{Q_{1,\boldsymbol{X},n}-\gamma_n >0 \} \right] \\
 -& |Q_{2,\boldsymbol{X},n}+\gamma_n|^{1/\rho} \mathrm{sgn}(Q_{2,\boldsymbol{X},n}+\gamma_n)\left[\frac{1}{L_{-1}^{1/\rho}}\textbf{1}\{Q_{2,\boldsymbol{X},n}+\gamma_n <0 \} + \frac{1}{L_{1}^{1/\rho}}\textbf{1}\{Q_{2,\boldsymbol{X},n}+\gamma_n >0 \} \right] \\
 \leq & X_{(k_{n,\alpha}+1)} - X_{(n-k_{n,\alpha})} \\
 \leq & |Q_{1,\boldsymbol{X},n}+\gamma_n|^{1/\rho} \mathrm{sgn}(Q_{1,\boldsymbol{X},n}+\gamma_n)\left[\frac{1}{L_{-1}^{1/\rho}}\textbf{1}\{Q_{1,\boldsymbol{X},n}+\gamma_n <0 \} + \frac{1}{L_{1}^{1/\rho}}\textbf{1}\{Q_{1,\boldsymbol{X},n}+\gamma_n >0 \} \right] \\
 -& |Q_{2,\boldsymbol{X},n}-\gamma_n|^{1/\rho} \mathrm{sgn}(Q_{2,\boldsymbol{X},n}-\gamma_n)\left[\frac{1}{L_{-1}^{1/\rho}}\textbf{1}\{Q_{2,\boldsymbol{X},n}-\gamma_n <0 \} + \frac{1}{L_{1}^{1/\rho}}\textbf{1}\{Q_{2,\boldsymbol{X},n}-\gamma_n >0 \} \right].
    \end{split}
\end{equation*}
Let us define the following function,
\[
\Tilde{h}(x) = \frac{1}{L_{-1}^{1/\rho}}\textbf{1}\{x<0 \} + \frac{1}{L_1^{1/\rho}}\textbf{1}\{x>0 \} .
\]
It can be easily checked that the function $\Tilde{h}(\cdot)$ is holder continuous and that the following holds for all $x,y \in \mathbb{R}$,
\[
|\Tilde{h}(x) -\Tilde{h}(y) | \leq \max\left\{\frac{1}{L_{-1}^{1/\rho}},\frac{1}{L_1^{1/\rho}} \right\}|x-y| = \frac{1}{L_0^{1/\rho}}|x-y|.
\]
Using the holder continuity of $h(\cdot)$ and $\Tilde{h}(\cdot)$ we have the following,
\begin{equation*}
    \begin{split}
       & |\Tilde{h}(h(x))- \Tilde{h}(h(y))| \leq \frac{2}{L_0^{1/\rho}}|x-y|^{1/\rho} \quad \mbox{for all $x,y \in \mathbb{R}$ if $1 \leq \rho < \infty$},\\
        &|\Tilde{h}(h(x))- \Tilde{h}(h(y))| \leq \frac{C_{\rho, \alpha}}{L_0^{1/\rho}}|x-y| \quad \mbox{for all $x,y \in [-E_0,E_0]$ if $0 \leq \rho < 1$},
    \end{split}
\end{equation*}
Let us define the composition of functions $\Tilde{h}(h(\cdot))$ as $\hat{h}(\cdot)$. We introduce the following notation, 
\begin{equation*}
\begin{split}
    \overline{\mathscr{G}}(a,b) = &  |a|^{1/\rho} \mathrm{sgn}(a)\left[\frac{1}{L_{-1}^{1/\rho}}\textbf{1}\{a <0 \} + \frac{1}{L_{1}^{1/\rho}}\textbf{1}\{a >0 \} \right] \\
    &-|a-b|^{1/\rho} \mathrm{sgn}(a-b)\left[\frac{1}{L_{-1}^{1/\rho}}\textbf{1}\{a < b \} + \frac{1}{L_{1}^{1/\rho}}\textbf{1}\{a >b \} \right].
\end{split}
\end{equation*}
Using the holder continuity of the function $\hat{h}(\cdot)$ (in place of $h(\cdot)$) we can re-do the steps in the proof of \Cref{thm:asym_result_irreg_case} and we obtain that with probability greater than or equal to $(1-(6+2D_2)n^{-2})\mathbbm{1}\{A_n < \zeta \}\mathbbm{1}\{K_n\}$ the following event holds true for $\rho \geq 1$,
\begin{equation*}
\begin{split}
 & n^{1/(2\rho)}\left|\mathrm{Width}(\widehat{\mathrm{CI}}_{n,\alpha}) - \overline{\mathscr{G}}(Q_{1,\boldsymbol{X},n},(z_{\alpha/2}/\sqrt{n})) \right| \\
 \leq &  \frac{4}{L_0^{1/\rho}}(\sqrt{n}\gamma_n)^{1/\rho} + \frac{2}{L_0^{1/\rho}}\left( \frac{z_{\alpha/2}^3}{5n} + \frac{4.2}{\sqrt{n}}+ \sqrt{\frac{6(3+\sqrt{n}z_{\alpha/2})\log(n)}{n}}\right)^{1/\rho}.
  \end{split}
\end{equation*}
We also obtain that with probability greater than or equal to $(1-(10+2D_2)n^{-2})\mathbbm{1}\{A_n < \zeta \}\mathbbm{1}\{K_n\}$ the following event holds true for $0< \rho < 1$,
\begin{equation*}
\begin{split}
& n^{1/(2\rho)}\left|\mathrm{Width}(\widehat{\mathrm{CI}}_{n,\alpha}) - \overline{\mathscr{G}}(Q_{1,\boldsymbol{X},n},(z_{\alpha/2}/\sqrt{n}))  \right|  \\
\leq & \frac{2C_{\rho, \alpha}}{L_0^{1/\rho}}\sqrt{n}\gamma_n + \frac{C_{\rho, \alpha}}{L_0^{1/\rho}}\left( \frac{z_{\alpha/2}^3}{5n} + \frac{4.2}{\sqrt{n}}+ \sqrt{\frac{6(3+\sqrt{n}z_{\alpha/2})\log(n)}{n}}\right).
\end{split}
\end{equation*}
By continuous mapping theorem we know that the following holds,
\[
n^{1/(2\rho)}\overline{\mathscr{G}}(Q_{1,\boldsymbol{X},n},(z_{\alpha/2}/\sqrt{n})) \stackrel{d}{\xrightarrow{}} \overline{\mathscr{G}}(W_1,z_{\alpha/2}),
\]
where $W_1 \sim N(z_{\alpha/2}/2,1/4)$. The proof of the theorem is completed by dividing the above distributional convergence result with the width of the oracle confidence interval $2^{-1/\rho}n^{-1/(2\rho)}z_{\alpha/2}^{1/\rho}[M_-^{-1/\rho} + M_+^{-1/\rho} ]$ after suitable scaling. 
    
\end{proof}

\section{Proof of Theorem~\ref{thm:coverage_ghulc}}
\label{appendix:thm:coverage_ghulc}
The estimators $\widehat \theta_j$ ($1\leq j \leq B$) are independent with maximum median-bias $\mathcal{E}_B$. We observe the following, 
\begin{equation*}
    \begin{split}
        \mathbb{P}(\theta_0 \notin \widehat{\mathrm{CI}}_{N,\alpha}^{\mathtt{GHulC}}) &= \mathbb{P}(\widehat \theta_{(\lfloor B/2 \rfloor - c^*_{B,\alpha})} > \theta_0) + \mathbb{P}(\widehat \theta_{(\lceil B/2 \rceil + c^*_{B,\alpha} + 1)} < \theta_0) \\
        &= \mathbb{P}(\sum_{j=1}^B \textbf{1}\{\widehat \theta_j \leq \theta_0 \} < \lfloor B/2 \rfloor - c^*_{B,\alpha}) + \mathbb{P}(\sum_{j=1}^B \textbf{1}\{\widehat \theta_j \geq \theta_0 \} < \lfloor B/2 \rfloor - c^*_{B,\alpha}) \\
        & = \tau_{\alpha} \left[\mathbb{P}(\sum_{j=1}^B \textbf{1}\{\widehat \theta_j \leq \theta_0 \} < \lfloor B/2 \rfloor - (c_{B,\alpha} - 1)) + \mathbb{P}(\sum_{j=1}^B \textbf{1}\{\widehat \theta_j \geq \theta_0 \} < \lfloor B/2 \rfloor - (c_{B,\alpha} - 1)) \right]\\
        & + (1 - \tau_{\alpha}) \left[ \mathbb{P}(\sum_{j=1}^B \textbf{1}\{\widehat \theta_j \leq \theta_0 \} < \lfloor B/2 \rfloor - c_{B,\alpha} ) + \mathbb{P}(\sum_{j=1}^B \textbf{1}\{\widehat \theta_j \geq \theta_0 \} < \lfloor B/2 \rfloor - c_{B,\alpha} ) \right] \\
        &= \tau_{\alpha} \left[ \mathbb{P}(\sum_{j=1}^B Y_j < \lfloor B/2 \rfloor - (c_{B,\alpha} - 1)) + \mathbb{P}(\sum_{j=1}^B Z_j < \lfloor B/2 \rfloor - (c_{B,\alpha} - 1)) \right] \\
        & + (1 - \tau_{\alpha}) \left[ \mathbb{P}(\sum_{j=1}^B Y_j < \lfloor B/2 \rfloor - c_{B,\alpha}) + \mathbb{P}(\sum_{j=1}^B Z_j < \lfloor B/2 \rfloor - c_{B,\alpha}) \right],
    \end{split}   
\end{equation*}
where $Y_j = \textbf{1}\{\widehat \theta_j \leq \theta_0 \} \sim \mbox{Ber}(p_j), \mbox{ } Z_j = \textbf{1}\{\widehat \theta_j \geq \theta_0 \} \sim \mbox{Ber}(q_j)$ for $1 \leq j \leq B$. Moreover we have the additional constraints that $p_j + q_j \geq 1$, $p_j \geq 1/2 - \mathcal{E}_B$, $q_j \geq 1/2 - \mathcal{E}_B$ for $ 1\leq j \leq B$. We see that $\sum_{j=1}^B Y_j \sim \mbox{Poi-Bin}(p_1,\cdots,p_B)$ and $\sum_{j=1}^B Z_j \sim \mbox{Poi-Bin}(q_1,\cdots,q_B)$. Therefore we can say that, 
\begin{equation*}
    \begin{split}
   \mathbb{P}(\theta_0 \notin \widehat{\mathrm{CI}}_{N,\alpha}^{\mathtt{GHulC}}) & =   \tau_{\alpha} \left[ \mathbb{P}(\sum_{j=1}^B Y_j < \lfloor B/2 \rfloor - (c_{B,\alpha} - 1)) + \mathbb{P}(\sum_{j=1}^B Z_j < \lfloor B/2 \rfloor - (c_{B,\alpha} - 1)) \right] \\
        & + (1 - \tau_{\alpha}) \left[ \mathbb{P}(\sum_{j=1}^B Y_j < \lfloor B/2 \rfloor - c_{B,\alpha}) + \mathbb{P}(\sum_{j=1}^B Z_j < \lfloor B/2 \rfloor - c_{B,\alpha}) \right] \\
   &\leq \tau_{\alpha} \left[  \sup_{p_j + q_j \geq 1; \mbox{ } p_j,q_j \geq 1/2 - \mathcal{E}_B} f_1(p_1,\cdots,p_B) + f_1(q_1,\cdots,q_B) \right] \\
   & + (1 - \tau_{\alpha}) \left[  \sup_{p_j + q_j \geq 1; \mbox{ } p_j,q_j \geq 1/2 - \mathcal{E}_B} f_2(p_1,\cdots,p_B) + f_2(q_1,\cdots,q_B) \right],
    \end{split}
\end{equation*}
where $f_1(p_1,\cdots,p_B)$ is the probability that a $\mbox{Poi-Bin}(p_1,\cdots,p_B)$ is less than $\lfloor B/2 \rfloor - (c_{B,\alpha} - 1)$ and $f_2(p_1,\cdots,p_B)$ is the probability that a $\mbox{Poi-Bin}(p_1,\cdots,p_B)$ is less than $\lfloor B/2 \rfloor - c_{B,\alpha} $. Observe that,
\[
\mathbb{P}\left( \mbox{Poi-Bin}(p_1,\cdots,p_B) = k \right) = \sum_{A \in F_k} \prod_{i \in A}p_i \prod_{j \in A^c}(1 - p_j),
\]
where $F_k$ is the collection of all $k$-subsets of $\{1,\cdots, B\}$. Thus $f_1(p_1,\cdots,p_B)$ is a linear function in each of the $p_j$'s, in particular it is a linear function in $p_1$. Let $f_1(p_1,\cdots,p_B) = ap_1$ ($a>0$). Similarly we let $f_1(q_1,\cdots,q_B) = bq_1$ ($b >0$) as $f_1(q_1,\cdots,q_B)$ is a linear function in $q_1$. We concentrate on the sub-problem of maximizing $ap_1 + bq_1$ under the constraints $p_1 + q_1 \geq 1$, $p_1 \geq 1/2 - \mathcal{E}_B$, and $q_1 \geq 1/2 - \mathcal{E}_B$. This is a classical convex optimization problem and can be solved easily using Lagrange multipliers. It can be shown that the maxima is attained only at the boundary points i.e. either $p_1 = 1/2 - \mathcal{E}_B, q_1 = 1/2 + \mathcal{E}_B$ or $p_1 = 1/2 + \mathcal{E}_B, q_1 = 1/2 - \mathcal{E}_B$. The same can be said about each $p_j,q_j$ for $1 \leq j \leq B$. 

We note that the number of times $p_j$ appear in $f_1(p_1,\cdots,p_B)$ is same as number of times $q_j$ appear in $f_1(q_1,\cdots,q_B)$. Therefore in the optimum case, the sum $f_1(p_1,\cdots,p_B) + f_1(q_1,\cdots,q_B)$ is a sum of product of some permutations of sequences containing equal number of $1/2 - \mathcal{E}_B$ and $1/2 + \mathcal{E}_B$. Using the result in \cite{ruderman1952two} we reach to the conclusion that the maximum is attained when all the sequences are arranged in the same order i.e.\ either $p_1 = \cdots = p_B = 1/2 - \mathcal{E}_B$, $q_1 = \cdots = q_B = 1/2 + \mathcal{E}_B$ or $p_1 = \cdots = p_B = 1/2 + \mathcal{E}_B$, $q_1 = \cdots = q_B = 1/2 - \mathcal{E}_B$. Therefore we have the following bound using \Cref{thm:quantile_coverage} (we have $B$ in place of $n$), 
\begin{equation*}
    \begin{split}
        &  \sup_{p_j + q_j \geq 1; \mbox{ } p_j,q_j \geq 1/2 - \mathcal{E}_B} f_1(p_1,\cdots,p_B) + f_1(q_1,\cdots,q_B)  \\
      = &   \mathbb{P}(\mbox{Bin}(B, 1/2 - \mathcal{E}_B) < \lfloor B/2 \rfloor - (c_{B,\alpha} - 1)) +   \mathbb{P}(\mbox{Bin}(B, 1/2 +\mathcal{E}_B) < \lfloor B/2 \rfloor - (c_{B,\alpha} - 1))  \\
   \leq  &  (1 - P(B, c_{B,\alpha} - 1) )\left(1 + 2B(B-1)(1+2\mathcal{E}_B)^{B-2}\mathcal{E}_B^2 \right). 
    \end{split}
\end{equation*}
Similarly we can show that, 
\begin{equation*}
    \begin{split}
        &  \sup_{p_j + q_j \geq 1; \mbox{ } p_j,q_j \geq 1/2 - \mathcal{E}_B} f_2(p_1,\cdots,p_B) + f_2(q_1,\cdots,q_B)  \\
      = &   \mathbb{P}(\mbox{Bin}(B, 1/2 - \mathcal{E}_B) < \lfloor B/2 \rfloor - c_{B,\alpha}) +   \mathbb{P}(\mbox{Bin}(B, 1/2 +\mathcal{E}_B) < \lfloor B/2 \rfloor - c_{B,\alpha} )  \\
   \leq  &  (1 - P(B, c_{B,\alpha} )) \left(1 + 2B(B-1)(1+2\mathcal{E}_B)^{B-2}\mathcal{E}_B^2 \right). 
    \end{split}
\end{equation*}
The miscoverage probability can be bounded above using these results, 
\begin{equation}
    \begin{split}
    \mathbb{P}(\theta_0 \notin \widehat{\mathrm{CI}}_{N,\alpha}^{\mathtt{GHulC}}) & \leq \tau_{\alpha} \left[  \sup_{p_j + q_j \geq 1; \mbox{ } p_j,q_j \geq 1/2 - \mathcal{E}_B} f_1(p_1,\cdots,p_B) + f_1(q_1,\cdots,q_B) \right] \\
   & + (1 - \tau_{\alpha}) \left[  \sup_{p_j + q_j \geq 1; \mbox{ } p_j,q_j \geq 1/2 - \mathcal{E}_B} f_2(p_1,\cdots,p_B) + f_2(q_1,\cdots,q_B) \right] \\
    & \leq \left[\tau_{\alpha} (1 -  P(B, c_{B,\alpha} - 1)) + (1 - \tau_{\alpha}) (1 - P(B, c_{B,\alpha} )) \right] \left(1 + 2B(B-1)(1+2\mathcal{E}_B)^{B-2}\mathcal{E}_B^2 \right) \\
     & =  \alpha \left(1 + 2B(B-1)(1+2\mathcal{E}_B)^{B-2}\mathcal{E}_B^2 \right) \\
    & \leq \alpha\left( 1 + 2 B^2 \mathcal{E}_B^2 e^{2B\mathcal{E}_B} \right).
    \end{split}
\end{equation}
This completes the proof of the first part of the result. Under the additional assumption that $\mathbb P(\widehat \theta_j \leq \theta_0) = \mathbb P(\widehat \theta_1 \leq \theta_0) $ and $\mathbb P(\widehat \theta_j = \theta_0) = 0$ for all $j \geq 1$, we can obtain a lower bound on the miscoverage probability. Suppose $\mathbb P(\widehat \theta_1 \leq \theta_0)  = p$ where $ 0 < p < 1$. Under the given assumption we have $\mathbb P(\widehat \theta_1 \geq \theta_0) =  1 - p = q$ (say). The miscoverage probability can be represented as follows, 
\begin{equation*}
    \begin{split}
        & \mathbb{P}(\theta_0 \notin \widehat{\mathrm{CI}}_{N,\alpha}^{\mathtt{GHulC}}) \\
        = & \tau_{\alpha} \left[ \mathbb{P}(\sum_{j=1}^B Y_j < \lfloor B/2 \rfloor - (c_{B,\alpha} - 1)) + \mathbb{P}(\sum_{j=1}^B Z_j < \lfloor B/2 \rfloor - (c_{B,\alpha} - 1)) \right] \\
        & + (1 - \tau_{\alpha}) \left[ \mathbb{P}(\sum_{j=1}^B Y_j < \lfloor B/2 \rfloor - c_{B,\alpha}) + \mathbb{P}(\sum_{j=1}^B Z_j < \lfloor B/2 \rfloor - c_{B,\alpha}) \right] \\
        = & \tau_{\alpha} \left[ \mathbb{P}(\mbox{Bin}(B, p) < \lfloor B/2 \rfloor - (c_{B,\alpha} - 1)) + \mathbb{P}(\mbox{Bin}(B, q) < \lfloor B/2 \rfloor - (c_{B,\alpha} - 1)) \right] \\
        & + (1 - \tau_{\alpha}) \left[ \mathbb{P}(\mbox{Bin}(B, p) < \lfloor B/2 \rfloor - c_{B,\alpha}) + \mathbb{P}(\mbox{Bin}(B, q) < \lfloor B/2 \rfloor - c_{B,\alpha}) \right] \\
        = &  \tau_{\alpha} \mathbb{P}(\mbox{Bin}(B, p) \notin \left[\lfloor B/2 \rfloor - (c_{B,\alpha} - 1), \lceil B/2 \rceil + (c_{B,\alpha} - 1) \right])  \\
        & + (1 - \tau_{\alpha}) \mathbb{P}(\mbox{Bin}(B, p) \notin \left[\lfloor B/2 \rfloor - c_{B,\alpha}, \lceil B/2 \rceil + c_{B,\alpha}  \right]) \\
        = & \tau_{\alpha} f_3(p, c_{B,\alpha} - 1) + ( 1 - \tau_{\alpha}) f_3(p, c_{B,\alpha} ),
    \end{split}
\end{equation*}
where $f_3(p, c) = \mathbb P(Z \notin [\lfloor B/2 \rfloor - c, \lceil B/2 \rceil + c]) = 1 - g_3(p, c)$ where $Z \sim \mbox{Bin}(B, p)$ and $g_3(p, c) = \mathbb P(Z \in [\lfloor B/2 \rfloor - c, \lceil B/2 \rceil + c])$. We aim to minimize the function $p \mapsto f_3(p, c)$ (or equivalently maximize the function $p \mapsto g_3(p, c)$). We can rewrite $g_3(p, c)$ as, 
\[
g_3(p, c) = \mathbb{P}\{L\le Z\le U\}=F_U(p)-F_{L-1}(p),
\]
where $L = \lfloor B/2 \rfloor - c, U = \lceil B/2 \rceil + c$, and $F_m(p)=\mathbb{P}(Z\le m)$. We will use the following identity that can be easily checked for $m\in\{0,1,\dots,B\}$,
\begin{equation}\label{eq:binom_cdf_deriv}
\frac{\partial}{\partial p}F_m(p)=-B\,\mathbb{P}\{Y=m\},\quad\mbox{where} \quad Y\sim \mbox{Bin}(B-1,p).
\end{equation}
Using \eqref{eq:binom_cdf_deriv} we have the following, 
\begin{equation}
\label{eq:imp_deriv_coverage}
    \begin{split}
        \frac{\partial}{\partial p}g_3(p, c) = & \frac{\partial}{\partial p}F_U(p) -   \frac{\partial}{\partial p}F_{L-1}(p) \\
        = & -B \mathbb{P}\{Y=U\} + B \mathbb{P}\{Y=L - 1\}. 
    \end{split}
\end{equation}
We now observe the following, 
\[
\frac{\mathbb{P}\{Y=U\}}{\mathbb{P}\{Y=L - 1\}} = \frac{\binom{B-1}{U}p^U(1 - p)^{B - 1- U}}{\binom{B-1}{L-1}p^{L-1}(1 - p)^{B-L}} = \left( \frac{p}{1 - p} \right)^{U - (L-1)}. 
\]
In the above equality we have used the fact that $U + L = B$. Therefore the ratio $\mathbb{P}\{Y=U\}/\mathbb{P}\{Y=L - 1\} $ is less than $1$ if $p < 1/2$, equal to $1$ if $ p = 1/2$ and greater than $1$ if $p > 1/2$. Combining this observation with \eqref{eq:imp_deriv_coverage} we can say that the function $p \mapsto g_3(p, c)$ increases up to $p = 1/2$ and decreases after $p=1/2$. In other words the function $p \mapsto g_3(p, c)$ is maximized (equivalently the function $p \mapsto f_3(p, c)$ is minimized) at $p = 1/2$. Hence we can lower bound the miscoverage probability as follows, 
\begin{equation*}
    \begin{split}
        & \mathbb{P}(\theta_0 \notin \widehat{\mathrm{CI}}_{N,\alpha}^{\mathtt{GHulC}}) \\
        = & \tau_{\alpha} f_3(p, c_{B,\alpha} - 1) + ( 1 - \tau_{\alpha}) f_3(p, c_{B,\alpha} ) \\
        \geq & \tau_{\alpha} f_3(1/2, c_{B,\alpha} - 1) + ( 1 - \tau_{\alpha}) f_3(1/2, c_{B,\alpha} ) \\
        = & \tau_{\alpha} (1 - P(B, c_{B,\alpha} - 1)) + ( 1 - \tau_{\alpha})(1 - P(B, c_{B,\alpha} )) \\
        = & \alpha. 
    \end{split}
\end{equation*}
The last equality follows from the definition of $\tau_{\alpha}$. This completes the proof of the lower bound on miscoverage probability under the stated assumptions. 

\section{Proof of Theorem~\ref{thm:easy_width_ghulc}}
\label{appendix:easy_ghulc}
Suppose the data has been split into $B$ disjoint sets of equal size. Let $\mathcal{F}$ be the distribution function of $\widehat \theta_j - \theta_0$ for $1 \leq j \leq B$. Suppose the following holds true, 
\begin{equation*}
 |\mathcal{F}(x) - \mathcal{F}(0)| > C_1 |x|^{\rho} \mbox{  for }|x| < \Delta, 
\end{equation*}
where $\rho, C_1,C_2 > 0$. We also assume the following, 
\[
\left| \mathcal{F}(0) - \frac{1}{2} \right| \leq \mathcal{E}_B.
\]
We know from Lemma 3.1.1 of \cite{reiss2012approximate} that the following one-sided concentration bounds hold true, 
\begin{equation}
\label{eq:reiss_100}
\begin{split}
      \mathbb{P}\left(\frac{\sqrt{n}}{\sqrt{\frac{r}{n+1}\left( 1- \frac{r}{n+1} \right)}}\left( U_{r:n} - \frac{r}{n+1} \right) \leq 5\xi \right) \geq & 1 - \exp(-\xi), \\
        \mathbb{P}\left(\frac{\sqrt{n}}{\sqrt{\frac{r}{n+1}\left( 1- \frac{r}{n+1} \right)}}\left( \frac{r}{n+1}  -U_{r:n} \right) \leq 5\xi \right) \geq & 1 - \exp(-\xi).
\end{split}
\end{equation}
We observe the following for the $r$-th order statistic $\widehat \theta_{(r)}$ ($1 \leq r \leq B$), 
\begin{equation*}
    \begin{split}
        &\mathbb{P}(\widehat \theta_{(r)} - \theta_0 \geq \upsilon) \\
        =& \mathbb{P}(\mathcal{F}(\theta_{(r)} - \theta_0) \geq \mathcal{F}(\upsilon)) \\
        =& \mathbb{P}\left( \frac{\sqrt{B}}{\sqrt{\frac{r}{B+1}\left(1 - \frac{r}{B+1} \right)}}\left(U_{r:B} - \frac{r}{B+1} \right) \geq \frac{\sqrt{B}}{\sqrt{\frac{r}{B+1}\left(1 - \frac{r}{B+1} \right)}} \left( \mathcal{F}(\upsilon) - \frac{r}{B+1}\right) \right).
    \end{split}
\end{equation*}
Let $r^- = \lfloor B/2 \rfloor - c_{B,\alpha} = B - k_{B,\alpha}$ and $r^+ = \lceil B/2 \rceil + c_{B,\alpha} + 1= k_{B,\alpha} + 1$. Then we have, 
\begin{equation*}
    \begin{split}
        & \frac{\sqrt{B}}{\sqrt{\frac{r^+}{B+1}\left(1 - \frac{r^+}{B+1} \right)}} \left( \mathcal{F}(\upsilon) - \frac{r^+}{B+1}\right) \\
        \geq& \frac{\sqrt{B}}{\sqrt{\frac{r^+}{B+1}\left(1 - \frac{r^+}{B+1} \right)}} \left( \mathcal{F}(\upsilon) - (\mathcal{F}(0) - (1/2 - \mathcal{E}_B)) - \frac{r^+}{B+1}\right) \\
        \geq & \frac{\sqrt{B}}{\sqrt{\frac{r^+}{B+1}\left(1 - \frac{r^+}{B+1} \right)}} \left( (\mathcal{F}(\upsilon) - \mathcal{F}(0)) + (1/2 - \mathcal{E}_B) - \frac{r^+}{B+1}\right) \\
        \geq& \frac{\sqrt{B}}{\sqrt{\frac{r^+}{B+1}\left(1 - \frac{r^+}{B+1} \right)}} \left( C_1 \upsilon^{\rho} + (1/2 - \mathcal{E}_B) - \frac{r^+}{B+1}\right) \mbox{ provided } |\upsilon| < \Delta.
    \end{split}
\end{equation*}
The final quantity in the above analysis is greater than or equal to $5 \xi$ if,
\begin{equation*}
    \begin{split}
        \upsilon &= \frac{1}{C_1^{1/\rho}}\left\{\frac{\sqrt{\frac{r^+}{B+1}\left(1 - \frac{r^+}{B+1} \right)}}{\sqrt{B}} 5 \xi - \frac{1}{2} + \mathcal{E}_B + \frac{r^+}{B+1}  \right\}^{1/\rho} \\
        &\leq \frac{1}{C_1^{1/\rho}}\left\{\frac{5\xi}{2\sqrt{B}} + \frac{z_{\alpha/2}}{2\sqrt{B}} + \frac{2}{B} + \mathcal{E}_B \right\}^{1/\rho} \\
        &\leq \frac{1}{C_1^{1/\rho}}\left\{\frac{5\xi + \sqrt{2\log(2/\alpha)}}{2\sqrt{B}} + \frac{2}{B} + \mathcal{E}_B \right\}^{1/\rho}\\
        &= c_+ \mbox{ (say)}.
    \end{split}
\end{equation*}
Thus using \eqref{eq:reiss_100} we can say that, 
\begin{equation}
    \label{eq:c+}
    \mathbb{P}(\widehat \theta_{(r^+)} - \theta_0 \leq c_+) \geq 1 - \exp(-\xi)\quad  \mbox{provided} \quad c_+ < \Delta.  
\end{equation}
We can also put an upper bound in the following way,
\begin{equation*}
    \begin{split}
        & \frac{\sqrt{B}}{\sqrt{\frac{r^-}{B+1}\left(1 - \frac{r^-}{B+1} \right)}} \left( \mathcal{F}(\upsilon) - \frac{r^-}{B+1}\right) \\
        \leq& \frac{\sqrt{B}}{\sqrt{\frac{r^-}{B+1}\left(1 - \frac{r^-}{B+1} \right)}} \left( \mathcal{F}(\upsilon) + ((1/2 + \mathcal{E}_B)- \mathcal{F}(0)) - \frac{r^-}{B+1}\right) \\
        \leq & \frac{\sqrt{B}}{\sqrt{\frac{r^-}{B+1}\left(1 - \frac{r^-}{B+1} \right)}} \left( (\mathcal{F}(\upsilon) - \mathcal{F}(0)) + (1/2 + \mathcal{E}_B) - \frac{r^-}{B+1}\right) \\
        \leq& \frac{\sqrt{B}}{\sqrt{\frac{r^-}{B+1}\left(1 - \frac{r^-}{B+1} \right)}} \left( -C_1 (-\upsilon)^{\rho} + (1/2 + \mathcal{E}_B) - \frac{r^-}{B+1}\right) \mbox{ provided } |\upsilon| < \Delta. 
    \end{split}
\end{equation*}
The final quantity in the above analysis is less than or equal to $-5\xi$ if, 
\begin{equation*}
    \begin{split}
        -\upsilon &= \frac{1}{C_1^{1/\rho}}\left\{\frac{\sqrt{\frac{r^-}{B+1}\left(1 - \frac{r^-}{B+1} \right)}}{\sqrt{B}} 5 \xi + \frac{1}{2} + \mathcal{E}_B - \frac{r^-}{B+1}  \right\}^{1/\rho} \\
        &\leq \frac{1}{C_1^{1/\rho}}\left\{\frac{5\xi}{2\sqrt{B}} + \frac{z_{\alpha/2}}{2\sqrt{B}} + \frac{3}{2B} + \mathcal{E}_B  \right\}^{1/\rho}\\
        &\leq \frac{1}{C_1^{1/\rho}}\left\{\frac{5\xi + \sqrt{2\log(2/\alpha)}}{2\sqrt{B}} + \frac{3}{2B} + \mathcal{E}_B  \right\}^{1/\rho}\\
        &= c_- \mbox{ (say)}.
    \end{split}
\end{equation*}
Thus using \eqref{eq:reiss_100} we can say that, 
\begin{equation}
    \label{eq:c-}
    \mathbb{P}(\widehat \theta_{(r^-)} - \theta_0 \geq -c_-) \geq 1 - \exp(-\xi)\quad  \mbox{provided} \quad c_- < \Delta.  
\end{equation}
Combining \eqref{eq:c+} and \eqref{eq:c-} we can say that the following holds with probability at-least $1 - \delta$ (we substitute $\xi = \log(2/\delta)$),
\begin{equation*}
    \begin{split}
        & \quad\mathrm{Width}(\widehat{\mathrm{CI}}_{N,\alpha}^{\mathtt{GHulC}}) \\
        &= \widehat \theta_{(r^+)} - \widehat \theta_{(r^-)} \\
        &= (\widehat \theta_{(r^+)} - \theta_0) - (\widehat \theta_{(r^-)} - \theta_0) \\
        & \leq c_+ + c_- \\
        &= \frac{1}{C_1^{1/\rho}} \left[\left\{\frac{5\log(2/\delta) + \sqrt{2\log(2/\alpha)}}{2\sqrt{B}} + \frac{2}{B} + \mathcal{E}_B \right\}^{1/\rho} +  \left\{\frac{5\log(2/\delta) + \sqrt{2\log(2/\alpha)}}{2\sqrt{B}} + \frac{3}{2B} + \mathcal{E}_B  \right\}^{1/\rho}\right] \\
        &< \frac{2}{C_1^{1/\rho}} \left\{\frac{5\log(2/\delta) + \sqrt{2\log(2/\alpha)}}{2\sqrt{B}} + \frac{2}{B} + \mathcal{E}_B \right\}^{1/\rho}.
    \end{split}
\end{equation*}
provided the last expression is less than $2\Delta$. 

Now we return to the condition given in the theorem. The sample-size of each of the splits is approximately $N/B$. The distribution function $\Tilde{\mathcal{F}}$ (this is same as $\Tilde F_{N/B}$ mentioned in the theorem) of $r_{N/B}(\widehat \theta_j - \theta_0)$ satisfies the following, 
\[
|\Tilde{\mathcal{F}}(x) - \Tilde{\mathcal{F}}(0)| > \mathscr{C}|x|^{\rho} \quad \mbox{ for } |x| < \Tilde{\Delta},
\]
for some $\mathcal{C}, \Tilde{\Delta} > 0$. Note the following relation between $\mathcal{F}$ and $\Tilde{\mathcal{F}}$, $\mathcal{F}(x) = \mathbb{P}(\widehat \theta_j - \theta_0 \leq x)$ which is same as $\mathbb{P}(r_{N/B}(\widehat \theta_j - \theta_0) \leq r_{N/B} x) = \Tilde{\mathcal{F}}(r_{N/B}x)$. Thus the condition on $\Tilde{\mathcal{F}}$ translates to the following condition on $\mathcal{F}$, 
\[
|\mathcal{F}(x) - \mathcal{F}(0) | > C_1 |x|^{\rho} \quad \mbox{for } |x| < \Delta,
\]
where $C_1 = \mathscr{C}r_{N/B}^{\rho}$ and $\Delta = \Tilde{\Delta}/r_{N/B}$. Using our previous analysis we obtain that the following event holds with probability at-least $1 - \delta$, 
\begin{equation*}
    \begin{split}
        &  \mathrm{Width}(\widehat{\mathrm{CI}}_{N,\alpha}^{\mathtt{GHulC}})  \\
        < & \frac{2}{C_1^{1/\rho}} \left\{\frac{5\log(2/\delta) + \sqrt{2\log(2/\alpha)}}{2\sqrt{B}} + \frac{2}{B} + \mathcal{E}_B \right\}^{1/\rho} \\
        = & \frac{2}{\mathscr{C}^{1/\rho}r_{N/B}} \left\{\frac{5\log(2/\delta) + \sqrt{2\log(2/\alpha)}}{2\sqrt{B}} + \frac{2}{B} + \mathcal{E}_B \right\}^{1/\rho},
    \end{split}
\end{equation*}
provided the last expression is less than $2\Delta = 2 \Tilde{\Delta}/r_{N/B}$. This completes the proof of the result.

\section{Proof of Theorem~\ref{thm:width_ghulc}}
\label{appendix:width_ghulc}
We state and prove a stronger version of the result stated in \Cref{thm:width_ghulc}.
\begin{theorem}
\label{thm:width_ghulc_1}
Suppose $ \widehat{\mathrm{CI}}_{N,\alpha}^{\mathtt{GHulC}}$ is the confidence interval returned by GHulC (\Cref{alg:ghulc}) using approximately equal $B$ splits. Let $\widehat \theta^m$ be an estimator of $\theta_0$ based on a sample of size $m$ and let $r_m$ be its rate of convergence i.e., 
\[
r_m(\widehat \theta^m - \theta_0) = O_p(1) \quad \mbox{ as } m \rightarrow \infty. 
\]
We assume the following regarding the distribution function $\Tilde F_{N/B}(\cdot)$ of $r_{N/B}(\widehat \theta_j^{N/B} - \theta_0)$, 
\[
\left|\Tilde F_{N/B}(t) - \Tilde F_{N/B}(0) - M_N|t|^{\rho}\mathrm{sgn}(t) \right| \leq C_N |t|^{\rho+\Delta} \quad \mbox{for } |t|< \eta,
\]
where $0< M_N,C_N,\Delta,\eta,\rho< \infty$. Let $|\Tilde F_{N/B}(1/2) - (1/2)| \leq \mathcal{E}_B$. We introduce the following notations,
\begin{equation*}
    \begin{split}
        A_{B,\xi} &= \frac{z_{\alpha/2}}{2\sqrt{B}} + \frac{2}{B} + \sqrt{\frac{\xi}{2(B+1)}} + \left(\frac{z_{\alpha/2}}{\sqrt{B}}+ \frac{4}{B} \right) \frac{\xi}{B+1}, \\
        \zeta_{\rho,N} &= \frac{M_N}{2} \min\left\{\eta^{\rho}, \left(\frac{M_N}{2C_N} \right)^{1/\delta} \right\}, \\
        D_{\xi,n} &= \left\{ \frac{11}{4}\sqrt{\frac{\log(n+2)}{n}} + \frac{\xi + (13/2)}{\sqrt{n\log(n+2)}} \right\}\max\left\{\sqrt{ \frac{5\xi}{2\sqrt{n}}} , \sqrt{\frac{\log(n+2)}{n}}\right\} , \\
        \gamma_{\xi,B,N} &= D_{\xi,B} + \mathcal{E}_B + C_N \left(\frac{2(A_{B,\xi}+\mathcal{E}_B)}{M_N} \right)^{1+\delta}, \\
         Q_B &= \left(\frac{c_{B,\alpha} + \lceil B/2 \rceil+1}{B} - \frac{1}{2}\right) - \left(\frac{1}{B}\sum_{i=1}^B \mathbf{1}\left\{U_i \le \frac{c_{B,\alpha} + \lceil B/2 \rceil+1}{B}\right\} - \frac{c_{B,\alpha} + \lceil B/2 \rceil+1}{B}\right), 
    \end{split}
\end{equation*}
where $U_1,\cdots,U_B$ are iid uniform(0,1) random variables and $\delta = \Delta/\rho$ and
\[
\mathscr{G}(a,b) := |a|^{1/\rho}\mathrm{sgn}(a) - |a - b|^{1/\rho}\mathrm{sgn}(a-b).
\]
Then the following event holds with probability greater than or equal to $(1 - 16\exp(-\xi))\boldsymbol{1}\{A_{B,\xi} + \mathcal{E}_B < \zeta_{\rho,N} \}$, 
\begin{equation*}
    \begin{split}
  &\left|(B)^{1/2\rho}r_{N/B}M_N^{1/\rho}\mathrm{Width}(\widehat{\mathrm{CI}}_{N,\alpha}^{\mathtt{GHulC}}) -  (B)^{1/2\rho}\mathscr{G}(Q_B, z_{\alpha/2}/\sqrt{B})  \right|    \\
  \leq&   \max\{4, 2C_{\rho, \xi, \alpha}\}(\sqrt{B}\gamma_{\xi,B,N})^{\min\{1, 1/\rho\}} + \max\{2, C_{\rho, \xi, \alpha}\}\left( \frac{z_{\alpha/2}^3}{5B} + \frac{4.2}{\sqrt{B}}+ \sqrt{\frac{3(3+\sqrt{B}z_{\alpha/2})\xi}{B}}\right)^{\min\{1, 1/\rho\}},
    \end{split}
\end{equation*}

where $C_{\rho, \xi, \alpha}$ is a constant (see \eqref{eq:crhoxi}). We have the following distributional convergence as $N/B,B \rightarrow \infty$,
\begin{equation*}
    \begin{split}
       (B)^{1/2\rho}r_{N/B}M_N^{1/\rho}\mathrm{Width}(\widehat{\mathrm{CI}}_{N,\alpha}^{\mathtt{GHulC}}) 
    \stackrel{d}{\xrightarrow{}} \mathscr{G}(W, z_{\alpha/2}),
    \end{split}
\end{equation*}
where $W \sim N(z_{\alpha/2}/2,1/4)$.
\end{theorem}

\begin{proof}[Proof of \Cref{thm:width_ghulc_1}]
We know the following regarding uniform order statistics ($U_{k:n}$) from the proof of \Cref{thm:assym_of_cn},
\begin{equation}
\begin{split}
\label{eq:conc_unif_reiss}
   & \mathbb{P}(\Psi(U_{k:n}, k/(n+1)) \geq \xi) \leq 2 \exp(-(n+1)\xi) \\
   \implies & \mathbb{P}((n+1)\Psi(U_{k:n}, k/(n+1)) \geq \xi) \leq 2 \exp(-\xi) \\
   \implies & \mathbb{P}\left\{\left|U_{k:n} - \frac{k}{n+1} \right| \leq \sqrt{\frac{2k}{n+1}\left( 1 - \frac{k}{n+1}\right) \frac{\xi}{n+1}} + \left| 1 - \frac{2k}{n+1}\right|\frac{\xi}{n+1} \right\} \geq 1 - 2\exp(-\xi).
\end{split}
\end{equation}
We denote by $\widehat \theta_j^{N/B}$ ($1 \leq j \leq B$) the estimators of $\theta_0$ based on the equal sized splits of size roughly equal to $N/B$. It is given that $r_{N/B}$ is the rate of convergence of each of the above estimators i.e.\ $r_{N/B}(\widehat \theta_j^{N/B} - \theta_0) = O_P(1)$. The following is known for the distribution function $\Tilde F_{N/B}(\cdot)$ of $r_{N/B}(\widehat \theta_j^{N/B} - \theta_0)$, 
\[
\left|\Tilde F_{N/B}(t) - \Tilde F_{N/B}(0) - M_N|t|^{\rho} \mathrm{sgn}(t) \right| \leq C_N |t|^{\rho+\Delta} \quad \mbox{for } |t|< \eta.
\]
The above condition translates to the following condition for the distribution function $F_{N/B}(\cdot)$ of $\widehat \theta_j^{N/B}$, 
\[
\left| F_{N/B}(\theta_0 + s) -  F_{N/B}(\theta_0) - M_Nr_{N/B}^{\rho}|s|^{\rho} \mathrm{sgn}(s) \right| \leq C_N r_{N/B}^{\rho + \Delta}|s|^{\rho+\Delta} \quad \mbox{for } r_{N/B}|s|< \eta.
\]
We define the distribution function $H_{N/B}(\cdot)$ as follows, $H_{N/B}(t) = F_{N/B}(\theta_0 + |t|^{1/\rho} \mathrm{sgn}(t))$. The condition for $F_{N/B}(\cdot)$ translates to the following condition for $H_{N/B}(\cdot)$,
\begin{equation}
\label{eq:H_result_ghulc}
\left| H_{N/B}(t) -  H_{N/B}(0) - M_Nr_{N/B}^{\rho}t \right| \leq C_N r_{N/B}^{\rho(1 + \delta)}|t|^{1 + \delta} \quad \mbox{for } r_{N/B}^{\rho}|t|< \eta^{\rho},
\end{equation}
where $0< \delta = \Delta/\rho < \infty$. From here on we shall denote $\widehat \theta_j^{N/B}$ by $\widehat \theta_j$ unless otherwise mentioned. Since $\widehat \theta_j \sim F_{N/B}$ we obtain that $H_{N/B}^{-1}(F_{N/B}(\widehat \theta_j)) = |\widehat \theta_j - \theta_0|^{\rho} \mathrm{sgn}(\widehat \theta_j - \theta_0) = T_j \mbox{ (say)} \sim H_{N/B}$. We prove some concentration bounds on uniform order-statistics below which will be used in our proof later on. We note that $H_{N/B}(T_j) \sim U(0,1)$ for $1 \leq j \leq B$. Using \eqref{eq:conc_unif_reiss} we obtain the following with probability greater than or equal to $1 - 2 \exp(-\xi)$,
\begin{equation*}
    \begin{split}
       &\left|H_{N/B}(T_{( k_{B,\alpha}+1)}) - H_{N/B}(0) \right| \\
       \leq& \left|H_{N/B}(T_{( k_{B,\alpha}+1})) - \frac{ k_{B,\alpha}+1}{B+1} \right| + \left|\frac{ k_{B,\alpha}+1}{B+1} - H_{N/B}(0) \right|  \\
        \leq & \sqrt{\frac{2( k_{B,\alpha}+1)}{B+1}\left( 1 - \frac{( k_{B,\alpha}+1)}{B+1}\right) \frac{\xi}{B+1}} + \left| 1 - \frac{2( k_{B,\alpha}+1)}{B+1}\right|\frac{\xi}{B+1} + \left|\frac{ k_{B,\alpha}+1}{B+1} - \frac{1}{2} \right| + \left|\frac{1}{2} - H_{N/B}(0) \right| \\
        \leq & \sqrt{\frac{2( k_{B,\alpha}+1)}{B+1}\left( 1 - \frac{( k_{B,\alpha}+1)}{B+1}\right) \frac{\xi}{B+1}} + \left| 1 - \frac{2( k_{B,\alpha}+1)}{B+1}\right|\frac{\xi}{B+1} + \frac{z_{\alpha/2}}{2\sqrt{B}} + \frac{2}{B} + \mathcal{E}_B \\
        \leq & \frac{z_{\alpha/2}}{2\sqrt{B}} + \frac{2}{B} + \sqrt{\frac{\xi}{2(B+1)}} + \left(\frac{z_{\alpha/2}}{\sqrt{B}}+ \frac{4}{B} \right) \frac{\xi}{B+1} + \mathcal{E}_B \\
       =& A_{B,\xi} \mbox{ (say) } + \mathcal{E}_B,
    \end{split}
\end{equation*}
where $\mathcal{E}_B$ is the median bias of the estimator $\widehat \theta_j^{N/B}$ based on a sub-sample of size approximately $N/(B)$. Similarly we can show that with probability greater than or equal to $1 - 2\exp(-\xi)$ the following holds,
\begin{equation*}
    \begin{split}
       &\left|H_{N/B}(T_{(B -  k_{B,\alpha})}) - H_{N/B}(0) \right| \\ 
       \leq & \frac{z_{\alpha/2}}{2\sqrt{B}} + \frac{3}{2B} + \sqrt{\frac{\xi}{2(B+1)}} + \left(\frac{z_{\alpha/2}}{\sqrt{B}}+ \frac{3}{B} \right) \frac{\xi}{B+1} + \mathcal{E}_B \\
       =& B_{l,\xi} \mbox{ (say) } + \mathcal{E}_B \\
       \leq & A_{B,\xi} \mbox{ (say) } + \mathcal{E}_B. 
    \end{split}
\end{equation*}
Following the steps in the proof of \Cref{lem:order_stat_behav} we obtain the following bound,
\[
\mathbb{P}(|T_{( k_{B,\alpha}+1)}| \leq \epsilon) \geq \mathbb{P}\left(\left|H_{N/B}(T_{( k_{B,\alpha}+1)}) - H_{N/B}(0) \right| \leq \frac{M_Nr_{N/B}^{\rho}\epsilon}{2} \right) \mbox{  } \mbox{for } \epsilon < \frac{1}{r_{N/B}^{\rho}}\min \left\{\eta^{\rho}, \left(\frac{M_N}{2C_N}\right)^{1/\delta} \right\}.
\]
We set $\epsilon = (2(A_{B,\xi}+ \mathcal{E}_B))/(M_Nr_{N/B}^{\rho})$ in the above inequality. The earlier bound on $\epsilon$ now changes to a bound on $A_{B,\xi} + \mathcal{E}_B$. We thus have the following concentration inequality for the order statistic $T_{( k_{B,\alpha}+1)}$,
\begin{equation*}
    \begin{split}
   \mathbb{P}\left(|T_{( k_{B,\alpha}+1)}| \leq \frac{2}{M_N r_{N/B}^{\rho}}\left(A_{B,\xi} + \mathcal{E}_B \right) \right) \geq (1 - 2 \exp(-\xi))  \boldsymbol 1 \left\{A_{B,\xi} + \mathcal{E}_B < \zeta_{\rho,N}  \right\}     
    \end{split}
\end{equation*}

where,
\[
\zeta_{\rho,N} = \frac{M_N}{2} \min\left\{\eta^{\rho}, \left(\frac{M_N}{2C_N} \right)^{1/\delta} \right\}.
\]
Similarly we have a concentration inequality for the order statistic $T_{(B -  k_{B,\alpha})}$, 
\begin{equation*}
  \mathbb{P}\left(|T_{(B -  k_{B,\alpha})}| \leq \frac{2}{M_N r_{N/B}^{\rho}}\left(A_{B,\xi} + \mathcal{E}_B \right) \right) \geq (1 - 2 \exp(-\xi))  \boldsymbol 1 \left\{A_{B,\xi} + \mathcal{E}_B < \zeta_{\rho,N}  \right\}       
\end{equation*}
The above statements also imply that with probability greater than or equal to \\ $(1 - 2 \exp(-\xi))  \boldsymbol 1 \left\{A_{B,\xi} + \mathcal{E}_B < \zeta_{\rho,N}  \right\}$ both $|T_{( k_{B,\alpha}+1)}|$ and $|T_{(B -  k_{B,\alpha})}|$ are less than $(2(A_{B,\xi} + \mathcal{E}_B))/(M_N r_{N/B}^{\rho})$ which is less than $(2\zeta_{\rho,N})/(M_N r_{N/B}^{\rho}) < \eta^{\rho}/r_{N/B}^{\rho} $. Using \eqref{eq:H_result_ghulc} we can say that with probability greater than or equal to $(1 - 2 \exp(-\xi))  \boldsymbol 1 \left\{A_{B,\xi} + \mathcal{E}_B < \zeta_{\rho,N}  \right\}$ each the following two events hold,
\begin{equation}
\label{ghulc_trinq_1}
    \begin{split}
        & \left|H_{N/B}(T_{( k_{B,\alpha}+1)}) - H_{N/B}(0) - M_Nr_{N/B}^{\rho} T_{( k_{B,\alpha}+1)} \right| \leq C_Nr_{N/B}^{\rho(1+\delta)}T_{( k_{B,\alpha}+1)}^{1+\delta} \leq C_N \left(\frac{2(A_{B,\xi}+\mathcal{E}_B)}{M_N} \right)^{1+\delta}, \\
        & \left|H_{N/B}(T_{(B -  k_{B,\alpha})}) - H_{N/B}(0) - M_Nr_{N/B}^{\rho} T_{(B -  k_{B,\alpha})} \right| \leq C_Nr_{N/B}^{\rho(1+\delta)}T_{( k_{B,\alpha}+1)}^{1+\delta} \leq C_N \left(\frac{2(A_{B,\xi}+\mathcal{E}_B)}{M_N} \right)^{1+\delta}. 
    \end{split}
\end{equation}
We shall now prove a result regarding uniform order statistics ($U_{k:n}$) based on the derivations in \cite{reiss2012approximate}. Let $V$ and $V_n$ denote the CDF of Uniform(0,1) and empirical CDF of $U_{1:n},\cdots,U_{n:n}$ respectively. We know from Lemma 3.1.1 of \cite{reiss2012approximate} that, 
\begin{equation}
\label{eq:reiss_3}
\begin{split}
    & \mathbb{P}\left(\frac{\sqrt{n}}{\sqrt{\frac{r}{n+1}\left( 1- \frac{r}{n+1} \right)}}\left| U_{r:n} - \frac{r}{n+1} \right| \leq 5\xi \right) \geq 1 - 2\exp(-\xi) \\
 \implies &  \mathbb{P}\left(\left|V_n^{-1}(q) - q \right| \leq \sqrt{\frac{r(q)}{n+1}\left( 1- \frac{r(q)}{n+1} \right)} \frac{5\xi}{\sqrt{n}} \mbox{ for some } q \in (0,1) \right) \geq 1 - 2\exp(-\xi), 
\end{split}
\end{equation}
where $r(q) = nq$ if $nq$ is integer and $r(q) = \lfloor nq \rfloor + 1$ otherwise. We also know from the proof of Lemma 6.3.2 of \cite{reiss2012approximate} that for given $\epsilon,\rho>0$ the following holds,
\begin{equation*}
    \mathbb{P}\left\{\sup_{I \in \mathcal{I}} \frac{\sqrt{n}|Q_n(I) - Q_0(I)|}{\max\{ \sigma(I),\rho/\sqrt{n}\}} \geq \epsilon \right\} \leq (n+2)^2 \exp\left[-\epsilon\rho + \frac{3}{4}\rho^2 + \frac{13}{2} \right] .
\end{equation*}
We set $\rho = \sqrt{\log(n+2)}$ and $\epsilon = (11/4)\sqrt{\log(n+2)} + (\xi+13/2)/\sqrt{log(n+2)}$ in the above statement to obtain the following, 
\begin{equation}
\label{eq:reiss_6}
    \mathbb{P}\left\{\sup_{I \in \mathcal{I}} \frac{\sqrt{n}|Q_n(I) - Q_0(I)|}{\max\{ \sigma(I),\sqrt{\log(n+2)}/\sqrt{n}\}} \leq \frac{11}{4}\sqrt{\log(n+2)} + \frac{\xi + (13/2)}{\sqrt{\log(n+2)}} \right\} \geq 1 - \exp(-\xi). 
\end{equation}
We use \eqref{eq:reiss_3} and \eqref{eq:reiss_6} and follow the same steps as in the proof of the Theorem 6.3.1 of \cite{reiss2012approximate}. We obtain that the following event holds with probability greater than or equal to $1 - 3\exp(-\xi)$ for some $q \in (0,1)$, 
\begin{equation*}
\begin{split}
    & |V_n^{-1}(q) - q + V_n(q) - q| \\
    \leq & \left\{ \frac{11}{4}\sqrt{\frac{\log(n+2)}{n}} + \frac{\xi + (13/2)}{\sqrt{n\log(n+2)}} \right\}\max\left\{\sqrt{\sqrt{\frac{r(q)}{n+1}\left( 1- \frac{r(q)}{n+1} \right)} \frac{5\xi}{\sqrt{n}}} , \sqrt{\frac{\log(n+2)}{n}}\right\} .
    \end{split}
\end{equation*}
In particular for $q = k/n$ we can say that the following event holds with probability greater than or equal to $1 - 3 \exp(-\xi)$,
\begin{equation*}
\begin{split}
   & \left|U_{k:n} - \frac{1}{2} - \left( \frac{k}{n} - \frac{1}{2} \right) + \frac{1}{n}\sum_{i = 1}^n \boldsymbol 1 \left\{U_{i:n} \leq \frac{k}{n}\right\} - \frac{k}{n}\right| \\
    \leq &\left\{ \frac{11}{4}\sqrt{\frac{\log(n+2)}{n}} + \frac{\xi + (13/2)}{\sqrt{n\log(n+2)}} \right\}\max\left\{\sqrt{\sqrt{\frac{k(n+1 - k)}{n}} \frac{5\xi}{n+1}} , \sqrt{\frac{\log(n+2)}{n}}\right\} \\
    \leq & \left\{ \frac{11}{4}\sqrt{\frac{\log(n+2)}{n}} + \frac{\xi + (13/2)}{\sqrt{n\log(n+2)}} \right\}\max\left\{\sqrt{ \frac{5\xi}{2\sqrt{n}}} , \sqrt{\frac{\log(n+2)}{n}}\right\} = D_{\xi,n} \mbox{ (say) }. 
\end{split}
\end{equation*}
Using the above result we obtain that the following two events hold, each with probability greater than or equal to $1 - 3 \exp(-\xi)$, 
\begin{equation}
\label{eq:ghulc_tring_2}
    \begin{split}
        &\left|H_{N/B}(T_{( k_{B,\alpha}+1)}) - H_{N/B}(0) - Q_{1,B} \right| \leq D_{\xi,B} + \mathcal{E}_B, \\
       &  \left|H_{N/B}(T_{(B -  k_{B,\alpha})}) - H_{N/B}(0) - Q_{2,B} \right| \leq D_{\xi,B} + \mathcal{E}_B,
    \end{split}
\end{equation}
where,
\begin{equation*}
    \begin{split}
        Q_{1,B} &= \left(\frac{ k_{B,\alpha}+1}{B} - \frac{1}{2} \right) - \frac{1}{B}\sum_{i = 1}^{B} \boldsymbol{1} \left\{ F_{N/B}(\widehat \theta_j) \leq \frac{ k_{B,\alpha}+1}{B}\right\} + \frac{ k_{B,\alpha}+1}{B}, \\
        Q_{2,B} &= \left(\frac{B -  k_{B,\alpha}}{B} - \frac{1}{2} \right) - \frac{1}{B}\sum_{i = 1}^{B} \boldsymbol{1} \left\{ F_{N/B}(\widehat \theta_j) \leq \frac{B -  k_{B,\alpha}}{B}\right\} + \frac{B -  k_{B,\alpha}}{B}.
    \end{split}
\end{equation*}
Note that $Q_{1,B}$ is same as $Q_B$ defined in \Cref{thm:width_ghulc}.  We note that in the above deduction we used the fact that $|H_{N/B}(0) - 1/2| = | F_{N/B}(\theta_0) - 1/2| = \mathcal{E}_B$. Combining \eqref{ghulc_trinq_1} and \eqref{eq:ghulc_tring_2} using triangle inequality we obtain that the following two events hold, each with probability greater than or equal to $(1 - 5\exp(-\xi))\boldsymbol{1}\{A_{B,\xi} + \mathcal{E}_B < \zeta_{\rho,N} \}$, 
\begin{equation*}
    \begin{split}
      & \left|M_Nr_{N/B}^{\rho}T_{( k_{B,\alpha}+1)} -  Q_{1,B}  \right| \leq \gamma_{\xi,B,N}, \\
      & \left|M_Nr_{N/B}^{\rho}T_{(B -  k_{B,\alpha})} -  Q_{2,B}  \right| \leq \gamma_{\xi,B,N},
    \end{split}
\end{equation*}
where, 
\[
\gamma_{\xi,B,N} = D_{\xi,B} + \mathcal{E}_B + C_N \left(\frac{2(A_{B,\xi}+\mathcal{E}_B)}{M_N} \right)^{1+\delta}. 
\]
Combining the above two equations and using the monotonicity of the transformation of $T_j$ from $\widehat \theta_j$ we conclude that the following event holds with probability greater than or equal to $(1 - 10\exp(-\xi))\boldsymbol{1}\{A_{B,\xi} + \mathcal{E}_B < \zeta_{\rho,N} \}$, 
\begin{equation*}
    \begin{split}
       & \frac{1}{M_N^{1/\rho}r_{N/B}} \{|Q_{1,B} - \gamma_{\xi,B,N}|^{1/\rho} \mathrm{sgn}(Q_{1,B} - \gamma_{\xi,B,N}) - |Q_{2,B} + \gamma_{\xi,B,N}|^{1/\rho} \mathrm{sgn}(Q_{2,B} +\gamma_{\xi,B,N})  \} \\
      \leq & \widehat \theta_{ k_{B,\alpha} +1} - \widehat \theta_{B -  k_{B,\alpha}} = \mathrm{Width}\left(\widehat{\mathrm{CI}}_{N,\alpha}^{\mathtt{GHulC}} \right) \\
      \leq & \frac{1}{M_N^{1/\rho}r_{N/B}} \{|Q_{1,B} + \gamma_{\xi,B,N}|^{1/\rho} \mathrm{sgn}(Q_{1,B} + \gamma_{\xi,B,N}) - |Q_{2,B} - \gamma_{\xi,B,N}|^{1/\rho} \mathrm{sgn}(Q_{2,B} -\gamma_{\xi,B,N})  \}.
    \end{split}
\end{equation*}
For the remaining part of the proof we follow exactly the same steps as in the proof of \Cref{thm:asym_result_irreg_case}. We update the Chernoff bounds used in that proof. It can shown that each of the three events below hold with probability greater than or equal to $(1 - 2\exp(-\xi))$, 
\begin{equation*}
    \begin{split}
       & |\sqrt{B}(Q_{1,B}-Q_{2,B})-z_{\alpha/2}| \leq \frac{z_{\alpha/2}^3}{5B} + \frac{4.2}{\sqrt{B}}+ \sqrt{\frac{3(3+\sqrt{B}z_{\alpha/2})\xi}{B}}, \\
       & |\sqrt{B}Q_{1,B}-(z_{\alpha/2}/2)| > \frac{z_{\alpha/2}^3}{10B} + \frac{2}{\sqrt{B}} + \sqrt{\frac{3(B+\sqrt{B}z_{\alpha/2}+4)\xi}{2B}}, \\
       & |\sqrt{B}Q_{2,B}+(z_{\alpha/2}/2)| > \frac{z_{\alpha/2}^3}{10B} + \frac{2}{\sqrt{B}} + \sqrt{\frac{3(B-\sqrt{B}z_{\alpha/2}+(z_{\alpha/2}^3/(5\sqrt{B}))+3.2)\xi}{2B}} .
    \end{split}
\end{equation*}
Analogous to the proof of \Cref{thm:asym_result_irreg_case} we define the following quantities using the above Chernoff bounds, 
\begin{equation*}
    \begin{split}
        E_{1,B,\xi} &=\frac{z_{\alpha/2}^3}{10B} + \frac{2}{\sqrt{B}} + \sqrt{\frac{3(B+\sqrt{B}z_{\alpha/2}+4)\xi}{2B}}, \\
        E_{2,B,\xi} &= \frac{z_{\alpha/2}^3}{10B} + \frac{2}{\sqrt{B}} + \sqrt{\frac{3(B-\sqrt{B}z_{\alpha/2}+(z_{\alpha/2}^3/(5\sqrt{B}))+3.2)\xi}{2B}}, \\
        E_{3,B,\xi} &=  \sqrt{B}\gamma_{\xi,B,N},\\
        E_{1,\xi} & = \sup_{l \in \mathbb{N}} E_{1,B,\xi},\\
        E_{2,\xi} &=  \sup_{l \in \mathbb{N}} E_{2,B,\xi},\\
        E_{3,\xi} &= \sup_{l \in \mathbb{N}} E_{3,B,\xi} .
    \end{split}
\end{equation*}
We also define, 
\begin{equation}
\label{eq:crhoxi}
C_{\rho, \xi, \alpha} = (2/\rho)E_{0,\xi}^{(1/\rho)-1} \mbox{  where } E_{0,\xi} = (z_{\alpha/2}/2) + \max\{E_{1,\xi}, E_{2,\xi} \}+ E_{3,\xi},    
\end{equation}
for $0< \rho < 1$. On going through all the steps in the proof of \Cref{thm:asym_result_irreg_case} as they are, we obtain that for $\rho \geq 1$, the following event holds with probability greater than or equal to $(1 - 12\exp(-\xi))\boldsymbol{1}\{A_{B,\xi} + \mathcal{E}_B < \zeta_{\rho,N} \}$, 
\begin{equation*}
    \begin{split}
  &\left|(B)^{1/(2\rho)}\mathrm{Width}(\widehat{\mathrm{CI}}_{N,\alpha}^{\mathtt{GHulC}}) -  \frac{(B)^{1/2\rho}}{M_N^{1/\rho}r_{N/B}}\left\{|Q_{1,B}|^{1/\rho} \mathrm{sgn}(Q_{1,B})- \left|Q_{1,B}-\frac{z_{\alpha/2}}{\sqrt{B}}\right|^{1/\rho} \mathrm{sgn}\left(Q_{1,B}-\frac{z_{\alpha/2}}{\sqrt{B}}\right) \right\} \right|    \\
  \leq & \frac{4}{M_N^{1/\rho}r_{N/B}}(\sqrt{B}\gamma_{\xi,B,N})^{1/\rho} + \frac{2}{M_N^{1/\rho}r_{N/B}}\left( \frac{z_{\alpha/2}^3}{5B} + \frac{4.2}{\sqrt{B}}+ \sqrt{\frac{3(3+\sqrt{B}z_{\alpha/2})\xi}{B}}\right)^{1/\rho},
    \end{split}
\end{equation*}
On the other hand for $0< \rho <1$ we get that the following event holds with probability greater than or equal to $(1 - 16\exp(-\xi))\boldsymbol{1}\{A_{B,\xi} + \mathcal{E}_B < \zeta_{\rho,N} \}$,
\begin{equation*}
    \begin{split}
  &\left|(B)^{1/(2\rho)}\mathrm{Width}(\widehat{\mathrm{CI}}_{N,\alpha}^{\mathtt{GHulC}}) -  \frac{(B)^{1/2\rho}}{M_N^{1/\rho}r_{N/B}}\left\{|Q_{1,B}|^{1/\rho} \mathrm{sgn}(Q_{1,B})- \left|Q_{1,B}-\frac{z_{\alpha/2}}{\sqrt{B}}\right|^{1/\rho} \mathrm{sgn}\left(Q_{1,B}-\frac{z_{\alpha/2}}{\sqrt{B}}\right) \right\} \right|    \\
  \leq & \frac{2C_{\rho, \xi, \alpha}}{M_N^{1/\rho}r_{N/B}}\sqrt{B}\gamma_{\xi,B,N}^{1/\rho} + \frac{C_{\rho, \xi, \alpha}}{M_N^{1/\rho}r_{N/B}}\left( \frac{z_{\alpha/2}^3}{5B} + \frac{4.2}{\sqrt{B}}+ \sqrt{\frac{3(3+\sqrt{B}z_{\alpha/2})\xi}{B}}\right).
    \end{split}
\end{equation*}
As before we have the following distributional convergence as $l \rightarrow \infty$,
\begin{equation*}
    \begin{split}
       &\frac{(B)^{1/2\rho}}{M_N^{1/\rho}}\left\{|Q_{1,B}|^{1/\rho} \mathrm{sgn}(Q_{1,B})- \left|Q_{1,B}-\frac{z_{\alpha/2}}{\sqrt{B}}\right|^{1/\rho} \mathrm{sgn}\left(Q_{1,B}-\frac{z_{\alpha/2}}{\sqrt{B}}\right) \right\}   \\
    \stackrel{d}{\xrightarrow{}} &\frac{1}{M_N^{1/\rho}}\{|W_1|^{1/\rho} \mathrm{sgn}(W_1)-|W_1 - z_{\alpha/2}|^{1/\rho} \mathrm{sgn}(W_1 - z_{\alpha/2}) \},   
    \end{split}
\end{equation*}
where $W_1 \sim N(z_{\alpha/2}/2,1/4)$. This completes the proof of the theorem.  
\end{proof}

\section{Further simulations}
\label{appendix:further_simulations}
\subsection{Performance of distribution-free CI for median under standard assumptions}
\label{subsec:bahadur}

\begin{figure}[!htb]
    \centering
    \includegraphics[width=\textwidth,keepaspectratio]{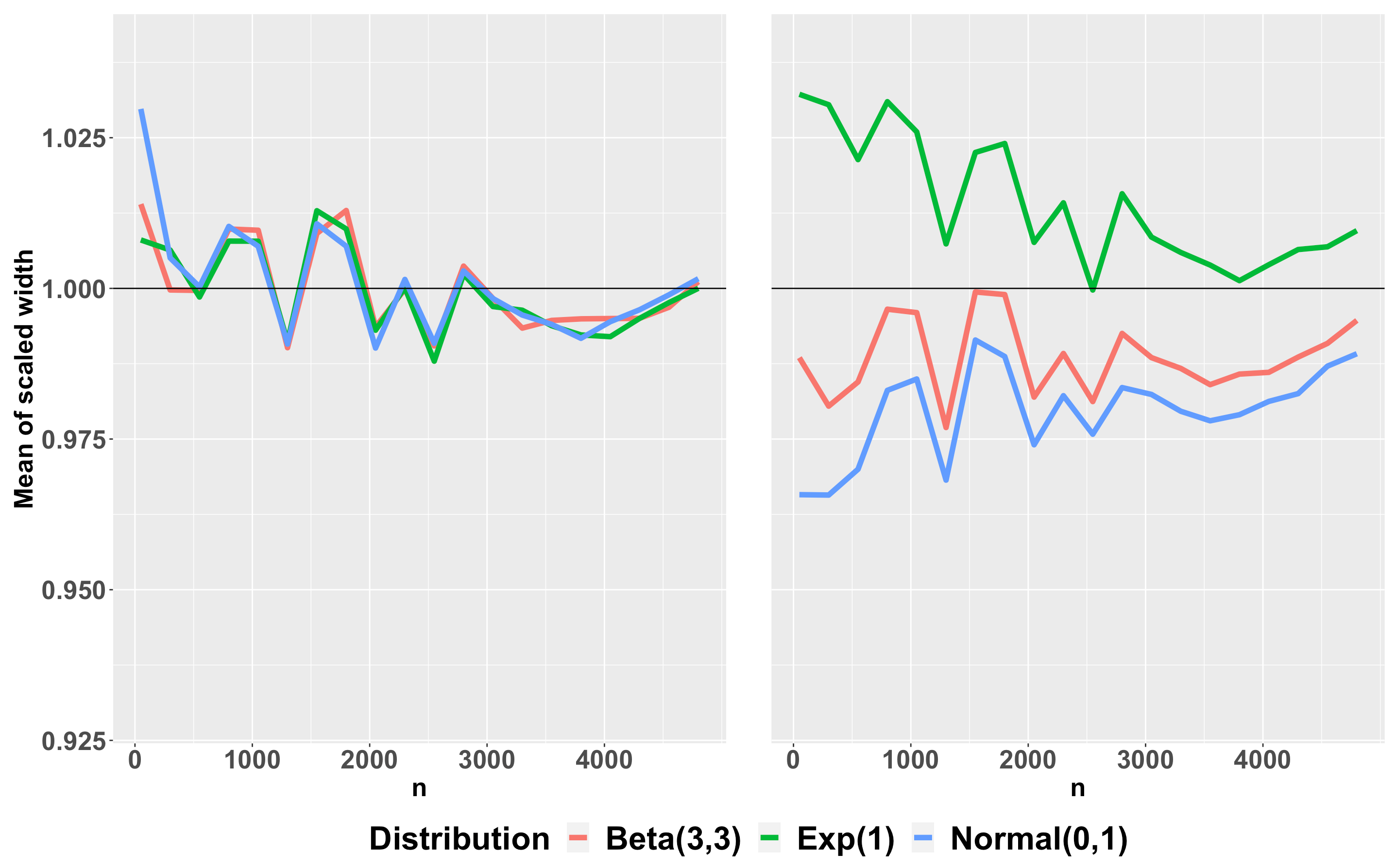}
    \caption{In the first plot (from left) we see the comparison of the mean of $\mathrm{WR}_{n,\alpha}$ with maximum likelihood based variance estimator (set-up \texttt{B-I}). In the second plot we see the comparison of the mean of $\mathrm{WR}_{n,\alpha}$ with kernel density based variance estimator (set-up \texttt{B-II}).}
    \label{fig:mean_scaled_width}
\end{figure}

\begin{figure}[!ht]
    \centering
    \includegraphics[width=\textwidth,keepaspectratio]{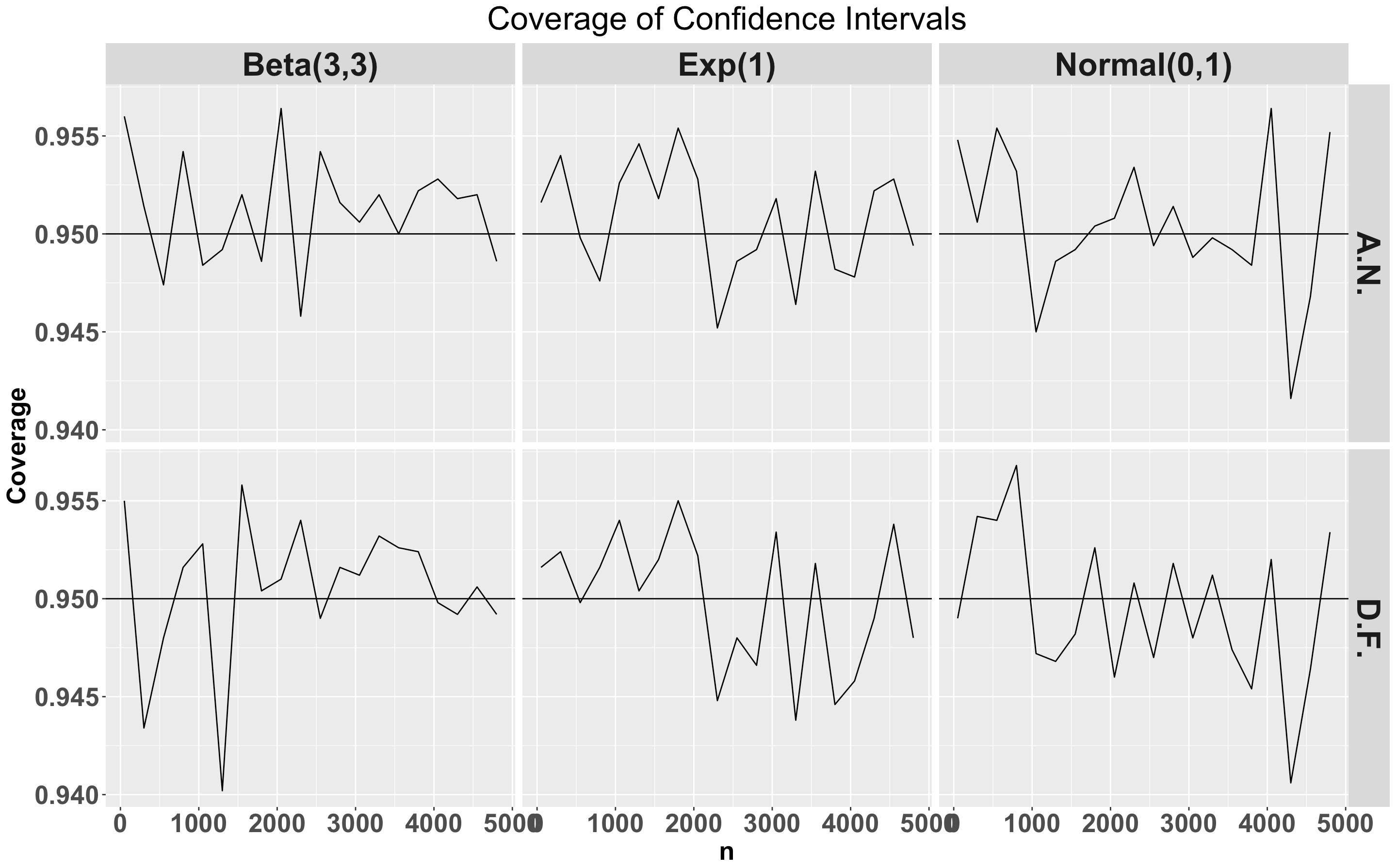}
    \caption{\texttt{B-I}: Comparison of the coverage obtained by using $\widehat{\mathrm{CI}}_{n,\alpha}$ \eqref{eq:rewritten-CI} and the traditional Wald confidence interval for various sample sizes, $n$. Here \textbf{A.N.} denotes the asymptotic normality-based confidence interval and \textbf{D.F.} denotes the distribution-free confidence interval.}
    \label{fig:coverage}
\end{figure}

\begin{figure}[!ht]
    \centering
    \includegraphics[width=\textwidth,keepaspectratio]{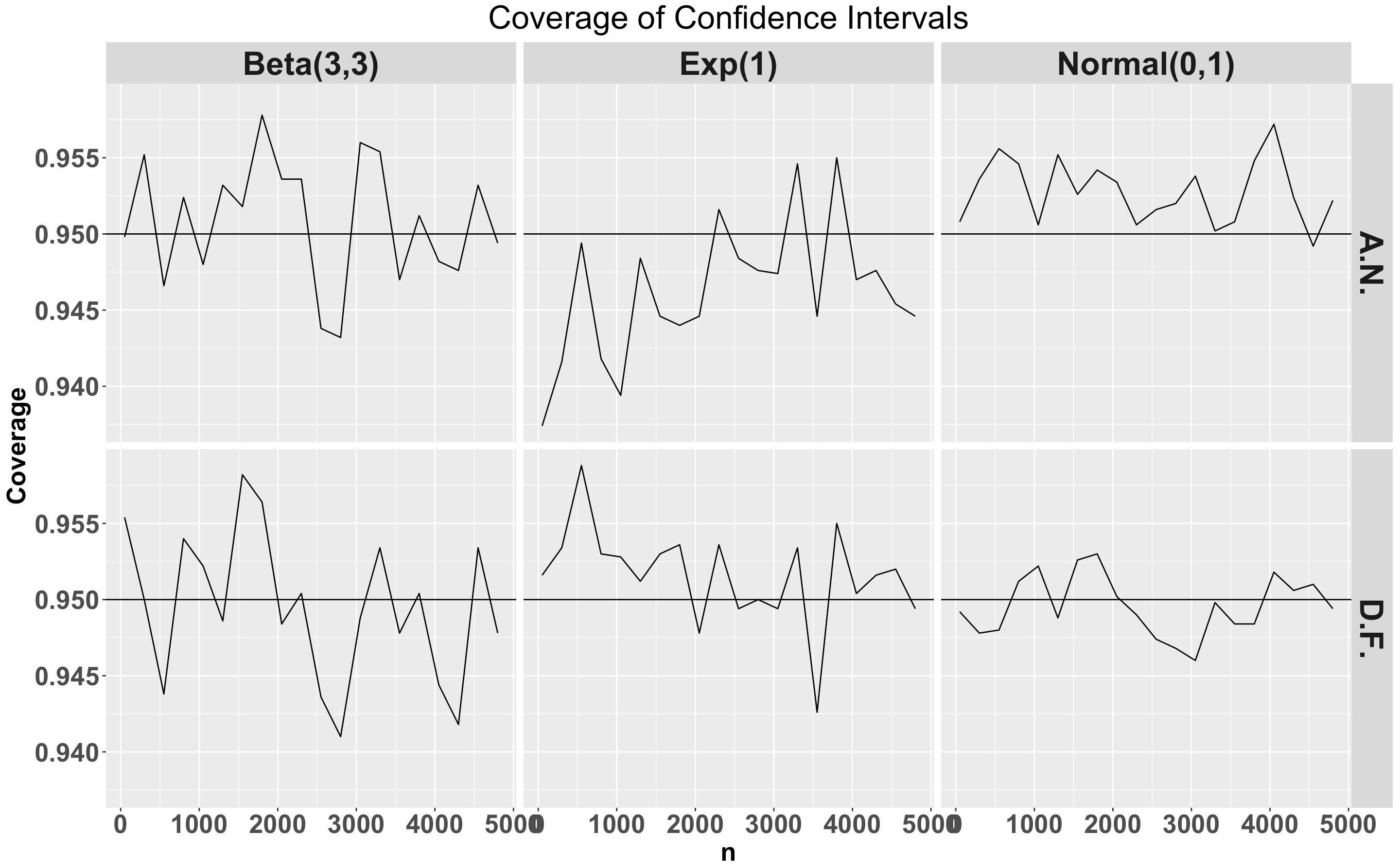}
    \caption{\texttt{B-II}: Comparison of the coverage obtained by using $\widehat{\mathrm{CI}}_{n,\alpha}$ \eqref{eq:rewritten-CI} and the traditional Wald confidence interval for various sample sizes, $n$. Here \textbf{A.N.} denotes the asymptotic normality-based confidence interval and \textbf{D.F.} denotes the distribution-free confidence interval.}
    \label{fig:coverage_unknown}
\end{figure}
In this sub-section, we shall test the performance of $\widehat{\mathrm{CI}}_{n,\alpha}$ \eqref{eq:rewritten-CI} for distributions for which the assumptions of Bahadur representation hold true (in other words the assumptions of \Cref{thm:bahadur_asump} hold true) i.e.\ the underlying distribution $F$ is differentiable at the population median $\theta_0$ with $F'(\theta_0)>0$ and $F(\theta_0) = 1/2$. We shall compare the performance of $\widehat{\mathrm{CI}}_{n,\alpha}$ \eqref{eq:rewritten-CI} against the Wald confidence interval. See \Cref{alg:proposed-conf-int} and \Cref{alg:asym_norm-conf-int} for details regarding construction of these confidence intervals. We describe the simulation procedure below for $N(0,1)$ distribution. We shall repeat the same simulation steps for $\text{Exp}(1)$ and $\text{Beta}(3,3)$ distribution. 

\begin{algorithm}
    \caption{Wald confidence interval for the population median}
    \label{alg:asym_norm-conf-int}
    \KwIn{Sample: $X_1,\cdots,X_n$ and Confidence Level: $1-\alpha$}
    \KwOut{Wald confidence interval for the population median}
   Compute the sample median $\widehat{\theta}_n$ as follows,
   \[
\widehat{\theta}_n = \begin{cases}
    X_{((n+1)/2)} \text{  if $n$ is odd.} \\
    X_{(n/2)} \text{ if $n$ is even.} 
\end{cases}
\] \\
  Compute an estimator $\widehat{\sigma}$ of the standard deviation of the asymptotic standard deviation of the sample median $\widehat{\theta}_n$,
  \[
  \widehat{\sigma} = \begin{cases}
  \frac{1}{2\widehat{f}(\widehat{\theta}_n)} \text{  if the family of densities is unknown.} \\
  \frac{1}{2f_{\widehat{\eta}}(\widehat{\theta}_n)} \text{  if the family of densities is known.}
  \end{cases}
  \]
Here in the first case (when the family of densities is unknown) $\widehat{f}(.)$ is the kernel density estimate of the underlying density using a gaussian kernel. To compute the bandwidth of the kernel density estimator, the \texttt{bw.nrd0} function of the \texttt{stats} package (\cite{rsoftware}) is used which uses a rule-of-thumb for choosing the bandwidth of a Gaussian kernel density estimator (see Section-$3.1$ of \cite{sheather2004density}). In the second case, if the family of densities is known to be $\{f_\eta\}$ (where $\eta$ is the parameter), an estimate of $\eta$ namely $\widehat{\eta}$ is obtained from the sample and used for computing $\widehat{\sigma}$. \\
Compute the set,
\[
\widehat{\mathrm{CI}}_{n,\alpha}^{\mathtt{AN}} := \left[\widehat{\theta}_n - \frac{\widehat{\sigma}z_{\alpha/2}}{\sqrt{n}},\, \widehat{\theta}_n + \frac{\widehat{\sigma}z_{\alpha/2}}{\sqrt{n}}\right].
\] \\
Return the confidence interval $\widehat{\mathrm{CI}}_{n,\alpha}^{\mathtt{AN}}$.

\end{algorithm}

\begin{enumerate}
    \item For each value of $n$ ($n$ may vary from $50$ to $5000$), we generate $n$ many observations from the $N(0,1)$ distribution $5000$ times.
    \item For each time we compute both $\widehat{\mathrm{CI}}_{n,\alpha}$ \eqref{eq:rewritten-CI} and the Wald confidence interval (see \Cref{alg:asym_norm-conf-int}) for the given $n$ observations. 
    \item From this data, we calculate the coverage (the proportion of times the computed confidence interval contains the true parameter i.e.\ median) for both the types of confidence intervals. 
    \item We also estimate the mean of the ratio of width of distribution-free confidence interval to the width of the Wald confidence interval) from this data using the idea of delta-method. 
\end{enumerate}

It should be noted that there are two ways of evaluating $\widehat{\sigma}$ as mentioned in \Cref{alg:asym_norm-conf-int}. We will perform simulation under both the scenarios for all three distributions. We shall refer the simulation setup as \texttt{B-I} for the case when assumptions of \Cref{thm:bahadur_asump} hold and the family of densities is known and we shall refer the simulation setup as \texttt{B-II} for the case when assumptions of \Cref{thm:bahadur_asump} hold and the family of densities is unknown. The inferences made from the simulations are noted below.
\begin{itemize}
    \item We observe the plot of the mean of $\mathrm{WR}_{n,\alpha}$ with $n$ in \Cref{fig:mean_scaled_width}. It can be observed that although there are initial fluctuations, the ratio of widths approaches $1$ as the value of $n$ increases. This supports the conclusions of \Cref{thm:bahadur_asump}.
    \item We also observe that the behaviour of mean of $\mathrm{WR}_{n,\alpha}$ is almost same for all the three distributions in \Cref{fig:mean_scaled_width}.

    \item  We further analyse the coverage of both the confidence intervals under consideration for all the three distributions in \Cref{fig:coverage,fig:coverage_unknown}. We note that not only does the coverage lies very close to the expected value of $0.95$ in all the six scenarios, both the range and the fluctuations of the coverage match for the distribution-free confidence interval and the Wald confidence interval under both the scenarios. 
\end{itemize}

\subsection{Performance of distribution-free CI for median under non-standard assumptions}
\label{sec:simulation}
In this section, we shall compare the observed distribution of $\mathrm{WR}_{n,\alpha}$ with its limiting distribution. We shall also demonstrate the performance of the distribution-free confidence interval. The performance shall be primarily assessed on the basis of coverage and the width of the confidence interval.
\subsubsection{Comparing observed and limiting distribution of $\mathrm{WR}_{n,\alpha}$}
In order to compare the observed distribution of $\mathrm{WR}_{n,\alpha}$ with the limiting distribution $\mathscr{G}((Z/z_{\alpha/2}) + 1, 2)/2$ ($Z \sim N(0, 1 )$) we simulate observations from the distribution $F_{\rho}$, 
\begin{equation}
 \label{eq:frho}
 F_{\rho}(x) = 0.5|x|^\rho \mathrm{sgn}(x) +0.5 \mbox{ for } x \in \left[ -1,1 \right] \mbox{ and } \rho >0.
\end{equation}
Note that median of $F_{\rho}$ is $\theta_0 = 0$ and we have the following limit, 
    \[
    \lim_{h\xrightarrow{} 0}\frac{|F_{\rho}(\theta_0+h)-F_{\rho}(\theta_0)|}{|h|^\rho} = 0.5 = M.
    \]
For $\rho = 0.75, 2, 10$ we generate $n = 20000$ observations from $F_{\rho}$ and compare the resulting histogram of the observed width ratio $\mathrm{WR}_{n,\alpha}$ (based on $1000$ Monte-Carlo draws) with the corresponding limiting distribution. The resulting plots can be seen in \Cref{fig:observedvslimit}. We observe that the observed distribution of $\mathrm{WR}_{n,\alpha}$ converges to the limiting distribution $\mathscr{G}((Z/z_{\alpha/2}) + 1, 2)/2$ and the convergence to the limiting distribution is quite slow. 
\begin{figure}[!htb]
    \centering
    \includegraphics[width=\textwidth,keepaspectratio]{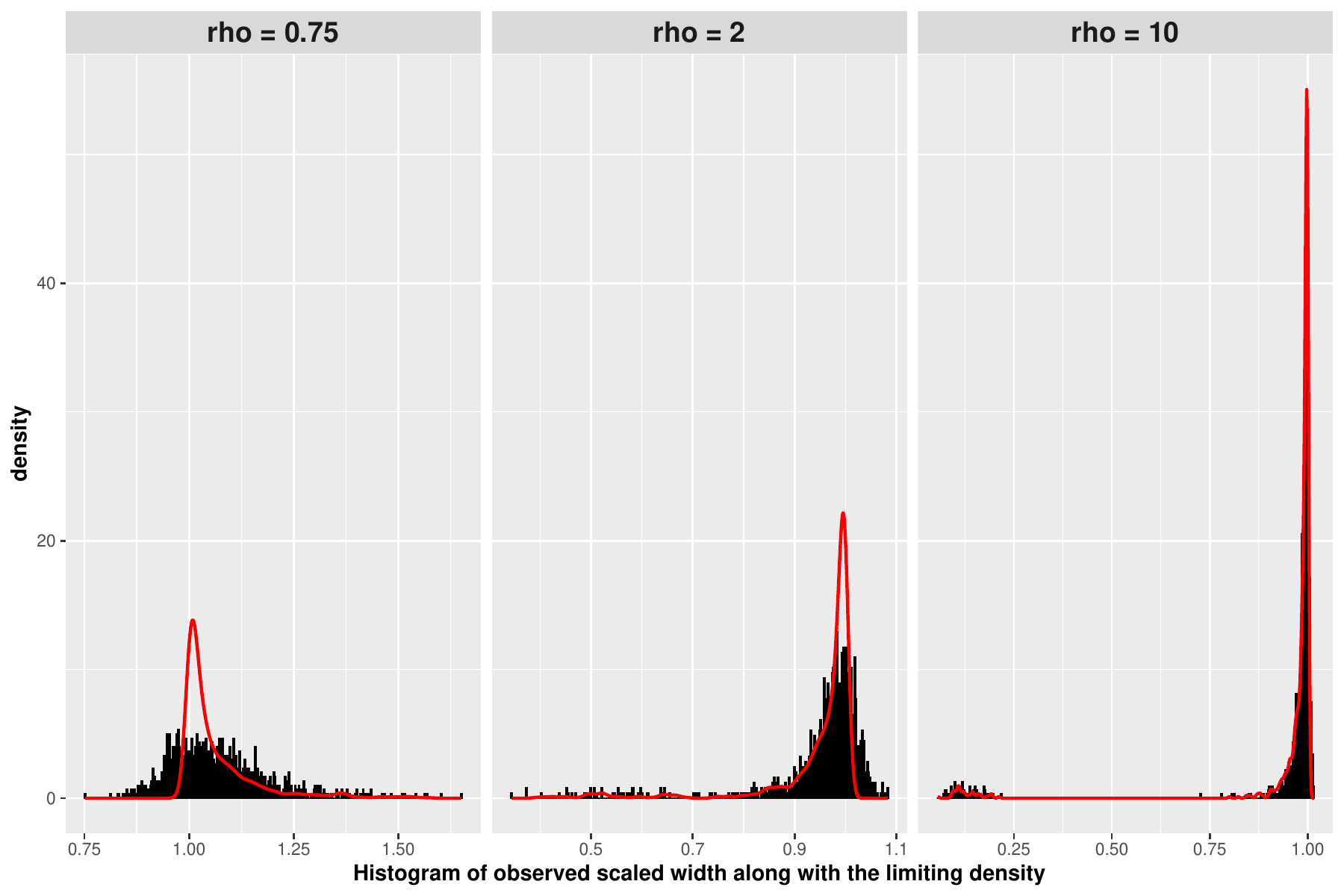}
    \caption{Histogram of the width ratio $\mathrm{WR}_{n,\alpha}$ (for sample-size $n = 20000$ with $1000$ Monte Carlo replications) along with the density of the limiting distribution of $\mathrm{WR}_{n,\alpha}$ for different values $\rho = 0.75, 2, 10$. }
    \label{fig:observedvslimit}
\end{figure}

\subsubsection{Performance of distribution-free C.I.\ against benchmark methods.}
We describe the process of generating observations from different distributions and comparing the distribution-free confidence interval $\widehat{\mathrm{CI}}_{n,\alpha}$ with various benchmark methods such as classical bootstrap and sub-sampling with estimated rate of convergence. 
\begin{enumerate}
    \item We shall use the distribution function $F_{\rho}$ (see \eqref{eq:frho}). We consider $n$ to be $50,200,500,1000$. We vary $\rho$ from $0.2$ to $10$. 
     \item For each value of $\rho$ we generate $n$ many observations from the distribution $F_{\rho}$ $1000$ times. 
    \item For each iteration we compute the distribution-free C.I.\, the sub-sample based C.I.\ (with estimated rate of convergence and with sub-sample size $n^{1/2}$), and the classical bootstrap based C.I.\ for the given $n$ observations at the level of confidence $100(1-\alpha)\%$ (we use $\alpha=0.05$). 
    \item From this data, we calculate the coverage for all the three types of C.I.'s. 
    \item We also provide box-plots for $n^{1/(2\rho)}\mathrm{Width}$ for all the three types of C.I.'s obtained from this data.
\end{enumerate}
\begin{figure}[!htb]
    \centering
    \includegraphics[width=\textwidth,keepaspectratio]{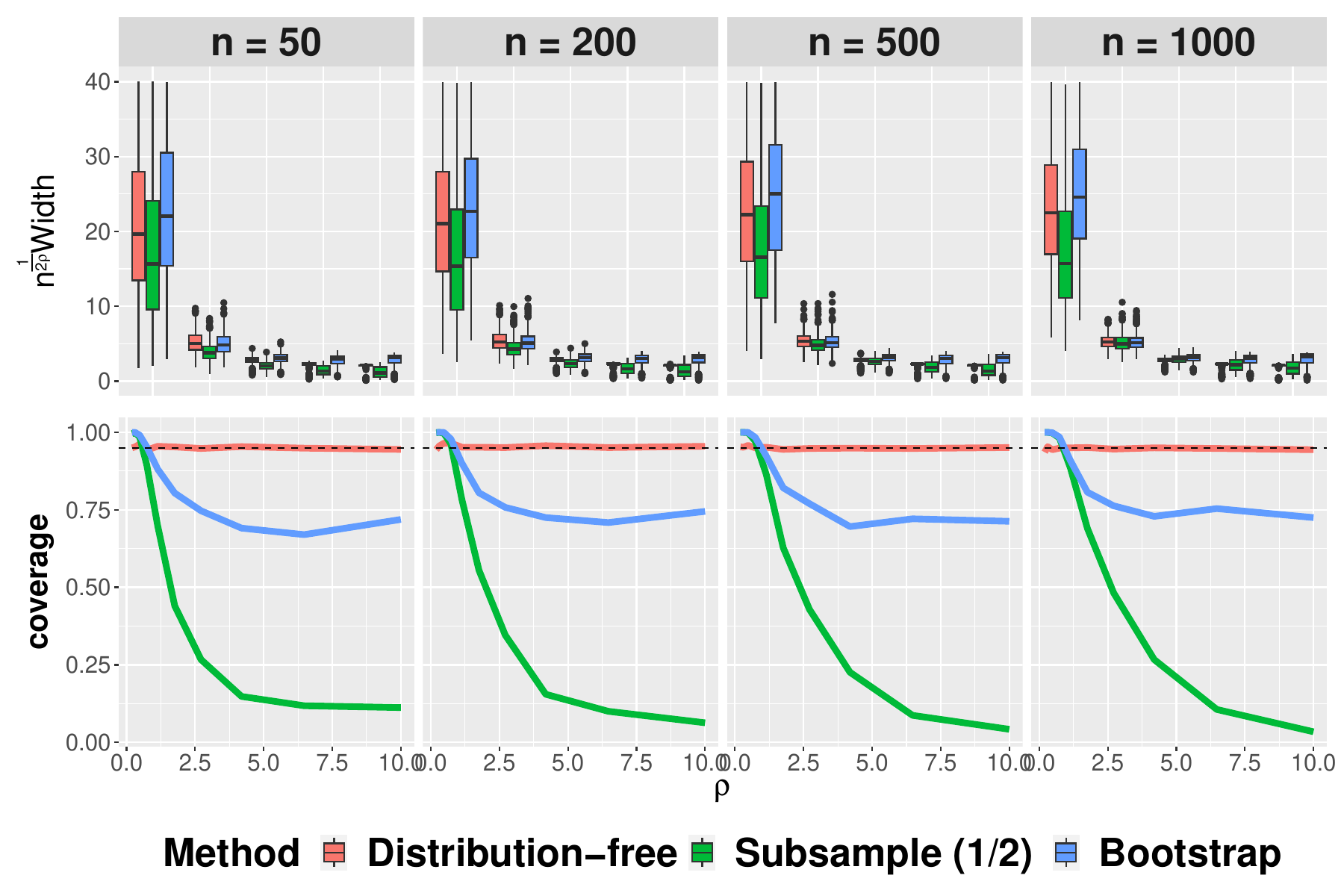}
    \caption{Comparison of the width and coverage of different types of confidence intervals (the distribution-free C.I.\, the sub-sample based C.I.\ (with estimated rate of convergence and with sub-sample size $n^{1/2}$), and the classical bootstrap based C.I.\ ) for different values of sample sizes ($n = 50, 200, 500, 1000$) and for different growth-rates of the distribution function on either side of the median ($\rho$ takes values from $0.2$ to $10$). The box-plots have been thresholded at $y = 40$ for better visibility, so there might be some outlying observations beyond the threshold.}
    \label{fig:subsample_compare}
\end{figure}
The result of the simulation can be seen in \Cref{fig:subsample_compare}. We see that only the distribution-free confidence interval maintains the required coverage of $95\%$ for all sample sizes and for all values of $\rho$. The classical bootstrap and the sub-sample based confidence intervals on the other hand do the maintain the required coverage for $\rho \neq 1$. It can also be seen from the inter-quartile ranges of the confidence intervals that the variance of the the width of the distribution-free confidence interval is in general lower than that of the bootstrap and sub-sampling based confidence intervals for higher values of $\rho$. More simulations illustrating the performance of the distribution-free confidence interval against the benchmark methods can be found in \Cref{appendix:further_simulations}.

\subsection{Application of GHulC to univariate quantile regression}
To understand the validity and power of the confidence intervals generated by GHulC we consider the following numerical example of univariate quantile regression. Suppose $(X_i,Y_i) \in \mathbb{R}^2$, $1 \leq i \leq n$ are independent and identically distributed random vectors. We define the estimator $\widehat \theta_n$ as follows,
\[
\widehat \theta_n = \mbox{arg min}_{\theta \in \mathbb{R}}\sum_{i = 1}^n |Y_i - \theta X_i|.
\]
The objective function is a convex function of $\theta$. Suppose $M_n(\theta) = \sum_{i = 1}^n |Y_i - \theta X_i|$ then we have the following, 
\[
\{\widehat \theta_n \geq \theta_0\} = \{\Dot M_n(\theta_0) \leq 0 \} \quad \mbox{if} \quad \{\widehat \theta_n \leq \theta_0\} = \{\Dot M_n(\theta_0) \geq 0 \}.
\]
Since $\Dot M_n(\theta_0)$ is a sum of mean $0$ independent random variables by CLT we have that $\mathbb{P}(\Dot M_n(\theta_0) \leq 0 ) \rightarrow 1/2$ and $\mathbb{P}(\Dot M_n(\theta_0) \geq 0 ) \rightarrow 1/2$ as $n \rightarrow \infty$. Thus the estimator $\widehat \theta_n$ is asymptotically median-unbiased. 

Suppose $X_i \sim \mbox{Unif}[-1,1]$ and $Y_i = X_i + \epsilon$. We suppose that $\epsilon_i$ and $X_i$ are independent and $F_i(x) = \mathbb{P}(\epsilon_i \leq x) = 0.5(1 +  \mathrm{sgn}(x)|x|^{\beta})$ where $x \in [-1,1]$ for some $\beta > 0$. If $\beta = 1$ then this is the classical setting of error distribution with density bounded away from zero. If $\beta < 1$ then the rate of convergence of the quantile estimator is faster than $n^{1/2}$. If $\beta > 1$ then the rate of convergence is slower than $n^{1/2}$. We generate data for values of $\beta \in \left[0,2\right)$ and compare the performance of HulC and GHulC (at level $\alpha = 0.05$) for higher values of $B$. In particular for the purpose of simulations, we have take $B = 24, 48, 96$ which are essentially multiples of $\lceil \log_2(2/\alpha) \rceil = 6$ (for $\alpha = 0.05$). The performance of each procedure is based on $1000$ Monte Carlo replications for each sample size ($n = 250, 500, 1000, 1500$) and each $\beta$. We observe from \Cref{fig:compare_ghulc} that like HulC, the generalized version of HulC also maintains the coverage at the nominal level of $0.95$ for all sample sizes. Moreover from \Cref{fig:compare_ghulc}, we can also infer that GHulC with higher value of $B$ yields confidence intervals of smaller width. \Cref{fig:ghulc_c} clearly suggests that the dispersion of the width of the confidence interval returned by GHulC also tends to decrease with increasing values of $B$. 
 \begin{figure}[!ht]
    \centering
    \includegraphics[width=\textwidth,keepaspectratio]{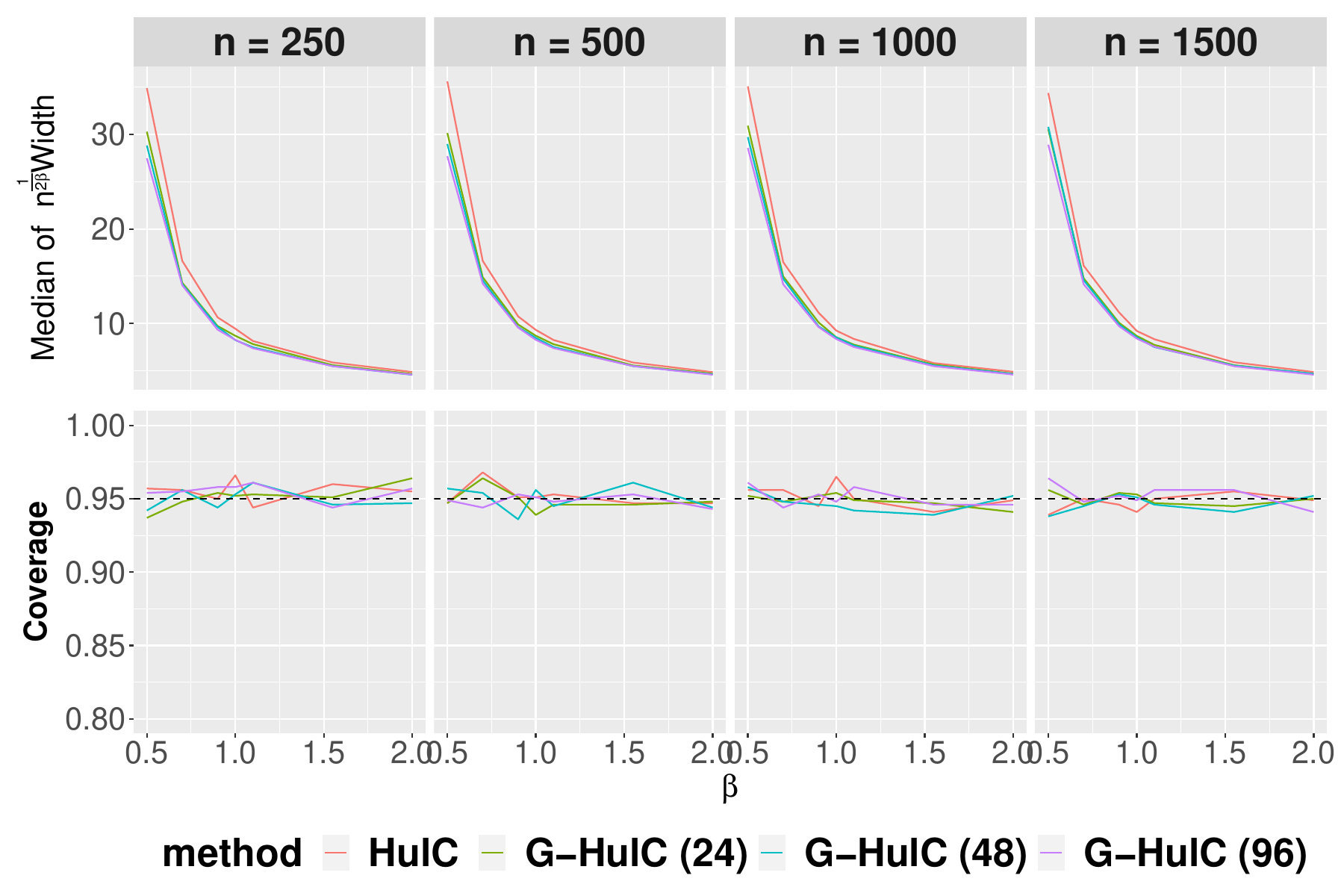}
    \caption{Comparison of the coverage and median of the scaled width ($n^{1/(2\beta)}$Width) of HulC and GHulC (for $B > \log_2(2/\alpha)$) in univariate quantile regression under non-standard conditions. The sample size is mentioned at the top of each plot and the smoothness parameter of the distribution $\beta$ is on the $x$-axis. The tuning parameter $B$ is mentioned in the parenthesis.}
    \label{fig:compare_ghulc}
\end{figure}

\begin{figure}[!ht]
    \centering
    \includegraphics[width=\textwidth,keepaspectratio]{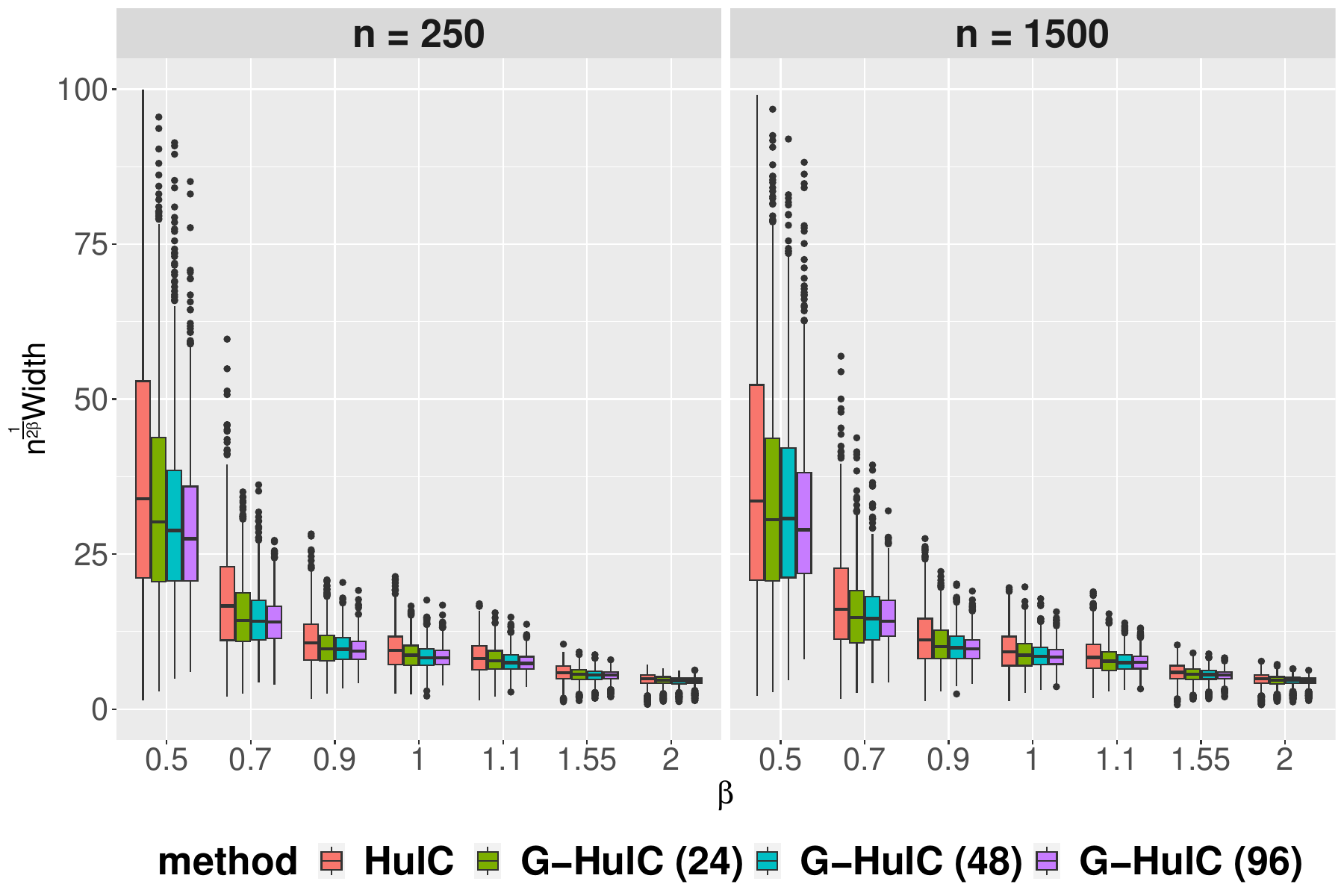}
    \caption{Comparison of box-plots of the scaled width ($n^{1/(2\beta)}$Width) of HulC and GHulC (for $B > \log_2(2/\alpha)$) in univariate quantile regression under non-standard conditions. The sample size is mentioned at the top of each plot and the smoothness parameter of the distribution $\beta$ is on the $x$-axis. The tuning parameter $B$ is mentioned in the parenthesis. The box-plots have been thresholded at $y = 100$ for better visibility, so there might be some outlying observations beyond the threshold.}
    \label{fig:ghulc_c}
\end{figure}
\end{document}